\documentclass[reqno,twoside,11pt]{amsart}

\usepackage{amsmath,amsfonts,calrsfs,fullpage,amssymb,color,verbatim,eucal,yfonts,mathrsfs}

\newtheorem{Theorem}{Theorem}[section]
\newtheorem{Definition}[Theorem]{Definition}
\newtheorem{Proposition}[Theorem]{Proposition}
\newtheorem{Lemma}[Theorem]{Lemma}
\newtheorem{Corollary}[Theorem]{Corollary}
\newtheorem{Remark}[Theorem]{Remark}

\newtheorem{Hypothesis}[Theorem]{Hypothesis}
\makeatletter
\@addtoreset{equation}{section}

\makeatother

\def\oo{\mathaccent23}

\def\R{\mathbb R}
\def\N{\mathbb N}
\def\Z{\mathbb Z}

\def\E{\mathbb E}
\def\P{\mathbb P}

\def\eps{\varepsilon}
\def\ds{\displaystyle}

\newcommand{\esssup}{\operatorname{ess\,sup}}

\newcommand{\Tr}{\operatorname{Tr}}

\newcommand{\one}{1\mkern -4mu\mathrm{l}}

\title[BV functions in Hilbert spaces]{\bf BV functions in Hilbert spaces}

\author[G. Da Prato]{Giuseppe Da Prato}
\address{Scuola Normale Superiore\\
Piazza dei Cavalieri, 7\\ 
56126 Pisa, Italy}
\email{g.daprato@sns.it}

\author[A. Lunardi]{Alessandra Lunardi}
\address{
Dipartimento di Scienze Matematiche, Fisiche e Informatiche \\
Universit\`a di Parma\\
Parco Area delle Scienze, 53/A\\
43124 Parma, Italy}
\email{alessandra.lunardi@unipr.it}

\subjclass[2010]{28C20, 26E15, 49Q15}

\keywords{Infinite dimensional analysis, measures in Hilbert spaces, bounded variation functions.}

\begin{document}

 \begin{abstract}  
We study $BV$ functions in a Hilbert space $X$  endowed with a probability measure $\nu$, assuming that $\nu$ is Fomin differentiable along suitable directions. We establish basic characterizations, and we apply the general theory to relevant examples, including invariant measures of some stochastic PDEs.

 \end{abstract}

\maketitle

 \tableofcontents 
  

\section{Introduction}

In this paper we develop the theory of bounded variation ($BV$) functions  in a separable real Hilbert space $X$ endowed with a Borel probability  measure $\nu$. 
The very definition of $BV$ functions relies on integration by parts formulae; therefore we have to assume that $\nu$ is Fomin differentiable along suitable directions. We recall 
that $\nu$ is Fomin differentiable along $h\in X$ if there exist $\beta \in L^1(X, \nu)$ such that for every $\varphi\in C^1_b(X)$ (the space of the bounded Fr\'echet differentiable functions from $X$ to $\R$ with bounded gradient) we have
$$\int_X \langle \nabla \varphi,h \rangle\,d\nu = \int_X \varphi \,\beta\, d\nu .$$
In this case the function $-\beta$ is called logarithmic derivative of $\nu$ along $h$. 

The subspace of all the elements $h$ such that $\nu$ is Fomin differentiable along $h$ plays an important role in the study of the properties of $\nu$. 
Here we assume that it contains the range of a bounded operator. More precisely, we assume 
that
there exists $R\in \mathcal L(X)$ such that the following hypothesis holds, 
\begin{Hypothesis}
\label{h1'}
  For any  $z\in X$ there exists $v_z\in \cap_{p>1}L^p(X,\nu)$ such that  
\begin{equation}
\label{e1a}
 \int_X\langle R\nabla \varphi,z \rangle\,d\nu = \int_X v_z\varphi  \,d\nu, \quad \varphi\in C^1_b(X). 
\end{equation}
\end{Hypothesis}
This means that $\nu$ is Fomin differentiable along $R^*(X)$, with logarithmic derivatives belonging to all spaces $L^p(X, \nu)$. This hypothesis also let us  introduce Sobolev spaces and generalized gradients, since it allows to show 
that the operator 
$R\nabla : D(R\nabla)= C^1_b(X) \mapsto L^p(X, \nu;X)$ is closable in $L^p(X, \nu)$, for every $p\in [1, +\infty)$. The space $W^{1,p}(X, \nu)$ is defined as the
domain of the closure $M_p$ of such operator. 
Of course, the spaces  $W^{1,p}(X, \nu)$ and the operators $M_p$ depend on $R$. However, since $R$ is fixed once and for all,  we do not emphasize this dependence. 

The assumption that the functions $v_z$ belong to all $L^p $ spaces is made  to have well defined $W^{1,p}$ spaces for every $p$; moreover in the most important and treatable examples we have 
$v_z\in \cap_{p>1}L^p(X,\nu)$ for every $z$. We could assume that $v_z\in  L^{p_0}(X,\nu)$
just for some $p_0>1$ and in this case most of  the statements of the paper should be modified accordingly.  

The Sobolev spaces inherit formula \eqref{e1a} and its variants. In particular, for every $u\in W^{1,p}(X, \nu)$ with $p>1$ we have 
\begin{equation}
\label{parti_introd}
 \int_X u (\langle R\nabla \varphi,z \rangle - v_z\varphi)\,d\nu= - \int_X \langle M_pu, z\rangle \varphi  \,d\nu, \quad z\in X, \;\varphi\in C^1_b(X).
 \end{equation}
This holds also for $p=1$ for certain measures $\nu$ such as Gaussian measures, but in general for   $u\in W^{1,1}(X, \nu)$ the product $uv_z$ may not belong to  $L^1(X,\nu)$, because $v_z\notin L^{\infty}(X, \nu)$ and embedding theorems guaranteeing that $uv_z\in L^1(X,\nu)$ are not available. 

The right hand side of \eqref{parti_introd} can be seen as the negative integral of $\varphi$ with respect to the real measure $m_z:= \langle M_pu, z\rangle  \nu$. The notion of $BV$ function comes from  a generalization of \eqref{parti_introd}: 
 given $u\in L^1(X, \nu)$ such that $uv_z\in L^1(X, \nu)$ for every $z\in X$, we say that $u$ belongs to $BV(X, \nu)$ if there exists a Borel $X$-valued vector measure $m$ such that, setting $m_z(B) := \langle m(B), z\rangle$ for every $z\in X$ and for every Borel set $B\subset X$, we have 
\begin{equation}
\label{defBVintro}
 \int_X u (\langle R\nabla \varphi,z \rangle - v_z\varphi)\,d\nu= - \int_X  \varphi  \,dm_z, \quad z\in X, \;\varphi\in C^1_b(X).
 \end{equation}
Hypothesis \ref{h1'} yields that good vector fields with finite dimensional range, $F(x)= \sum_{i=1}^n f_i(x) z_i$ with $f_i\in C^1_b(X)$ and $z_i\in X$, belong to the domain of the adjoint operator $M^*_p$ for $p>1$, and $M^*_pF(x) = \sum_{i=1}^n (\langle R\nabla f_i,z_i \rangle - v_{z_i}f_i)$. Since $M_p$ plays the role of a (generalized, stretched) gradient, $M^*_p$ plays the role of the negative divergence. 
Denoting by $\widetilde{C}^1(X, X)$ the space of such vector fields, the total variation of the $X$-valued measure $M_pu \,\nu$ is given by 
$$V(u) :=    \sup \bigg\{ \int_X u \,M^*_p F \, d\nu: \; F\in \widetilde{C}^1_b(X ,X), \; \|F(x)\| \leq 1\; \forall x\in X\bigg\}. $$

The right hand side is meaningful for every $u\in L^p(X, \nu)$ with $p>1$, and, more generally, for every $u\in L^1(X, \nu)$ such that $uv_z\in L^1(X, \nu)$ for every $z\in X$. 
A natural basic question is whether, given any $u\in L^p(X, \nu)$ with $p>1$,  $V(u)<+\infty$ is equivalent to $u\in BV(X, \nu)$. 
While  it is not difficult  to see that  if $u\in BV(X, \nu)$  then $V(u)<+\infty$,    if $X$ is infinite dimensional the converse  is a tough question. 
To give a positive answer, it is sufficient to prove the following theorem. 

\begin{Theorem}
\label{Th:BVz_intro}
Fix any $z\in X$. Let $u\in L^1(X, \nu)$ be such that $uv_z\in L^1(X, \nu)$ and 
$$V_z(u) :=    \sup \bigg\{ \int_X u (\langle R\nabla \varphi,z \rangle - v_z\varphi)\, d\nu: \; \varphi \in C^1_b(X), \; \|\varphi\|_{\infty}  \leq 1 \bigg\} <+\infty . $$
Then  there exist a real Borel measure $m_z$ such that \eqref{defBVintro} holds. 
\end{Theorem}

Indeed, if $V(u) <+\infty$ then $V_z(u) <+\infty$ for every $z$, and having the measures $m_z$ at our disposal, a vector measure $m$ such that $m_z(B) = \langle m(B), z\rangle$ for every $z\in X$ and for every Borel set $B$ may be constructed by a  natural procedure. 

However, this slightly simplified problem is hard, too. The assumption $V_z(u)<+\infty$ means that the linear operator
$$T_{u,z}:D(T_{u,z}) :=C^1_b(X)\mapsto \R, \quad T_{u,z}\varphi := \int_X u (\langle R\nabla \varphi,z \rangle - v_z\varphi)\, d\nu, $$
is bounded in the $L^{\infty}$ norm, and therefore it has a linear bounded extension to $BUC(X)$, the closure of $C^1_b(X)$ in the sup norm topology. However, since  $X$ is not locally compact, no version of the Riesz representation Theorem is available and it is not obvious that $T_{u,z}$ may be represented through a measure $m_z$.  

The problem of finding $m_z$ was solved by Fukushima several years ago, by a very complicated procedure that works in a much more  general context of Hausdorff topological vector spaces, and that substantially relies on a change of topology  earlier used in the theory of Dirichlet forms  (\cite{Fu,FOT}). In the  case that $\nu$ is a nondegenerate Gaussian measures in a separable Banach space, for  $u$ in a suitable Orlicz space an independent  much simpler proof that exploited the properties of Gaussian measures was given in \cite{AMMP}.  

Here, for general measures satisfying Hypothesis \ref{h1'},  we take advantage of our Hilbert space setting  to give a much simpler and self-contained proof of Theorem \ref{Th:BVz_intro}. We solve the problem in two steps. In the first step we assume that $u$ vanishes outside some ball $B(0,R)$. In this case we find that $|T_{u,z}\varphi| \leq V_z(u)\|\varphi\|_{L^{\infty}(B(0,R))}$ for every $\varphi\in C^1_b(X)$. 
We endow $X$ with  the weak topology,  that lets $B(0,R)$ be compact. The restriction of $T_{u,z}$ to the (restrictions to $B(0,R)$) of the cylindrical smooth functions$\footnote{namely,  functions of the type  $\varphi(x) =f( \langle x, x_1\rangle, \ldots , \langle x, x_n\rangle)$ for some $f\in C^1_b(\R^n)$ and $x_1, \ldots, x_n\in X$. }$ is therefore a bounded operator in the sup norm, defined in a dense set of the space $C_w(B(0,R))$ of the weakly continuous functions in $B(0,R)$. It has a bounded extension to the whole $C_w(B(0,R))$, which has an integral representation, $\varphi\mapsto \int_{B(0,R)} \varphi\, d\mu$ for some real Borel measure $\mu$ on $B(0,R)$,  by the Riesz Theorem. Since the Borel sets with respect to the weak topology coincide with the  Borel sets with respect to the norm topology, $\mu$ is in fact a Borel measure in $B(0,R)$ with respect to the norm topology. It is extended in a trivial way to all the Borel sets in $X$, and the extension $m_z$ satisfies  \eqref{defBVintro}. 

In the second step we use a $C^1$ partition of $1$ associated to the covering $\{ \oo B(0, k+1)\setminus B(0, k-1):\; k\in \N\} \cup \oo B(0, 1)$ of $X$, 
to write  any $u$ as the series $u= \sum_{k=0}^{\infty} u_k(x)$, where $u_0$ vanishes for $\|x\|\geq 1$ and $u_k$ vanishes for $\|x\|\leq k-1$ and for $\|x\|\geq k+1$, for $k\in \N$. 
If   $V_z(u) <+\infty$, for every $k$ we have $V_z(u_k) <+\infty$, and by the first step there are real Borel measures $m_{z,k}$ such that $T_{u_k,z}\varphi = \int_X \varphi \,dm_{z,k}$ for every $\varphi\in C^1_b(X)$. Patching together the measures $m_{z,k}$ we find that $m_z:= \sum_{k=1}^{\infty} m_{z,k}$ is a well defined real measure such that  \eqref{defBVintro} holds, and that  satisfies $|m_z(X)| = V_z(u)$. 

If the operator $R$ is very good, namely $R=R^*$ is one to one and there exists an orthonormal basis $\{e_k:\; k\in\N\}$ of $X$ consisting  of eigenvectors of $R$, other Sobolev spaces and spaces of $BV$ functions may  be defined. Indeed, in this case every partial derivative $\partial \varphi/\partial e_k : C^1_b(X)\mapsto L^p(X, \nu)$ is closable as an unbounded operator in $L^p(X, \nu)$ for $p\geq 1$, and therefore for any bounded nonnegative sequence $(\alpha_k)$ the operator $\varphi\mapsto \sum_{k=1}^{\infty} \alpha_k (\partial \varphi/\partial e_k )e_k: C^1_b(X)\mapsto L^p(X, \nu)$ is closable as an unbounded operator in $L^p(X, \nu)$ for $p\geq 1$. 
The domain of its closure, endowed with the graph norm,  is still a Sobolev space.   The case $\alpha_k=1$ for each $k$ is  of particular interest, since the above operator is just the gradient from $C^1_b(X)$ to $L^p(X, \nu)$. The domain of its closure $\nabla_p$  is called $W^{1,p}_0(X, \nu)$ and it is continuously embedded in $W^{1,p}(X, \nu)$, while in general $W^{1,p}(X, \nu)$ is not contained in $W^{1,p}_0(X, \nu)$. 
For every $u\in W^{1,p}_0(X, \nu)$ with $p>1$ the integration by parts formula \eqref{parti_introd} gives 
$$ \int_X u (\langle \nabla \varphi,y \rangle - v_{R^{-1}y}\varphi)\,d\nu= - \int_X \langle \nabla_p u, R^{-1}y\rangle \varphi  \,d\nu, \quad y \in X, \;\varphi\in C^1_b(X), $$
and, again, the right hand side can be seen as minus the integral of $\varphi$
 with respect to the real measure $\langle \nabla_p u,  R^{-1}y\rangle  \nu$. The corresponding notion of bounded variation  function is the following: for every  $u\in L^1(X, \nu)$ such that $uv_z\in L^1(X, \nu)$ for each $z\in X$, we say that $u$  belongs to $BV_0(X, \nu)$ if there exists a Borel $X$-valued vector measure $m_0$ such that  
\begin{equation}
\label{defBVintro_0}
 \int_X u (\langle  \nabla \varphi,y \rangle - v_{R^{-1}y}\varphi)\,d\nu= - \int_X  \varphi  \,d\langle m_0, y\rangle , \quad y\in X, \;\varphi\in C^1_b(X).
 \end{equation}
If $u\in BV_0(X, \nu)$, then $u\in BV(X, \nu)$ and the measure $m$ is just $Rm_0$. Using again Theorem \ref{Th:BVz_intro}, we show that 
if $u\in L^p(X, \nu)$ for some $p>1$, then $u\in BV_0(X, \nu)$ iff
$$V_0(u):= \sup   \bigg\{ \int_X u \,\nabla^* _pF \, d\nu: \; F\in \widetilde{C}^1_b(X ,X), \; \|F(x)\| \leq 1\; \forall x\in X\bigg\},  $$
 where $\nabla^*_p$ is the adjoint operator of $\nabla_p$, and for $F\in \widetilde{C}^1_b(X ,X)$, $F(x) = \sum_{i=1}^n f_iz_i$ we have $ \nabla^* _pF=
 \sum_{i=1}^n (\langle \nabla f_i, z_i\rangle - v_{R^{-1}z_i}f_i)$. 

If $\nu$ is a centered nondegenerate Gaussian measure, it is natural to choose $R= Q^{1/2}$, where $Q$ is the covariance of $\nu$, so that the range of $R$ is the Cameron-Martin space consisting of all $h\in X$ such that $\nu$ is Fomin differentiable along $h$. Since $R$ is compact and self-adjoint, there exists an orthonormal basis of $X$ consisting of eigenvectors of $R$, so that both spaces $BV(X, \nu)$ and $BV_0(X, \nu)$ are meaningful. 
$BV$ functions for Gaussian measures in separable Banach spaces were introduced in \cite{F} and subsequently studied   in \cite{FH,AMMP}. Although the notations of such papers are different from ours, our notion of $BV$ functions coincides with theirs.   $BV_0$ functions were considered in the Hilbert space setting; our notion of $BV_0$ functions coincides  with the one of \cite{AmDaPa10}, which was introduced in the last section of \cite{AMMP}. The paper  \cite{RZZ1} deals with a class of $BV$ functions for Gaussian measure in Hilbert spaces, depending on  a Hilbert space $H_1\subset X$; our notion of $BV$ and $BV_0$ functions coincide  with the ones of  \cite{RZZ1} with the choices $H_1=Q^{1/2}(X)$ and $H_1=X$, respectively.

Still in the case of Gaussian measures, an elegant characterization of $BV$ functions  through Ornstein-Uhlenbeck semigroups is available. Precisely,  in \cite{FH} it was shown that if $u$ belongs to the Orlicz space $Y= L(\log L)^{1/2}(X, \nu)$, then $u\in BV(X, \nu) $ if and only if 
$$\liminf_{t\to 0} \int_X \|M_1T(t)u\|\,d\nu <+\infty$$
where $T(t)$ is the classical Ornstein-Uhlenbeck semigroup, see Section 5. An analogous characterization for $BV_0$ functions was obtained in \cite{AmDaPa10} through another Ornstein-Uhlenbeck semigroup. 

If there exists a smoothing semigroup of operators $T(t)$ in $L^p(X, \nu)$ for some $p\geq 1$ having good commutation properties with partial derivatives, 
we obtain similar results for our general measures $\nu$ (Section 3.2). 

As in the finite dimensional case, if a characteristic function $\one_B$ belongs to $BV(X, \nu)$, we say that $B$ has finite perimeter, and $|m|(X)$ is called perimeter of $B$. 
Establishing whether a given Borel set $B$ has finite perimeter is not an easy task. We prove that every halfspace $H_{a,r}:=  \{ x\in X:\; \langle x, a\rangle <r\}$ with $a\in X$ and $r\in \R$ has finite perimeter, and we give a formula to compute its perimeter. Relying on the construction of surface measures of \cite{TAMS}, we show that if $g$ is a smooth enough function satisfying suitable nondegeneration assumptions, 
the sublevel sets $ \{ x\in X:\; g(x) <r\}$ have finite perimeter for every $r\in \R$. If $g$ belongs just to $W^{1,1}(X, \nu)$ we can only prove that for almost all $r\in \R$ the  set $ \{ x\in X:\; g(x) <r\}$ has finite perimeter. 

Our general theory may be applied to Gaussian measures in Hilbert spaces, in which case we find again the results of the above mentioned papers \cite{F,FH,AMMP,AmDaPa10}. 
We find new results for weighted Gaussian measures $\nu= e^{-2U}\gamma$, where $U$ is a convex $C^1$ function with Lipschitz continuous gradient and $\gamma $ is a centered nondegenerate Gaussian measure. Again it is convenient to choose $R=Q^{1/2}$, where $Q$ is the covariance of $\gamma$. We prove that if $u\in L^2(X, \gamma)$ then 
\begin{equation}
\label{equv_intro}
u\in BV(X, \nu) \Longleftrightarrow \liminf_{t\to 0} \int_X \|M_2T(t)u\|\,d\nu <+\infty, 
\end{equation}
where $T(t)$ is the semigroup generated by the self-adjoint operator $K$ associated to the quadratic form
\begin{equation}
\label{K_intro}
(u,v)\mapsto \int_X \langle  \nabla_2u, \nabla_2v\rangle \,d\nu, \quad u, \, v\in W^{1,2}_0(X, \nu). 
\end{equation}
%
%
$T(t)$ is obviously smoothing, since it is an analytic semigroup that maps $L^2(X, \nu)$ into the domain of $K$ which is contained in $ W^{1,2}_0(X, \nu)$. 
Here we prove a commutation formula of independent interest,  
\begin{equation}
\label{comm_intro}
\frac{\partial T(t)f }{\partial e_k} (x) - e^{-t/2\lambda_k}T(t)\left( \frac{\partial f }{\partial e_k} \right)(x) = - \int_0^t e^{-(t-s)/2\lambda_k}(T(t-s) \langle D^2U(\cdot )e_k,   \nabla T(s)f (\cdot ))(x)\rangle  \,ds , 
\end{equation}
that holds for  $t>0$, $f \in C^1_b(X)$,  and any orthonormal basis $\{e_k:\; k\in \N\}$ of $X$ consisting of eigenvectors of $Q$,  with $Qe_k =\lambda_k e_k$. 
Such a formula is a key tool for the above characterization. As it frequently happens in infinite dimensional analysis, its formal derivation  is easy but its proof is complicated, and it is deferred to the Appendix. 

We also give a specific example, in which the assumption that $U$ has Lipschitz continuous gradient is not satisfied. Namely, we consider the case where  $\nu$ is the invariant measure of a stochastic reaction-diffusion equation, 
\begin{equation}
\label{reazdiff_intro}
dX(t)=[AX(t)- f(X(t))]dt+ dW(t),
\end{equation}
where $A$ is  the realization of the second order derivative with Dirichlet boundary condition in $X:=L^2(0, 1)$, and the nonlinearity $f:\R\mapsto \R$ is an increasing polynomial with degree $d>1$, and $W(t)$  is  any cylindrical $X$-valued Wiener process. We have  $\nu= e^{-2U}\gamma/\int_Xe^{-2U}d\gamma$, where $\gamma $ is the centered Gaussian measure with covariance $Q= (-2A)^{-1}$, and 
 $U(x)=   \int_0^{1} \Phi(x(\xi))d\xi$, $\Phi$ being any primitive of $f$.   This function is defined  $\gamma$-a.e., namely   in $L^{d+1}(0,1)$. We know from \cite{DPL} that $\gamma( L^q(0,1)) = 1$ for every $q\geq 2$, and   $U\in   W^{2,p}(X, \gamma)\cap   W^{1,p}_0(X, \gamma)$ for every $p\geq 1$. 
We show that the semigroup $T(t)$ defined as before, through the quadratic form \eqref{K_intro}, 
 is an extension to $L^2(X, \nu)$ of the transition semigroup of equation  \eqref{reazdiff_intro}, 
and we get estimates on  the commutators between $T(t)$ and partial derivatives approximating $U$ by its Yosida approximations $U_{\alpha}$,  and using \eqref{comm_intro} for the corresponding semigroups $T_{\alpha}(t)$. As a result,  we get the same characterization as in the general smooth case, namely  we prove that \eqref{equv_intro} holds for  $u\in L^2(X, \nu)$. 

Our last example concerns a class of product measures $\nu$  in $X$  that are not Gaussian nor weighted Gaussian measures. As a consequence of the already mentioned result of  \cite{TAMS}, for such measures the characteristic functions of balls centered at the origin belong to $BV(X, \nu)$, so that all such balls have finite perimeter. Here we show that in the particular case $X= L^2(0,1)$, for every $q>2$ the $L^q$ ball $\{ x\in L^q(0,1):\; \|x\|_{L^q(0,1)}< r\}$ has finite perimeter for a e. $r>0$.

\section{Notation and preliminaries}

Throughout the paper we assume that Hypothesis \ref{h1'} holds. It yields that 
 for every $p\geq 1$ the linear mapping $X\mapsto L^p(X, \nu)$, $z\mapsto v_z$, is closed and therefore continuous. Consequently, there exist $C_p>0$
such that 
\begin{equation}
\label{vz}
\|v_z\|_{L^p(X,\nu)}\leq C_p\|z\|, \quad z\in X. 
\end{equation}

In this section we collect notation and  results (mainly taken from  \cite{TAMS}) that will be used later. 

\subsection{General notation.}

We consider  a separable Hilbert space  $X$ with norm $\|\cdot\|$ and scalar product $\langle\cdot, \cdot\rangle$,  endowed with a Borel   probability measure $\nu$. For every $r>0$ and $x_0\in X$ we denote by $B(x_0, r)$ the closed ball centered at $0$ with radius $r$. 

Fixed any orthonormal basis $\{e_k: \; k\in \N\}$ of $X$, we denote by  $P_n$  the orthogonal projection 
\begin{equation}
\label{Pn}
P_nx := \sum_{k=1}^n \langle x, e_k\rangle e_k, \quad x\in X. 
\end{equation}

For $p>1$ we set as usual $p' = p/(p-1)$.

\subsection{Spaces of continuous and differentiable functions. }

For Fr\'echet differentiable functions $\varphi: X\mapsto \R$ we denote by  $\nabla \varphi(x)$ the  gradient of $\varphi $  at $x$, and by $\partial_z\varphi(x) = \langle \nabla \varphi(x), z\rangle$ its derivative along $z$, for every $z\in X$. 

By $C_b(X)$   we mean the space of all real continuous  and bounded mappings   $\varphi:  X\to \R$, endowed with the sup norm $\|\cdot \|_{\infty}$. Moreover, $C^1_b(X)$  is  the subspace of $C_b(X)$  of all continuously Fr\'echet  differentiable functions, with bounded   gradient. 

The space of the cylindrical functions $\mathcal{FC}^1_b(X)$ is the set of all functions $f:X\mapsto \R$ of the type 
$f(x) = \varphi (\langle x, z_1\rangle, \ldots \langle x, z_n\rangle)$, where $\varphi\in C^1_b(\R^n)$ and $z_k\in X$ for $k=1, \ldots, n$. 

We shall also consider special classes of vector fields, consisting of vector fields with values in a finite dimensional subspace of $X$, and marked by a tilde $\,\widetilde{}\;$. For every subspace $Y$ of $X$ we set
\begin{equation}
\label{C1btilde}
\widetilde{C}^{1 }_b(X, Y) := \{ F = \sum_{i=1}^n f_i\,z_i, \; n\in \N, \; f_i\in C^{1}_b(X), \; z_i\in Y\}, 
\end{equation}
\begin{equation}
\label{FC1btilde}
\widetilde{\mathcal{FC}}^{1 }_b(X, Y) := \{ F = \sum_{i=1}^n f_i\,z_i, \; n\in \N, \; f_i\in \mathcal{FC}^{1}_b(X), \; z_i\in Y\}. 
\end{equation}

 We shall use the following approximation lemma.

\begin{Lemma}
\label{Le:approx}
Let $\{e_k: \; k\in \N\}$ be any any orthonormal basis  of $X$, and let $P_n$ be defined by \eqref{Pn}. 
\begin{itemize}
\item[(i)] For every $\varphi\in C^1_b(X)$ the sequence $(\varphi_n):= (\varphi\circ P_n)$ converges pointwise to $\varphi$, $\partial \varphi_n/\partial e_k$ converges pointwise to $\partial \varphi/\partial e_k$ for each $k\in \N$, and $\|\varphi_n\|_{\infty}\leq \|\varphi\|_{\infty}$, $\|\partial \varphi_n/\partial e_k\|_{\infty} \leq \|\partial \varphi/\partial e_k\|_{\infty}$, 
$\sup_{x\in X} \|\nabla \varphi_n(x)\| \leq \sup_{x\in X} \|\nabla \varphi(x)\| $. 
\item[(ii)] For every $\varphi\in C_b(X)$ there exists a two-index sequence $(\varphi_{k,n})$ of   $\mathcal{FC}^1_b(X)$ functions such that 
$$\lim_{k\to \infty} \lim_{n\to \infty} \varphi_{k,n}(x) = \varphi(x); \quad |\varphi_{k,n}(x)|  \leq \|\varphi\|_{\infty}, \qquad  x\in X. $$
\item[(iii)] Let $r>0$. For every $\varphi\in C_b(X)$ such that $\varphi(x) =0$ for $\|x\|\geq r$, there exists a two-index sequence $(\widetilde{\varphi}_{k,n})$ of   $C^1_b(X)$ functions such that $\widetilde{\varphi}_{k,n}(x) = 0$ for $\|x\|\geq r$ and
$$\lim_{k\to \infty} \lim_{n\to \infty} \widetilde{\varphi}_{k,n}(x) = \varphi(x), \:\forall x\in X; \quad \|\widetilde{\varphi}_{k,n}\|_{\infty} \leq \|\varphi\|_{\infty}. $$
\end{itemize}
\end{Lemma}
\begin{proof}
Statement (i) is easily proved, noticing that $\partial \varphi_n/\partial e_k (x)= \partial \varphi/\partial e_k(P_nx)$ for $k\leq n$, and $\partial \varphi_n/\partial e_k (x)$ $=0$ for $k>n$. 

To prove Statement (ii), first of all we approach $\varphi$ by $\varphi \circ P_n$. In its turn, $\varphi \circ P_n$ is approached by
$$\varphi_{k,n}(x) := \int_{\R^n} \varphi \bigg( P_nx + \frac{1}{k}\sum_{j=1}^{n}\xi_k e_k\bigg) \rho_n(\xi) d\xi, $$
where $\rho_n $ is any smooth function supported in the unit ball of $\R^n$, such that  $\int_{\R^n} \rho_n(\xi) d\xi =1$. Statement (ii) follows. 

To prove Statement (iii) we fix a sequence of functions $\theta_k\in C^1(\R)$ such that 
$$\theta_k(\xi) = 1 \;\mbox{for}\; \xi\leq \bigg(r-\frac{1}{k}\bigg)^2, \quad \theta_k(\xi) = 0 \;\mbox{for}\; \xi \geq r^2, \quad \|\theta_k\|_{\infty} =1, $$
and we modify the sequence $(\varphi_{k,n})$ of Statement 1 setting 
$\widetilde{\varphi}_{k,n}(x) := \varphi_{k,n}(x) \eta_k(x)$, where $ \eta_k(x)= \theta_k(\|x\|^2)$. 
\end{proof}

\begin{Remark}
\label{Rem:normaL^1}
As a consequence of Lemma \ref{Le:approx}(ii), for every  $f\in L^1(X, \nu)$ we have
$$\|f\|_{L^1(X, \nu)} = \sup \bigg\{ \int_X f\,\varphi\,d\nu:\; \varphi\in C^1_b(X ), \,\|\varphi\|_{\infty} \leq 1\bigg\}. $$
Moreover, 
$$\|F\|_{L^1(X, \nu;X)} = \sup \bigg\{ \int_X \langle F, \Phi\rangle\,d\nu:\;\Phi \in \widetilde{C}^1_b(X ,X), \,\|\Phi\|_{\infty} \leq 1\bigg\}. $$
These equalities will be used later. 
\end{Remark}

\subsection{Sobolev spaces}

\begin{Proposition}
For every $p\in [1, +\infty)$ and for every 
 $z\in X$, the operator $C^1_b(X) \mapsto L^p(X, \nu)$, $\varphi\mapsto \langle R\nabla \varphi, z\rangle$ is closable in $L^p(X, \nu)$. Therefore,  the operator
$$ R\nabla : D(R\nabla):= C^1_b(X) \mapsto L^p(X, \nu;X)
$$
is closable   in $L^p(X, \nu)$. 
\end{Proposition}
\begin{proof} Let us consider the case  $p=1$. 

Let  $f_n\in C^1_b(X)$ be such that  $f_n\to 0$ in $L^1(X, \nu)$, and  $R\nabla f_n\to G$ in $L^1(X, \nu; X)$. We have to show that $G=0$. 
Without loss of generality, we may assume that 
$f_n(x)\to 0$ for a.e. $x\in X$. 

Fix a function $\theta \in C^1_b(\R)$ such that 
$$\theta(0) =0, \quad \theta'(0)=1. $$
For every $n\in \N$,  $\theta \circ f_n \to \theta(0) =0$ a.e., and  $|\theta \circ f_n(x)|\leq \|\theta\|_{\infty}$ for each $x\in X$, so that  $\theta \circ f_n \to 0$ in $L^1(X, \nu)$. 
Moreover,  $\nabla (\theta \circ f_n)  = ( \theta' \circ f_n)\nabla f_n$, so that  $R\nabla (\theta \circ f_n)  =  ( \theta' \circ f_n)R\nabla f_n$. For every $n\in \N$ we have
$$\| R\nabla (\theta \circ f_n) -G\|_{L^1(X, \nu; X)} \leq \|( \theta' \circ f_n)( R\nabla f_n-G)\|_{L^1(X, \nu; X)} +    \| (( \theta' \circ f_n)-1)G\|_{L^1(X, \nu; X)}. $$
The first addendum in the right hand side does not exceed $\|\theta'\|_{\infty} \|R\nabla f_n-G\|_{L^1(X, \nu; X)} $, so that it vanishes as $n\to \infty$. The second addendum vanishes too by the Dominated Convergence Theorem, since $( \theta' \circ f_n)-1$ converges to $0$ a.e, and  $\| ( \theta' \circ f_n)(x)-1)G(x)\| \leq (\|\theta'\|_{\infty}  +1)\|G(x)\|$ for every $n$. 
Therefore, $(R\nabla (\theta \circ f_n))$ converges to  $G$ in $L^1(X, \nu; X)$. 

Now we fix any orthonormal basis $\{e_k:\; k\in \N\}$ of $X$. For every $k\in\N$ we have
$$ \int_X \langle R\nabla  (\theta \circ f_n), e_k\rangle  \,\psi\, d\nu = \int_X (\theta \circ f_n)\left( - \langle R\nabla \psi, e_i\rangle  + v_{e_k} \psi\right)  .$$
Letting $n\to\infty$, since $(R\nabla (\theta \circ f_n))$ converges to  $G$ in $L^1(X, \nu; X)$, the left hand side converges to $\int_X\langle G, e_i\rangle \psi\, d\nu$. The right hand side converges to $0$, since $\theta \circ f_n \to  =0$ a.e. and $| (\theta \circ f_n)\ ( - \langle R\nabla \psi, e_i\rangle  + v_{e_k} \psi )| \leq $
 $\|\theta\|_{\infty}(\|R\|_{\mathcal L(X)}  \|\nabla  \psi \|_{\infty} + |v_{e_k}(x)|) \|\psi\|_{\infty}$. Consequently, 
$$\int_X\langle G, e_k\rangle \,\psi \, d\nu =0, \quad k\in \N $$
so that $\langle G, e_k\rangle =0$ a.e. for every $k\in \N$. 

For $p>1$ the proof is a simplification of this one, since $ v_{e_k}$ belongs to $L^{p'}(X, \nu)$ for every $k$ and the argument works using the functions $f_n$ instead of $\theta \circ f_n$. 
 \end{proof}

%
%
\begin{Definition}
\label{defSobolev}
For $p\in [1, +\infty)$ we denote by $M_p$ the closure of $R\nabla$   in $L^p(X,\nu)$, and by $W^{1,p}(X, \nu)$ its domain.  
$W^{1,p}(X, \nu)$ is a Banach space with the graph norm, 
\begin{equation}
\label{graphnorm}
\|f\|_{W^{1,p}(X, \nu)} = \bigg( \int_X|f(x)|^p\nu(dx)\bigg)^{1/p} +  \bigg( \int_X\|M_pf(x)\|^p\nu(dx)\bigg)^{1/p}. 
\end{equation}
For every $f\in W^{1,p}(X, \nu)$ and $z\in X$ we set $ \langle M_pf, z\rangle :=\partial f/\partial R^*z$. 
\end{Definition}
 
So, any $\varphi\in W^{1,p}(X, \nu)$ is the $L^p(X, \nu)$-limit of a sequence $(\varphi_n) \subset C^1_b(X)$, such that $(R\nabla \varphi_n)$ is convergent  sequence in $L^p(X, \nu; X)$. The approximating sequence may be taken in ${\mathcal FC}^1_b(X)$,  since, in its turn,    each $\varphi\in C^1_b(X)$ may be approximated  in the above norm 
by a sequence of elements of ${\mathcal FC}^1_b(X)$,  by Lemma \ref{Le:approx}(i).

We refer to   \cite[Sect. 2]{TAMS} for general properties of the $W^{1,p}$ spaces, of the operators $M_p$ and of their adjoint operators. 
In particular, we recall that the dual  spaces  $(L^p(X, \nu))'$,  $(L^{p}(X,\nu;X))'$ are canonically identified with $L^{p'}(X,\nu)$, 
 $L^{p'}(X,\nu;X)$ respectively, with $p'=p/(p-1)$ for $p>1$, $1'=+\infty$ (e.g., \cite{DU}). We denote  by $M_p^*: D(M_p^*)\subset 
L^{p'}(X,\nu;X)\to L^{p'}(X,\nu)$  the adjoint of $M_p$. 
So, we have
\begin{equation}
\label{e1m}
\int_X \langle M_p\varphi,F  \rangle\,d\nu=\int_X \varphi\,M_p^*(F)\,d\nu,\quad \varphi\in D(M_p),\;F\in D(M_p^*).
\end{equation}
We notice that 
Hypothesis \ref{h1'} implies  that  for every 
 $z\in X$   the constant vector field $F_z(x) := z$ belongs 
to $D(M_p^*)$ for every $p>1$, and the function $X\mapsto D(M_p^*)$, $z\mapsto F_z$, is continuous. 
  Indeed, \eqref{e1a} and the definition of $M_p$ yield
 \begin{equation}
\label{e2m}
   \int_X  \langle M_p\varphi,z\rangle\, d\nu= \int_X\varphi\, v_z\,d\nu , \quad \varphi \in D(M_p), 
\end{equation}
%
%
%
so that $F_z\in D(M^*_p)$ and $M^*_p F_z = v_z$ for every $z\in X$. However, the functions $v_z$ are not essentially bounded and $F_z$ does not belong to 
$D(M^*1)$, in general.  

 If $u\in W^{1,p}(X, \nu)$ and $\varphi\in C^1_b(X)$, the product $u\varphi $ belongs to $W^{1,p}(X, \nu)$ and we have
 $M_p(u\varphi) = M_pu\,\varphi + u \,R\nabla \varphi$. So, if $p>1$,  \eqref{e2m} yields 
 \begin{equation}
\label{parti}
   \int_X u  \langle M_p\varphi,z\rangle  d\nu =   \int_X  (v_z u\,\varphi - \langle M_pu,z\rangle \varphi) d\nu, \quad z\in X, 
\end{equation}
which can be considered as an integration by parts formula. If $p=1$ the right hand side of  \eqref{parti}
is not meaningful, in general. For certain measures, such as Gaussian or suitably weighted Gaussian measures, 
$W^{1,1}(X, \nu)$ is continuously embedded in an Orlicz space $Y\subset L^{1}(X, \nu)$ such that all the functionals $u\mapsto \int_X v_zu\,d\nu$ are well defined and belong to $Y'$, and  \eqref{parti}
holds  for every $u\in W^{1,1}(X, \nu)$. 

Moreover, for $p>q\geq 1$ we have 
$$W^{1,p}(X, \nu)\subset W^{1,q}(X, \nu),\; M_p f= M_qf, \quad f\in W^{1,p}(X, \nu), $$
$$D(M_q^*) \subset D(M_p^*), \; M_q^* F = M_p^*F, \quad F\in D(M_q^*). $$
Therefore, to simplify notation, for functions $f\in \cap_{p\geq 1}W^{1,p}(X, \nu)$ we set $Mf:=M_pf$ for every $p\geq 1$, and for vector fields 
$F\in  \cap_{p>1}D(M_p^*)$ we set $M^*F:= M_p^*F$ for every $p>1$. 

Accordingly to the notation of subsection 1.1, for any subspace $Y\subset X$ we set
\begin{equation}
\label{W1ptilde}
\widetilde{W}^{1,p}(X, \nu;Y) := \{ F = \sum_{i=1}^n f_i\,z_i: \; n\in \N, \; f_i\in W^{1,p}(X, \nu), \; z_i\in Y\}  . 
\end{equation}
For every $p>1$, the vector fields in $\widetilde{W}^{1,p'}(X, \nu;X)$ belong to $D(M^*_p)$, and for $F= \sum_{i=1}^n f_iz_i $ we have 
 \begin{equation}
\label{divW}
M^*_pF  =  \sum_{i=1}^n (- \langle M_{p'}f_i, z_i\rangle + v_{z_i} f_i).  =  \sum_{i=1}^n (- \langle R\nabla f_i, z_i\rangle + v_{z_i} f_i) . 
\end{equation}
Therefore,  the vector fields in $ \widetilde{C}^1_b(X, X)$  belong to $D(M^*_p)$ for every $p>1$, and if 
$F= \sum_{i=1}^n f_iz_i \in  \widetilde{C}^1_b(X, X)$   we have
 \begin{equation}
\label{div}
M^*F  =  \sum_{i=1}^n (- \langle M_{p'}f_i, z_i\rangle + v_{z_i} f_i) =  \sum_{i=1}^n (- \langle R\nabla f_i, z_i\rangle + v_{z_i} f_i) . 
\end{equation}
Formula \eqref{div} yields 
 \begin{equation}
\label{Th:div}
\int_X \langle M_p \varphi, F\rangle d\nu =   \int_X \varphi \, M_p^*F \,d\nu =   \int_X  \sum_{i=1}^n \varphi (- \langle R\nabla f_i, z_i\rangle + v_{z_i} f_i)  \,d\nu, \quad \varphi\in  W^{1,p}(X, \nu). 
\end{equation}

The following properties of Sobolev spaces will be used later in the paper.

\begin{Lemma}
\label{Le:Sobolev}
Let $1<p <\infty$. 
\begin{itemize}
\item[(i)] For every $\varphi\in W^{1,p}(X, \nu) \cap L^{\infty}(X, \nu)$ and $\delta >0$ there exists a sequence of $C^1_b(X)$ functions $(\varphi_n)$ such that 
$\|\varphi_n\|_{\infty} \leq (1+\delta)\|\varphi\|_{\infty}$ and $\varphi_n\to \varphi$ in $W^{1,p}(X, \nu)$. 
\item[(ii)] If $\varphi$, $\psi\in W^{1,p}(X, \nu) \cap L^{\infty}(X, \nu)$, the product $\varphi\psi$ belongs to $W^{1,p}(X, \nu) $, and $M_p(\varphi\psi) = \varphi M_p \psi +  \psi M_p\varphi$. 
\end{itemize}
\end{Lemma}
\begin{proof}
To prove Statement (i) we consider a $C^1$ function $\theta:\R\mapsto \R$ such that $\|\theta\|_{\infty} = (1+\delta)\|\varphi\|_{\infty}$ and 
$$  \theta(\xi ) = \left\{ \begin{array}{ll}
(1+\delta)\|\varphi\|_{\infty}, & \xi \geq (1+2\delta)\|\varphi\|_{\infty}, 
\\
\xi, & |\xi| \leq (1+\delta /2)\|\varphi\|_{\infty}, 
\\
(-1-\delta)\|\varphi\|_{\infty}, & \xi\leq (-1-2\delta)\|\varphi\|_{\infty}. 
\end{array} \right. $$
Let $(f_n)$ be a sequence of $C^1_b(X)$ functions that converges to $\varphi$ in $W^{1,p}(X, \nu)$ and almost everywhere, and set 
$$\varphi_n:= \theta \circ f_n, \quad n\in \N. $$
Then $\varphi_n \in C^1_b(X)$ and  $\|  \varphi_n\|_{\infty} \leq (1+\delta)\|\varphi\|_{\infty}$. Moreover, since  $f_n(x) \to \varphi(x) $ as $n\to \infty$ for a.e. $x$, we have 
$|f_n(x)| \leq (1+\delta/2 )\|\varphi\|_{\infty}$ for $n$ large enough, so that $ \varphi_n(x) = f_n(x)\to \varphi(x)$ as $n\to \infty$ for a.e. $x$. Since $|  \varphi_n(x) -\varphi(x)|\leq (2+\delta ) \|\varphi\|_{\infty}$ for a.e. $x\in X$, by the Dominated Convergence Theorem we obtain 
$\lim_{n\to \infty}  \varphi_n = \varphi$ in $L^{p}(X, \nu)$. 
Moreover, $\nabla \varphi_n = (\theta'\circ f_n)\nabla f_n$, so that 
$$\|R\nabla  \varphi_n - R\nabla \varphi\|_{L^{p}(X, \nu; X)} \leq \| ((\theta'\circ f_n)-1)R\nabla \varphi\|_{L^{p}(X, \nu; X)} + \| (\theta'\circ f_n)(R\nabla \varphi - R\nabla f_n)\|_{L^{p}(X, \nu; X)}. $$
As before, for a.e. $x\in X$ we have $\theta'(f_n(x)) = 1$ for large enough $n$, so that $(\theta'\circ f_n)(x)-1\to 0$ as $n\to \infty$, and by the Dominated Convergence Theorem 
$\lim_{n\to \infty}\| ((\theta'\circ f_n)-1)R\nabla \varphi\|_{L^{p'}(X, \nu; X)} =0$. Since also $\lim_{n\to \infty} \| (\theta'\circ f_n)(R\nabla \varphi - R\nabla f_n)\|_{L^{p}(X, \nu; X)}=0$, 
Statement (i) follows. 

\vspace{3mm}

To prove Statement (ii) we use (i), approaching $f$ and $g$ in $W^{1,p}(X, \nu)$ and $\nu$-a.e. by sequences of $C^1_b(X)$ functions $(f_n)$, $(g_n)$ such that $\|f_n\|_{\infty} \leq 2\|f\|_{\infty}$, $\|g_n\|_{\infty} \leq 2\|g\|_{\infty}$. The product $f_ng_n$ converges to $fg$ in $L^p(X, \nu)$ by the Dominated Convergence Theorem; moreover we have $R\nabla (f_ng_n) = f_n R\nabla g_n + g_n R\nabla f_n$, and
$$\|f_n R\nabla g_n - fM_pg\|_{L^p(X, \nu;X)} \leq \|(f_n-f) M_pg\|_{L^p(X, \nu;X)} + \|f( R\nabla g_n -  M_pg)\|_{L^p(X, \nu;X)} $$
where the first addendum in the right hand side vanishes as $n\to \infty$ again by the Dominated Convergence Theorem, and the second addendum vanishes too. 
Similarly, $\|g_n R\nabla f_n - gM_pf\|_{L^p(X, \nu;X)}\to 0$ as $n\to \infty$. Therefore, $R\nabla (f_ng_n)\to fM_pg+gM_pf$ in  $L^p(X, \nu;X)$, and (ii) is proved. 
 \end{proof}

\begin{Remark}
\label{Sobolevgenerali}
{\em If    there exists an orthonormal basis $\{e_k:\; k\in \N\}$ of $X$ contained in  $R^*(X)$,  it is possible to define different classes of Sobolev spaces. Indeed, in this case 
each partial derivative $\partial /\partial e_k:C^1_b(X)\mapsto L^p(X, \nu)$ is closable in $L^p(X, \nu)$ (since   $\partial \varphi/\partial e_k = \langle R\nabla \varphi, z\rangle$ for any $z$ such that $R^*z=e_k$), and therefore for any bounded nonnegative sequence $(\alpha_k)$ the operator 
$C^1_b(X)\mapsto L^p(X, \nu;X)$, $\varphi\mapsto \sum_{k^1}^{\infty} (\alpha_k \partial \varphi/\partial e_k) e_k$ is closable in $L^p(X, \nu)$. 
The domain of the closure $M_{(\alpha_k)}$, endowed with the graph norm,  is still a Sobolev space. Given two sequences $(\alpha_k)$, $(\beta_k)$, the respective Sobolev spaces have equivalent norms iff there exists $C\geq 1$ such that $C^{-1} \beta_k \leq \alpha_k \leq C\alpha_k$, for every $k\in \N$. 
The case $\alpha_k=1$ for each $k$ is  of particular interest, since $M_{(\alpha_k)}$ is just the closure of the gradient  in $L^p(X, \nu)$.  We shall consider it in Subsection \ref{subs:other}.}
\end{Remark}

In the sequel we shall consider also Sobolev spaces of order $2$. To define them, we use the following easy lemma. We recall that ${\mathcal L}_2 (X)$ is the subspace of 
$\mathcal L (X)$ consisting of the Hilbert-Schmidt operators, endowed with the norm $\|L\|_{{\mathcal L}_2 (X)} := \left( \sum_{h, k=1}^{\infty} \langle Le_h, e_k\rangle^2\right)^{1/2}$, where $\{e_k: \; k\in \N\}$ is any orthonormal basis of $X$.

\begin{Lemma}
\label{Le:chiusura}
For every $p>1$ the operator 
$${\mathcal D}^2_R:D({\mathcal D}^2_R)= {\mathcal FC}^2_b(X)\subset L^p(X, \nu) \mapsto (L^p(X, \nu;X) \times L^p(X, \nu; \mathcal L _2(X))), $$
$${\mathcal D}^2_R\varphi(x) := (R\nabla \varphi, RD^2\varphi(x)R^*)$$ 
is closable in $L^p(X, \nu)$. 
\end{Lemma}
\begin{proof}
Let $\varphi_n\in {\mathcal FC}^2_b(X)$ be such that $\lim_{n\to\infty} \varphi_n=0$ in $L^p(X, \nu)$, $\lim_{n\to \infty} {\mathcal D}^2_R\varphi_n = (F, L)$ in 
$ (L^p(X, \nu;X) \times L^p(X, \nu; \mathcal L _2(X))$. We already know that $F=0$, since $R\nabla $ is closable in $L^p(X, \nu)$ even with the bigger domain $C^1_b(X)$. 
We have to show that $L =0$ (as an element of $L^p(X, \nu; \mathcal L _2(X))$), which is equivalent to $\langle L(\cdot)e_h, e_k\rangle =0$, $\nu$-a.e.,   for every $h$, $k\in \N$. Since $C^2_b(X)$ is dense in $L^{p'}(X, \nu)$ it is sufficient to show that for every $h, \, k\in \N$ we have 
 \begin{equation}
\label{serve}
\int_X \langle L(\cdot)e_h, e_k\rangle \psi\,d\nu =0, \quad  \psi \in C^2_b(X). 
\end{equation}
For every $h, k\in \N$ and $\psi \in C^2_b(X)$ we have 
$$\int_X \langle L(x)e_h, e_k\rangle \psi\,\nu(dx) = \lim_{n\to\infty} \int_X \langle RD^2\varphi_n(x)R^*e_h, e_k\rangle \psi\,\nu(dx). $$
If $R^*e_h$ or $R^*e_k$ vanish, the right hand side is zero and \eqref{serve} holds. If both $R^*e_h$, $R^*e_k$ are different from $0$, using \eqref{parti}   we get 
$$\begin{array}{l}
\ds \int_X \langle RD^2\varphi_n(x)R^*e_h, e_k\rangle \psi\,\nu(dx) =  \int_X\frac{\partial}{\partial R^*e_k} \bigg(\frac{\partial \varphi_n}{\partial R^*e_h} \bigg)  \psi\,\nu(dx) 
\\
\\
\ds = - \int_X\frac{\partial \psi}{\partial R^*e_k}  \frac{\partial \varphi_n}{\partial R^*e_h} \, \nu(dx) + \int_X v_{e_k}  \frac{\partial \varphi_n}{\partial R^*e_h} \, \psi\, \nu(dx)
\end{array}$$
 for every $n\in \N$. Since  $v_{e_k}\in L^{p'}(X, \nu)$,   $\partial \psi/\partial R^*e_k\in C_b(X)$, 
 $\lim_{n\to\infty}  \varphi_n =0$ and $ \lim_{n\to\infty}\partial \varphi_n/\partial R^*e_h$ $ =  \lim_{n\to\infty} \langle R\nabla \varphi_n, e_h\rangle =0$ in 
$L^p(X, \nu)$, letting $n\to\infty$ yields \eqref{serve}. 
\end{proof}
 
The closure of  the operator ${\mathcal D}^2_R $ in $L^p(X, \nu)$ is denoted by $(M_p, M^2_p)$.  
The Sobolev space $W^{2,p}(X, \nu)$ is defined as the domain of the closure of  such operator. So, any $\varphi \in W^{2,p}(X, \nu)$ is the $L^p(X, \nu)$-limit of a sequence $L^p(X, \nu)$-limit of a sequence $(\varphi_n) \subset C^2_b(X)$, such that $(R\nabla \varphi_n)$ and $(RD^2 \varphi_nR^*)$ are convergente sequences in 
  $L^p(X, \nu; X)$ and in $L^p(X, \nu; \mathcal L _2(X))$, respectively. In particular,  $\varphi \in W^{1,p}(X, \nu)$ and the sequence $(R\nabla \varphi_n)$ converges to $M_p\varphi$ in   $L^p(X, \nu; X)$.   $M^2_p\varphi$ is the $L^p(X, \nu; \mathcal L _2(X))$-limit of the sequence $(RD^2 \varphi_nR^*)$.

Moreover, $W^{2,p}(X, \nu)$  is a Banach space with the graph norm
$$\|\varphi\|_{W^{2,p}(X, \nu)} := \|\varphi\|_{L^p(X, \nu)} + \|M_p\varphi\|_{L^p(X, \nu; X)} + \|M^2_p\varphi\|_{L^p(X, \nu; \mathcal L_2(X))}  . $$
 If $\{e_k: \; k\in \N\}$ is any orthonormal basis of $X$ and we set $\partial_h \varphi := \langle M_p\varphi, e_h\rangle$, $\partial^2_{hk} \varphi := \langle M^2_p\varphi e_k, e_h\rangle$, we obtain 
$$\|\varphi\|_{W^{2,p}(X, \nu)} = \|\varphi\|_{L^p(X, \nu)} + \int_X \left( \sum_{h=1}^{\infty} (\partial_h \varphi)^2\right)^{1/p}d\nu + \int_X \left( \sum_{h, k=1}^{\infty} (\partial^2_{h k}\varphi)^2\right)^{1/p}d\nu . $$
Notice that it is not convenient to define the operator ${\mathcal D}^2_R $ in $C^2_b(X)$, because in general $RD^2\varphi (x)R^* \notin {\mathcal L}_2(X)$ for $\varphi\in C^2_b(X)$.

\subsection{Real measures and vector measures.}

We refer to \cite{BogaDiff} for a general treatment of differentiable measures, and to \cite{BogaMeas} for general measure theory. A good reference for vector valued measures is \cite{DU}.

We denote by ${\mathcal B}(X)$ the $\sigma$-algebra of the Borel sets of $X$. A real valued Borel measure is any countably additive function $m: {\mathcal B}(X) \mapsto \R$. For such $m$ we define a nonnegative measure $|m|$ by
$$|m|(B)=   \sup \;  \sum_{n=1}^{\infty} |m(B_n)|,  \quad B\in {\mathcal B}(X), $$
where the supremum is taken over all the at most countable partitions of $B$ into pairwise disjoint Borel sets $B_n$. If $A$ is any open set, we have
$$|m|(A) =   \sup \bigg\{ \int_X f\, dm: \; f\in C_b(X), \; f_{|X\setminus A}\equiv 0, \; \|f\|_{\infty} \leq 1\bigg\}.$$
$|m|(X)$ is called the total variation of $m$. The space of all the real Borel measures in $X$  is denoted by 
${\mathcal M}(X,\R)$, the map $m\mapsto |m|(X)$ is a norm in ${\mathcal M}(X,\R)$.

We shall be concerned also with $X$-valued Borel measures, namely the countably additive functions $m: {\mathcal B}(X) \mapsto X$. As in the real case, for every  $X$-valued Borel measure $m$ we define a nonnegative measure $|m|$ by
$$|m|(B)=   \sup  \sum_{n=1}^{\infty} \|m(B_n)\|,   $$
where the supremum is taken over all the at most countable partitions of $B$ into pairwise disjoint Borel sets $B_n$, 
and $|m|(X)$ is called the total variation of $m$. It is not hard to see that
\begin{equation}
\label{vartot}
|m|(X) = \sup \left\{ \int_X \langle F, dM\rangle :\; F\in C^1_b(X, X), \; \|F\|_{\infty} \leq 1 \right\}. 
\end{equation}
We denote by ${\mathcal M}(X,X)$   the space of all $X$-valued Borel measures with finite total variation. Every $m\in 
 {\mathcal M}(X,X)$ is absolutely continuous with respect to $|m|$, and it may be written as 
\begin{equation}
\label{polar}
m(dx)= \sigma(x) |m|(dx), 
\end{equation}
where $\sigma :X\mapsto X$ is a $|m|$-measurable unit vector field (namely, $\|\sigma (x)\| =1$ for $|m|$-a.e. $x\in X$). 
 
We shall use the following lemma about vector measures. 

\begin{Lemma}
\label{supmisure vett} 
Let $\{e_k:\;k\in \N\}$ be an orthonormal basis in $X$, and for every $k\in \N$ let $\mu_k$ be a real valued Borel measure, such that setting
$$M_n (B) = \sum_{k=1}^n \mu_k(B)e_k, \quad n\in \N, \; B\in \mathcal B(X), $$
we have
$$\sup_{n\in \N} |M_n|(X) := C<+\infty . $$
Then there exists a vector measure $M\in \mathcal M(X, X)$ such that $|M|(X)\leq C$ and $\langle M(B), e_k\rangle  = \mu_k(B)$ for every $k\in \N$, so that we have the representation
$$M(B) = \sum_{k=1}^{\infty} \mu_k(B) e_k, \quad B\in \mathcal B(X). $$
\end{Lemma} 
\begin{proof} For every Borel set $B$ set
$$\Sigma (B) := \sup_{n\in \N} |M_n|(B) = \lim_{n\to \infty}  |M_n|(B). $$
it is easily seen that $\Sigma $ is finitely additive and countably subadditive, and therefore it is countably additive. Each $\mu_k$  is absolutely continuous with respect to $\Sigma $;  denoting by $\rho_k$   the respective densities we have 
$$M_n(B) =  \sum_{k=1}^{n} \mu_k(B) e_k = \int_B  \sum_{k=1}^{n}\rho_k(x) e_k\, d\Sigma ,\quad n\in \N,  $$
so that, taking $B=X$, 
$$|M_n|(X) =  \int_X  \left\| \sum_{k=1}^{n}\rho_k(x) e_k\right\|  d\Sigma  = \int_X \left( \sum_{k=1}^{n} \rho_k(x)^2\right)^{1/2} d\Sigma  ,\quad n\in \N, $$
which implies, letting $n\to \infty$,  
$$\int_X \left( \sum_{k=1}^{\infty } \rho_k(x)^2\right)^{1/2} d\Sigma \leq C. $$
Therefore, for every Borel set $B$ the series $M_n(B) = \sum_{k=1}^n  \int_X \rho_k(x) \Sigma(dx) e_k$ converges in $X$, and setting 
$$M (B) = \lim_{n\to \infty} M_n(B) = \sum_{k=1}^{\infty} \mu_k(B) e_k,  $$
the vector measure $M$ enjoys all the claimed properties. 
\end{proof}

\section{$BV$ functions}
\label{sectionBV}

Let $u\in W^{1,p}(X, \nu)$ for some $p > 1$. For every $z\in X$ we rewrite \eqref{parti}  as
 \begin{equation}
\label{parti0}
\int_X u(\langle R\nabla \varphi, z\rangle - v_z \varphi)d\nu = - \int_X \langle M_pu,z\rangle \varphi \,d\nu, \quad \varphi\in C^1_b(X) . 
\end{equation}
The right hand side may be read as  minus the integral of $\varphi$ with respect to the real measure $m_z:=  \langle M_pu,z\rangle \nu$.  The total variation of such a measure is estimated by
$$|m_z|(X) = \|\langle M_pu,z\rangle\|_{L^1(X, \nu)} \leq \|M_pu\|_{L^1(X, \nu;X)} \|z\|. $$
In addition, setting
$$m(B) := \int_B M_pu \, d\nu, \quad B\in {\mathcal B}(X), $$
(the right hand side being an $X$-valued  integral), $m$ is a $X$-valued Borel measure with  total variation equal to $\|M_pu\|_{L^1(X, \nu)}$, such that  $\langle m(B), z\rangle = m_z(B)$ for every $z\in X$ and $B\in {\mathcal B}(X)$. 

$BV$ functions are defined in order to generalize the above formulae. More precisely,

\begin{Definition}
\label{def:BV}
Let $u\in L^1(X, \nu)$ be such that $uv_z\in L^1(X, \nu)$ for every $z\in X$. We say that $u\in BV(X, \nu)$  if there exists a measure  $m\in {\mathcal M}(X, X)$ such that, setting 
\begin{equation}
\label{mz}
m_z(B) := \langle m(B), z\rangle , \quad z\in X, \;B\in {\mathcal B}(X), 
\end{equation}
 we have
\begin{equation}
\label{BV}
\int_X u(\langle R\nabla \varphi, z\rangle - v_z \varphi)d\nu = - \int_X\varphi\,dm_z, \quad z\in X, \; \varphi\in C^1_b(X). 
\end{equation}
\end{Definition}
 
Recalling  formula \eqref{polar}, the definition of $BV$ function may be rephrased as follows: $u\in BV(X, \nu)$  if 
there exist a nonnegative measure   $ |m| \in {\mathcal M}(X, \R)$, and a $ |m|$-measurable vector field  $\sigma$ such that  $\|\sigma(x)\| = 1$ for $ |m| $-a.-e. $x\in X$, such that for every $z\in X\setminus \{0\}$ we have 
\begin{equation}
\label{BV1}
\int_X u(\langle R\nabla \varphi, z\rangle - v_z \varphi)d\nu = - \int_X\varphi\,\langle \sigma(x), z\rangle  |m| (dx), \quad \varphi\in C^1_b(X) . 
\end{equation}

If $u\in L^1(X, \nu)$ and $uv_z\in L^1(X, \nu)$ for every $z\in X$,  we set 
\begin{equation}
\label{V(u)} 
V(u)  :=   \sup \bigg\{ \int_X u \,M^* F \, d\nu: \; F\in \widetilde{C}^1_b(X ,X), \; \|F(x)\| \leq 1\; \forall x\in X\bigg\}
\end{equation}
We aim to prove that $u\in BV(X, \nu)$ iff $V(u)<\infty$. This will be done at the end of this section (Theorem \ref{Th:BV}).  To avoid the difficulties in handling vector valued measures, we will use  real valued measures: more precisely, to check whether a function $u$ belongs to $BV(X, \nu)$, it is convenient to look for  measures $m_z\in {\mathcal M}(X, \R)$ that satisfy \eqref{mz}, and then to recover the measure $m\in {\mathcal M}(X,X)$ of Definition \ref{def:BV}  from them. Therefore, we introduce the notion of $BV_z $ functions, as follows.

\begin{Definition}
\label{def:BVz}
Let $u\in L^1(X, \nu)$ and $z\in X$ be such that   $uv_z\in L^1(X, \nu)$. We say that $u\in BV_z(X, \nu)$ if there exists a real Borel measure $m_z\in {\mathcal M}(X, \R)$ such that 
\begin{equation}
\label{BVz}
\int_X u(\langle R\nabla \varphi, z\rangle - v_{z} \varphi)d\nu = -\int_X\varphi\,dm_z , \quad \varphi \in C^1_b(X). 
\end{equation}
\end{Definition}
 
We notice that if   $R^*z = 0$, every $u\in  L^1(X, \nu)$ such that   $uv_z\in L^1(X, \nu)$ belongs to $BV_z(X, \nu)$, since 
$\langle R\nabla \varphi, z\rangle =0$ for each $\varphi\in C^1_b(X)$, so that we can take  $m_z = uv_{z}\nu$ in \eqref{BVz}. So, from now on we may assume that $R^*z \neq  0$.
In any case we set 
\begin{equation}
\label{partial*z}
\partial_z^*\varphi (x) = \langle R\nabla \varphi, z\rangle - v_{z} \varphi,  \quad \varphi\in C^1_b(X), 
\end{equation}
\begin{equation}
\label{Tz}
T_z\varphi := \int_X u \,\partial_z^*\varphi \,d\nu, \quad \varphi \in C^1_b(X), 
\end{equation}
and 
\begin{equation}
\label{V_z(u)}
V_z(u)   :=    \sup \bigg\{ T_z\varphi: \; \varphi \in C^1_b(X), \; \|\varphi\|_{\infty} \leq 1 \bigg\} =
\ds \sup \bigg\{ \int_X u\,\partial_z^*\varphi \,d\nu : \; \varphi \in C^1_b(X), \; \|\varphi\|_{\infty} \leq 1 \bigg\}. 
\end{equation}
So, $u\in BV_z(X, \nu)$ iff there exists $m_z\in {\mathcal M}(X, \R)$   such that $T_z\varphi = -\int_X \varphi\,dm_z$, for every $\varphi\in C^1_b(X)$. In this case, $T_z$ has an obvious extension to the whole $C_b(X)$, whose norm (as an element of the dual space $(C_b(X))'$) is equal to $|m_z|(X)$. On the other hand, by Lemma \ref{Le:approx}(i) and the Dominated Convergence Theorem, the space $\in C^1_b(X)$ is dense in $C_b(X)$ endowed with the $L^1(X, \nu)$ norm. Therefore, 
\begin{equation}
\label{norma}
\begin{array}{lll}
|m_z|(X) & = & \ds \sup  \bigg\{ \int_X \varphi  \,dm_z : \; \varphi \in C_b(X), \; \|\varphi\|_{\infty} \leq 1 \bigg\} 
\\
\\
&= &
\ds  \sup  \bigg\{ \int_X \varphi  \,dm_z  : \; \varphi \in C^1_b(X), \; \|\varphi\|_{\infty} \leq 1 \bigg\} 
  \\
  \\
  & = & \ds \sup \bigg\{ \int_X u\,\partial_z^*\varphi \,d\nu : \; \varphi \in C^1_b(X), \; \|\varphi\|_{\infty} \leq 1 \bigg\} 
\\
\\
& = & V_z(u). 
\end{array}
\end{equation}
So, if $u\in BV_z(X, \nu)$ then $V_z(u)<\infty$. The converse holds too, and it is the main result of this section. 
 
\begin{Theorem}
\label{carBVz}
Let   $u\in L^1(X, \nu)$, and let $z\in X$ be such that $R^*z\neq 0$ and $uv_z\in L^1(X, \nu)$. 
 Then $u\in BV_z(X, \nu)$ if and only if $V_z(u) <\infty$. In this case we have  $  |m_z|(X) = V_z(u)$. 
\end{Theorem}
\begin{proof}
%

One of the implications  is immediate: as remarked above, if $u\in BV_z(X, \nu)$ then  $V_z(u)$ is finite, and equal to $|m_z|(X)$. 

The main part of the theorem is the proof of the converse, namely that if $V_z(u)<+\infty$ then $u\in BV_z(X, \nu)$. This will be done in two steps. 
In the first step we consider the case where $u$ has bounded support, and the general case is treated in the second step. 
 

\vspace{3mm}
\noindent {\em First step: the case $u\equiv 0$ in $X\setminus B(0,r)$. }

We claim that
\begin{equation}
\label{restrizmagg}
|T_z\varphi| \leq V_z(u)\|\varphi\|_{L^{\infty}(B(0,r))}, \quad \varphi\in  C^1_b(X). 
\end{equation}
Indeed, for any 
$\eps >0$ let  $\theta_{\eps} \in C^1(\R)$ be such that  
$$\|\theta_{\eps}\|_{\infty} \leq 1, \quad \theta_{\eps} \equiv 1\; \mbox{\rm in} \;[0, r^2], \quad  \theta_{\eps} \equiv 0\; \mbox{\rm in} \;[(r+\eps)^2, +\infty), $$
and we set 
$$\eta_{\eps} (x) := \theta_{\eps} (\|x\|^2), \quad x\in X. $$
For any  $\varphi\in  C^1_b(X)$, since $u\equiv 0$ in $X\setminus B(0, r)$ and 
$\varphi \equiv  \varphi\eta_{\eps} $ in  $B(0,r)$,  we have 
$$T_z\varphi = \int_X u\, \partial_z^* \varphi \,d\nu   = \int_{B(0,r)} u\, \partial_z^*( \varphi \eta_{\eps} )\,d\nu =  \int_{X} u\, \partial_z^*( \varphi \eta_{\eps} )\,d\nu = T_z(\varphi \eta_{\eps} ). $$
Therefore, 
$$|T_z\varphi| = |T_z(\varphi\eta_{\eps} )| \leq V_z(u) \|\varphi \eta_{\eps} \|_{\infty} \leq  V_z(u) \|\varphi\|_{L^{\infty}(B(0,R+ \eps))}, \quad \eps >0. $$
Since  $\varphi $ is uniformly continuous,  $\lim_{\eps\to 0}  \|\varphi\|_{L^{\infty}(B(0,R+ \eps))} = \|\varphi\|_{L^{\infty}(B(0,r))}$, and \eqref{restrizmagg} follows.

Now we endow    $B(0,r)$ with  the weak topology. Since  $B(0,r) $ is closed, convex and bounded, it is weakly compact. 
We denote by 
 $C_w( B(0,r))$ the space of the weakly continuous functions (namely, continuous with respect to the weak topology) from $B(0,r)$ to $\R$. 
 
 The elements of  $\mathcal{FC}^1_b(X)$ are weakly continuous functions. Their restrictions to 
$B(0,r) $ constitute  an algebra  $\mathcal A$ that separates points and contains the constants. By the Stone-Weierstrass Theorem,  $\mathcal A$ is dense  in $C_w( B(0,r))$. 

Every  $f\in \mathcal A$ is the restriction to  $B(0,r) $ of a function  $\widetilde{f} \in \mathcal{FC}^1_b(X)$. 
We set  
$$\widetilde{T}_zf := T_z\widetilde{f} . $$
$\widetilde{T}_z$ is well defined:  if  $f$ is the restriction to $B(0,r) $ of  both $\widetilde{f}_1$ and   $\widetilde{f}_2$, we have  $T_z\widetilde{f}_1= T_z\widetilde{f}_2$ by estimate  \eqref{restrizmagg}. Still  \eqref{restrizmagg} implies  $|\widetilde{T}_zf | \leq  V_z(u)\|f\|_{L^{\infty}(B(0,r))}$ for every $f\in \mathcal A$. So,  $\widetilde{T}_z $ is bounded  on  $\mathcal A$, and therefore it has a linear bounded extension  (still called  $\widetilde{T}_z$)  to the whole  $C_w( B(0,r))$. By the Riesz Theorem, such extension has an integral representation: there exists a (unique) real valued   measure $\mu$ on the Borel sets of $B(0,r)$ with respect to the weak topology, such that   
$$  |\mu|(B(0,r)) = \|\widetilde{T}_z\| = V_z(u), $$
and 
$$\widetilde{T}_z\varphi = - \int_{B(0,r)} \varphi \,d\mu , \quad \varphi \in C_w( B(0,r)). $$
In particular, 
$$T_z\varphi = - \int_{B(0,r)} \varphi \,d\mu , \quad \varphi \in \mathcal{FC}^1_b(X). $$
We recall that the Borel subsets of  $B(0,r)$ with respect to the weak topology coincide with the Borel subsets of $B(0,r)$ with respect to the strong topology. So,  $\mu$  is a Borel measure on  $B(0,r)$, that we extend in an obvious way to all Borel sets of $X$, setting
$$m_z(B) := \mu(B\cap B(0,r)), \quad B\in \mathcal B(X). $$
Therefore  we have
\begin{equation}
\label{quasi}
T_z\varphi = - \int_{X} \varphi \,dm_z , \quad \varphi \in \mathcal{FC}^1_b(X). 
\end{equation}
To finish Step 1 we have to extend the validity of   formula \eqref{quasi} to all $\varphi \in    C^1_b(X)$. 
For  $\varphi \in    C^1_b(X)$ we consider the cylindrical approximations   $\varphi _n(x) := \varphi(P_nx)$. 
For every $n\in \N$, \eqref{quasi} yields 
\begin{equation}
\label{quasi_n}
T_z\varphi_n = -  \int_{X} \varphi_n \,dm_z . 
\end{equation}
We have $\lim_{n\to \infty} \varphi_n(x) = \varphi(x)$ for every $x\in X$, and $|\varphi_n(x) |\leq \|\varphi\|_{\infty}$. Therefore, the right-hand side of \eqref{quasi_n} goes to 
$  -  \int_{X} \varphi \,dm_z$ as $n\to \infty$. 
Moreover,  by Lemma \ref{Le:approx}(i) we have 
$\lim_{n\to \infty} \partial \varphi_n/\partial (R^*z)(x)$  $=$  $\partial \varphi/ \partial (R^*z)(x)$ for every $x\in X$
and $|\partial \varphi_n/\partial(R^*z)(x)| \leq \|\nabla \varphi\|_{\infty} /\|R^*z\| $. It follows 
$|\partial^*_z \varphi _n(x) |\leq  \|\nabla \varphi\|_{\infty} /\|R^*z\|  + |v_z(x)| \,\|\varphi\|_{\infty}$. By the Dominated Convergence Theorem, the left-hand side of  \eqref{quasi_n} goes to 
$ T_z\varphi$ as $n\to \infty$. So, letting $n\to \infty$ in \eqref{quasi_n} yields 
 $T_z\varphi=  - \int_{X} \varphi \,d\mu_z$, 
and recalling \eqref{norma} the statement is proved. 

\vspace{3mm}

Before going to Step 2, we prove an intuitive property of the measure $m_z$ of Step 1. We show that 
if 
$u\equiv 0$ on  $B(0,r_0)$ for some  $r_0<r$, then $|m_z|(\oo B(0,r_0)  ) =0$, where $\oo B(0,r_0) $ is the open ball centered at $0$ with radius $r$. To this aim we recall that
$$|m_z|(\oo B(0,r_0) )  = \sup \bigg\{ \int_X \varphi\,d m_z , \; \varphi\in C_b(X), \; \|\varphi\|_{\infty} \leq 1, \; \varphi (x) =0\; \mbox{\rm for}\; \|x\| \geq r_0\bigg\}$$
Fixed any  $\eps >0$ let  $\varphi\in  C_b(X)$ be such that  $\|\varphi\|_{\infty} \leq 1$,  $\varphi (x)= 0$ for $ \|x\|\geq r_0$ and 
$$|\mu_z|(\oo B(0,r_0) ) \leq \int_X \varphi \,d\mu_z + \eps. $$
By Lemma \ref{Le:approx}(iii) there exists a double-index sequence of $C^1_b(X)$ functions  $(\widetilde{\varphi}_{k,n})$ such that $|\widetilde{\varphi}_{k,n}(x)\|\leq \|\varphi\|_{\infty}$ for every $x\in X$, $\lim_{k\to \infty} \lim_{n\to \infty} \widetilde{\varphi}_{k,n}(x) = \varphi(x)$ for every $x$, and $\widetilde{\varphi}_{k,n}(x) = 0$ if $\|x\|\geq r_0$. By the Dominated Convergence Theorem  we get 
$$\lim_{k\to \infty} \lim_{n\to \infty} \int_X \widetilde{\varphi}_{k,n}\,dm_z =  \int_X \varphi \,dm_z$$
so that there are $k(\eps)$, $n(\eps)\in \N$ such that 

$$|m_z|(\oo B(0,r_0) ) \leq  \int_X  \widetilde{\varphi}_{k(\eps),n(\eps)} \,dm_z + 2\eps. $$
Since $ \widetilde{\varphi}_{k(\eps),n(\eps)}$ belongs to  $C^1_b(X)$ we have
$$ \int_X \widetilde{\varphi}_{k(\eps),n(\eps)} \,dm_z = T( \widetilde{\varphi}_{k(\eps),n(\eps)}) = \int_X u\, \partial_z^*( \widetilde{\varphi}_{k(\eps),n(\eps)} )d\nu   $$
and the last integral is zero,  because $u$ vanishes $\nu$-a.e in $B(0,r_0)$ and $ \widetilde{\varphi}_{k(\eps),n(\eps)}$ vanishes in $X\setminus B(0,r_0)$, so that   $\partial_z^*( \widetilde{\varphi}_{k(\eps),n(\eps)} )$ vanishes $\nu$-a.e. in 
$X\setminus B(0,r_0)$. 
Therefore  $|m_z|(\oo B(0,r_0) ) \leq   2\eps $ for every  $\eps >0$, which implies  $|m_z|(\oo B(0,r_0) )=0$.  

\vspace{3mm}

\noindent{\em Second step: the general case. }
 
Let $\theta \in C^1(\R)$ be an even function supported in  $(-1,1)$, such that  $\theta \equiv 1$ in a neighborhood of $0$, and such that  $\theta(\xi) = 1-\theta(1-\xi)$ for every $\xi\in (0,1)$. For instance, one can take the even extension of the function that is equal to $1$ in $[0, 1/4]$, to $\cos (2(\xi-1/4)/\pi)$ for $\xi\in [1/4, 3/4]$, to $0$ for $\xi \geq 3/4$. 
 
The set of functions $\{\theta_k:\;k\in \Z\}$ defined by  $\theta_k(\xi) := \theta(\xi -k)$, $k\in \Z$, is a  $C^1$ partition of $1$ in $\R$. Each $\theta_k$ is supported  in   $(k-1, k+1)$, and therefore every  $\xi\in \R$ belongs to the support of at most two of them. Define
$$\eta_k(x) := \theta_k(\|x\|), \quad k\in \N \cup \{0\}. $$
Every $\eta_k$ belongs to  $C^1_b(X)$, since every $ \theta_k$ belongs to  $C^1_b(\R)$ and it is constant in a neighborhood of $0$. Moreover 
$$\sum_{k=0}^{\infty} \eta_k(x) =1, \quad x\in X, $$
where for every $x\in X$ the series converges because it has at most two nonzero addenda; more precisely, 
if  $k\leq \|x\| \leq k+1$ we have $\eta_h(x)=0$ for  $h\leq k-1$ and for $h\geq k+2$. So, we have
$$u(x) = \sum_{k=0}^{\infty} \eta_k (x) u(x) =:   \sum_{k=0}^{\infty} u_k(x), \quad x\in X,  $$
where $u_k := u\eta_k$ vanishes outside $\{x: \; k-1 < x< k+1\}$. In particular, its support is contained in $B(0, k+1)$
and  for every $\varphi\in C^1_b(X)$ we have
$$\begin{array}{l}
\ds \bigg| \int_X u\,\eta _k \partial^*_z \varphi\,d\nu \bigg| = \bigg| \int_X u (\partial^*_z (\varphi\eta_k) -\varphi \partial_{R^*z}\eta_k)\,d\nu \bigg|
\\
\\
\ds
 \leq V(u)\| \varphi\eta_k\|_{\infty} + \bigg| \int_X u \,\varphi \,\partial_{R^*z}\eta_k\,d\nu \bigg| \leq V(u) \| \varphi \|_{\infty} + C_k\| \varphi \|_{\infty}
 \end{array}$$
where 
$$C_k = \int_{\{x: \; k-1< \|x\|<k+1\}} |u\, \partial_{R^*z}\eta_k| d\nu \leq \int_{\{x: \; k-1< \|x\| < k+1\}}|u | d\nu \|\theta'\|_{\infty}.  $$
By Step 1 there exists a Borel measure  $m_k$, supported  in $\{x: \; k-1\leq x\leq k+1\}$, whose total variation does not exceed $V_{z}(u) + C_k$, such that 
\begin{equation}
\label{Tk}
T_k\varphi := - \int_X u\,\eta _k  \partial^*_z \varphi\,d\nu =-  \int_{\{x: \; k-1< \|x\| < k+1\}}\varphi \,dm_k, \quad \varphi \in C^1_b(X). 
\end{equation}
The consequent estimate  $|m_k|(X) \leq V_z(u) + C_k$  is not enough for our aim. To improve it we recall \eqref{norma}, that yields
$$|m_k|(X) = V_z(\eta_ku) =   \sup \bigg\{ \int_X u\,\eta_k\, \partial^*_z \varphi\,d\nu: \; \; \varphi\in C^1_b(X), \; \|\varphi\|_{\infty} \leq 1 \bigg\}. $$
For every  $k\in \N \cap\{0\}$ let   $\varphi_k \in C^1_b(X)$ be such that  $\|\varphi_k\|_{\infty} \leq 1$ and such that 
\begin{equation}
\label{mk}
|m_k| (X)  \leq  \int_X u\,\eta_k\, \partial_z^* \varphi_k\,d\nu + \frac{1}{2^k}. 
\end{equation}
Set 
$$\varphi(x) := \sum_{k=1}^{\infty} \eta_k(x) \varphi_k(x), \quad x\in X. $$
The function $\varphi$ is well defined and it belongs to  $C^1_b(X)$, since every $x$ has a neighborhood where all the summands vanish, 
except for at most two of them. Moreover, 
$$|\varphi(x) | \leq  \sum_{k=1}^{\infty} \eta_k(x) |\varphi_k(x)| \leq  \sum_{k=1}^{\infty} \eta_k(x) = 1, \quad x\in X, $$
so that $\|\varphi\|_{\infty} \leq 1$. Then we have
\begin{equation}
\label{ricordare}
\int_X u\,\partial^*_z\varphi \,d\nu \leq V_z (u). 
\end{equation}
On the other hand, the left hand side is equal to
 $\int_X u\, \partial^*_z (\sum_{k=1}^{\infty} \eta_k  \varphi_k)\,d\nu$, and recalling that $\partial^*_z( \eta_k  \varphi_k) =  \eta_k \partial^*_z  \varphi_k +  \varphi_k\partial_{R^*z}\eta_k$ we get 
$$\int_X u\,\partial^*_z\varphi \,d\nu =  \sum_{k=1}^{\infty}\int_X u\, \eta_k  \,\partial^*_z\varphi_k \, d\nu + \sum_{k=1}^{\infty}\int_Xu \, \varphi_k \,
 \partial_{R^*z} \eta_k \, d\nu, $$
where for every $k$ we have 
$$\bigg| \int_Xu\,  \varphi_k  \,\partial_{R^*z} \eta_k \, d\nu\bigg| \leq  \int_{\{x:\, k-1<\|x\|<k\} }  |u \, \partial_{R^*z}  \eta_k| d\nu
\leq   \int_{\{x:\, k-1<\|x\|<k\} }  |u| d\nu   \|\theta' \|_{\infty}  , $$
so that, recalling \eqref{mk} and summing up, 
$$\begin{array}{lll}
\ds \int_X u\,\partial^*_z\varphi \,d\nu &  \geq & \ds \sum_{k=1}^{\infty}\left( |m_k|(X) - \frac{1}{2^k} \right)- \sum_{k=1}^{\infty} \|u\|_{L^1 (B(0, k+1)\setminus B(0,k-1))}  \|\theta' \|_{\infty}
\\
\\
& \geq & \ds
 \sum_{k=1}^{\infty} |m_k|(X)  - 1 -2\|u\|_{L^1 (X, \nu)}  \|\theta' \|_{\infty} 
 \end{array}$$
that implies, through  \eqref{ricordare}, 
$$ \sum_{k=1}^{\infty} |m_k| (X)\leq V_z(u) +1 +  2\|u\|_{L^1 (X, \nu)}   \|\theta' \|_{\infty} . $$
The signed measure  
$$m_z:=  \sum_{k=1}^{\infty}m_k$$
is therefore well defined, and for every $\varphi \in C^1_b(X)$ we have, using \eqref{Tk}, 
$$\int_X \varphi \, dm_z = \sum_{k=1}^{\infty} \int_X \varphi \, dm_k =  \sum_{k=1}^{\infty} \int_X u\,\eta_k \,\partial_z^* \varphi \, d\nu =  \int_X u \, \partial_z^* \varphi \,d\nu = T_z\varphi .$$
So,  $m_z$ is the measure that we were looking for. We already know that $|m_z|(X) = V_z(u)$, and the proof is complete. 
\end{proof}

Now we move to $BV$ functions. First we prove an easy lemma. 

\begin{Lemma}
\label{BVimpliesBVz}
Let $u\in L^1(X, \nu)$   be such that $uv_z\in L^1(X, \nu)$ for every $z\in X$ and $V(u) <+\infty$. Then $u\in BV_z(X, \nu)$ and  $V_z(u)\leq   V(u)\|z\|$ for every $z\in X\setminus  \{0\}$.  
\end{Lemma}
\begin{proof}  For every $\varphi\in C^1_b(X)$ with $\|\varphi\|_{\infty}\leq 1$ and $z\in X\setminus \{0\}$  consider the vector field $F(x) := \varphi(x) z/\|z\|$. By the definition of $V(u)$ we get
$$\int_X u( \langle R\nabla \varphi, z\rangle - v_z \varphi) d\nu = \int_X u \, M^*F\,d\nu \|z\| \leq V(u)\|z\| $$
and the statement follows. \end{proof}

\begin{Theorem}
\label{Th:BV}
Let $u\in L^p(X, \nu)$ for some $p>1$. Then, $u\in BV(X, \nu)$ if and only if  $V(u) <+\infty$. 
In this case, the measure $m$ of Definition \ref{def:BV} satisfies $|m|(X) = V(u)$. 
\end{Theorem}
\begin{proof}
The proof  that   $u\in BV(X, \nu) \Rightarrow V(u) <+\infty$  is easy. Indeed, let $m = \sigma |m| \in {\mathcal M}(X, X)$ be such that \eqref{BV1}
 holds for each $z\in X$. For every $F\in 
\widetilde{C}^1_b(X, X)$ such that $\|F\|_{\infty} \leq 1$, $F(x) =  \sum_{i=1}^n f_i(x)z_i$, we have 
$$   \int_X u M^* F \, d\nu 
= \int_X u \sum_{i=1}^n (-\langle R\nabla f_i, z_i\rangle + v_{z_i} f_i)d\nu  
= \int_X \langle \sigma(x), \sum_{i=1}^n f_i(x)z_i\rangle |m|(dx) =  \int_X \langle \sigma, F\rangle |m|(dx) .$$
In the last integral we have
$$|\langle \sigma (x), F(x)\rangle |  \leq \| \sigma(x)\|_X \|F(x)\|  \leq 1, \quad |m|-a.e.\;x\in X, $$
so that its modulus does not exceed  $  |m|(X) $. Taking the supremum over all $F\in 
\widetilde{C}^1_b(X, X)$ we get  $V(u)\leq |m|(X) <+\infty$.

To prove the converse we fix $u$ such that $V(u)<+\infty$, and consequently $V_z(u)<+\infty$ for every $z\in X$.  

We fix  any orthonormal basis $\{e_k: \; k\in \N\}$ of $X$. We consider  the real measures $m_{e_k}$ constructed in Theorem \ref{carBVz}, if $R^*e_k\neq 0$. 
If $R^*e_k =  0$ we set $m_{e_k}(dx) := u(x)v_{e_k}(x)\nu(dx)$ (see the comments after Definition \ref{def:BVz}). Our aim is to prove that 
\begin{equation}
\label{mvett}
m(B) : = \sum_{k=1}^{\infty} m_{e_k}(B) e_k, \quad B\in \mathcal B(X), 
\end{equation}
is a well defined vector measure belonging to $\mathcal M(X, X)$, that satisfies \eqref{BV} and such that $|m|(X) \leq V(u)$. 

To prove that $m$ is well defined we consider a sequence of vector measures with finite dimensional range, 
$$ M_n(B) =  \sum_{k=1}^{n} m_{e_k}(B) e_k, \quad B\in \mathcal B(X). $$
By \eqref{vartot}, for every $n\in \N$ we have
$$\begin{array}{lll}
|M_n|(X) & = & \ds  \sup   \left\{ \int_X \langle F, dM_n\rangle :\; F\in C^1_b(X, X), \; \|F\|_{\infty} \leq 1 \right\} 
\\
\\
&=& \ds  \sup   \left\{ \int_X \langle F, dM_n\rangle :\; F\in C^1_b(X, P_n(X)), \; \|F\|_{\infty} \leq 1 \right\} , \end{array}$$
where $P_n$ is the orthogonal projection defined in \eqref{Pn}. Each vector field $F\in C^1_b(X, P_n(X))$ may be written as $F(x) = \sum_{k=1}^n f_k(x) e_k$, with $f_k\in C_b(X)$. Therefore it belongs to $\widetilde{C}^1_b(X, X)$, and by Theorem \ref{carBVz} we have 
$$\int_X \langle F, dM_n\rangle = \int_X \sum_{k=1}^n f_{k} \, d m_{e_k} = -  \int_X \sum_{k=1}^n u\,  \partial^*_{e_k} f_k\,d\nu =  \int_X  u \, M^*F\,d\nu
 \leq V(u)\|F\|_{\infty},  $$
so that $|M_n|(X) \leq V(u)$ for every $n\in \N$. By lemma \ref{supmisure vett}, the series $\sum_{k=1}^{n} m_{e_k}(B) e_k$ converges for every Borel set $B$, and  the formula  \eqref{mvett} defines a vector measure $m\in \mathcal M(X,X)$ such that $|m|(X) \leq V(u)$. By definition, we have $\langle m(B), e_k\rangle = m_{e_k}(B)$ for every $k\in \N$ and therefore   \eqref{BV} holds if $z=e_k$ for some $k\in \N$. Consequently, since the function $z\mapsto m_z$ is linear,  \eqref{BV} holds if $z$ is a linear combination of the $e_k$. If  $z$  is not  a linear combination of the $e_k$, we approach it by  $z_n:= P_nz$. For each  $n$ and $ \varphi \in C^1_b(X)$ we have  
\begin{equation}
\label{approxM_n}
\int_X u( \langle R \nabla \varphi, z_n\rangle -  v_{z_n} \varphi)d\nu = - \int_X \varphi \,d m_{P_nz} =  - \int_X \varphi \sum_{k=1}^{n} \langle z, e_k\rangle  dm_{e_k}
 =  - \int_X \varphi   \langle z, M_n(dx) \rangle   . 
 \end{equation}
Letting $n\to \infty$, $\langle R \nabla \varphi, z_n\rangle \to \langle R \nabla \varphi, z\rangle$ and its modulus does not exceed  $\|R\nabla \varphi\|_{\infty} \|z\|$. Moreover, $\lim_{n\to\infty} v_{z_n} = v_z$ in $L^{p'}(X, \nu)$, 
by \eqref{vz}. So, the left hand side of \eqref{approxM_n}
converges to $\int_X u( \langle R \nabla \varphi, z \rangle -  v_{z } \varphi)d\nu$ as $n\to \infty$. By  Lemma  \ref{supmisure vett} the sequence of measures $(  \langle z, M_n(dx) \rangle  )$ converge weakly to  $ \langle z, m(dx) \rangle  $, so that the right hand side converges to $ - \int_X \varphi   \langle z, m(dx) \rangle  $ as $n\to \infty$. Therefore, 
\eqref{BV} holds for every $z\in X$. 
\end{proof}

\begin{Remark}
The assumption $u\in L^p(X, \nu)$ with $p>1$ is used only in the very last step of the proof, to prove that   $\int_X u  \, v_{P_nz} \varphi\, d\nu\to 
\int_X u  \, v_{z} \varphi\, d\nu$ as $n\to\infty$. To this aim the assumption that $u$ and $uv_z$ belong to $L^1(X, \nu)$ for every $z$ is not enough. If $\nu$ is a centered nondegenerate Gaussian measure, it is enough that $u$ belongs to a suitable Orlicz space $Y:=L(\log L)^{1/2}(X, \nu)$, since the space $X_{\nu}^*$ of all the functions $v_z$ is embedded in the dual space $Y'$. See \cite{FH}. 
\end{Remark}

As we may  expect, we can take other test functions without affecting the definitions of $V_z(u)$ and of $V(u)$. Indeed, the following lemma holds.

\begin{Lemma}
\label{Le:caratterizzazioni}
Let $u\in L^1(X, \nu)$. Then
\begin{itemize}
\item[(i)] For every $z\in X$ such that $uv_z\in L^1(X, \nu)$   we have
\begin{equation}
\label{equiv1}
V_z(u) =   \sup\bigg\{ \int_X u\,\partial _z^* \varphi\, d\nu :\; \varphi\in \mathcal{FC}^1_b(X), \; \|\varphi\|_{\infty} \leq 1\bigg\}
\end{equation}
If in addition $u\in L^p(X, \nu)$ for some $p>1$, 
\begin{equation}
\label{equiv1+}
V_z(u) = \sup\bigg\{ \int_X u\,\partial _z^* \varphi\, d\nu :\; \varphi\in W^{1, p'}(X, \nu)\cap L^{\infty}(X, \nu), \; \|\varphi\|_{\infty} \leq 1\bigg\}
\end{equation}
\item[(ii)] If $uv_z\in L^1(X, \nu)$  for every $z\in X$ we have
\begin{equation}
\label{equiv2}
V(u) =  \sup\bigg\{ \int_X u\, M^*F \, d\nu :\; F\in \widetilde{\mathcal{FC}}^1_b(X, X), \; \|F\|_{L^{\infty}(X, \nu;X)} \leq 1\bigg\}
\end{equation}
If in addition $u\in L^p(X, \nu)$ for some $p>1$, 
\begin{equation}
\label{equiv2+}
V(u) = \sup\bigg\{ \int_X u\, M^*F \, d\nu :\; F\in \widetilde{W}^{1,p'}(X, X)\cap L^{\infty}(X, \nu;X), \; \|F\|_{L^{\infty}(X, \nu;X)} \leq 1\bigg\}, 
\end{equation}
and  for any orthonormal basis $\{e_k:\; k\in \N\}$ of $X$ we have
\begin{equation}
\label{equiv3}
V(u) = \sup\bigg\{ \int_X u\, M^*F \, d\nu :\; F(x) = \sum_{i=1}^n \varphi_i(x) e_i, \; n\in \N, \; \varphi_i\in C^1_b(X),        \; \|F\|_{\infty}   \leq 1\bigg\}. 
\end{equation}
\end{itemize}
\end{Lemma}
\begin{proof}
Let us prove Statement (i). Since $\mathcal{FC}^1_b(X)\subset C^1_b(X)$, the supremum in the right hand side of \eqref{equiv1}(a) is less or equal to $V_z(u)$. To prove the equality in \eqref{equiv1} we use Lemma \ref{Le:approx}: we approach  any $\varphi\in C^1_b(X)$ such that $ \|\varphi\|_{\infty} \leq 1$ by the sequence  of cylindrical functions $(\varphi\circ P_n)$. Lemma \ref{Le:approx}(i) and the Dominated Convergence Theorem yield $\int_X u\, \partial _z^* \varphi\, d\nu = \lim_{n\to\infty} \int_X u\, \partial _z^* \varphi_n\, d\nu $, and  \eqref{equiv1}   follows. 
 
The proof of \eqref{equiv1+}  is a bit more complicated. Of course,  $V_z(u)$ is less or equal to the supremum in the right hand side of \eqref{equiv1}. To prove the equality,
for every $\eps>0$ and $\varphi\in W^{1, p'}(X, \nu)$ such that $\|\varphi\|_{L^{\infty}(X, \nu)} \leq 1$ we shall exhibit a function $\psi\in C^1_b(X)$ such that $\|\psi\|_{L^{\infty}(X, \nu)} \leq 1$ and 
\begin{equation}
\label{approxVz}
\bigg| \int_X u\, \partial _z^*( \varphi-\psi) \, d\nu\bigg| \leq  \eps. 
\end{equation}
We fix $\delta >0$ such that  
\begin{equation}
\label{delta}
\bigg| \frac{1}{1+\delta} \int_X u\, \partial _z^*  \varphi  \, d\nu -  \int_X u\, \partial _z^*  \varphi  \, d\nu\bigg| \leq \frac{\eps}{2}, 
\end{equation}
and we approach $\varphi$ pointwise $\nu$-a.e.  and in $W^{1,p'}(X, \nu)$ by a sequence of $C^1_b(X) $ functions $(\varphi_n)$ such that $\|\varphi_n\|_{\infty} \leq 1+\delta$, 
given by Lemma \ref{Le:Sobolev}(i). So, 
 $\lim_{n\to \infty} \langle R\nabla  \varphi_n, z\rangle  = \langle R\nabla\varphi, z\rangle $ in $L^{p'}(X, \nu)$. Moreover, since $ \varphi_n \to \varphi $ in $L^{q}(X, \nu)$ for every $q\in [1, \infty)$ and $v_z\in L^r(X, \nu)$ for every $r\in [1, +\infty)$, we have
 $\lim_{n\to \infty}   \varphi_n v_z  = \varphi v_z$ in $L^{p'}(X, \nu)$. Summing up,  $\lim_{n\to \infty} \partial^*_z  \varphi_n = \partial^*_z \varphi  $ in $L^{p'}(X, \nu)$
and therefore  $\lim_{n\to \infty} \partial^*_z  \varphi_n/(1+\delta) = \partial^*_z \varphi /(1+\delta)$ in $L^{p'}(X, \nu)$. So, for $n$ large enough we have 
$$\bigg| \frac{1}{1+\delta} \int_X u\partial _z^*  \varphi  \, d\nu - \frac{1}{1+\delta} \int_X u\partial _z^*  \varphi_n\, d\nu\bigg| \leq \frac{\eps}{2}$$
which yields \eqref{approxVz} with $\psi =  \varphi_n$, taking \eqref{delta} into account. 
\vspace{3mm}
 
 Concerning Statement (ii), the proofs of \eqref{equiv2} and of \eqref{equiv2+}  are the same of Statement (i), and they are left to the reader. 

Since the supremum in the right hand side of \eqref{equiv3} is less or equal to $V(u)$, 
to prove the equality in \eqref{equiv3} it is sufficient to approach every $F\in \widetilde{C}^1_b(X, X)$ with  $\|F\|_{L^{\infty}(X, \nu;X)}\leq 1$ by a sequence of vector fields $\Phi_n $ with range in the linear span of the basis, such that $\|\Phi_n\|_{L^{\infty}(X, \nu;X)}\leq 1$ and such that 
 \begin{equation}
 \label{passlimite}
 \lim_{n\to \infty} \int_X u\,M^*\Phi_n\, d\nu =  \int_X u\,M^*F d\nu . 
\end{equation}
 This is easily done, approaching  $F(x)= \sum_{i=1}^k f_i (x) z_i$ by 
 $F_n(x) := \sum_{i=1}^k f_i (x) P_nz_i =  \sum_{j=1}^n \varphi_j(x) e_j$, with   $ \varphi_j(x)= \sum_{i=1}^k f_i (x) \langle z_i , e_j\rangle $. 
 Then $F_n\to F$ in $L^{\infty}(X, \nu;X)$ since
 $$\|F_n-F\|_{\infty} = \sup_{x\in X} \| \sum_{i=1}^k f_i(x)(P_nz_i -z_i)\| \leq \bigg(\sum_{i=1}^k \|f_i\|_{\infty}^2\bigg)^{1/2} \bigg( \sum_{i=1}^k\|P_nz_i -z_i\|^2\bigg)^{1/2}. $$
 Moreover, 
 $M^*F_n =  \sum_{i=1}^k (\langle R\nabla f_i, P_nz_i\rangle - v_{P_nz_i}f_i)$ converges to 
$M^* F=  \sum_{i=1}^k (\langle R\nabla f_i,  z_i\rangle - v_{ z_i}f_i)$ in $L^{p'}(X, \nu;X)$  by the Dominated Convergence Theorem and estimate \eqref{vz}. 
 Therefore, setting  $\Phi_n:=  F_n/\|F_n\|_{L^{\infty}(X, \nu;X)}$ for sufficiently large $n$ (such that $ F_n \not\equiv 0$),  the vector fields $\Phi_n$ have values in the unit ball of $X$, have range in the linear span of the basis, and \eqref{passlimite} holds. 
  \end{proof}

 \subsection{Other classes of Sobolev and $BV$ functions}
 \label{subs:other}
 
Throughout this section we assume that 

\begin{Hypothesis}
\label{Hyp:R}
$R=R^*$ is one to one,  and   there exists an orthonormal basis $\{h_k:\; k\in \N\}$ of $X$ consisting of eigenvectors of $R$. 
\end{Hypothesis}

As mentioned in Remark \ref{Sobolevgenerali}, the gradient operator $\nabla:C^1_b(X, \nu)\mapsto L^p(X, \nu;X)$ is closable in $L^p(X, \nu)$ for every $p\geq 1$. We call $W^{1,p}_0(X, \nu)$ the domain of the closure, which is  denoted by $\nabla_p$. $W^{1,p}_0(X, \nu)$ is endowed with the graph norm
  $$\|f\|_{W^{1,p}_0(X, \nu)} := \|f\|_{L^p(X, \nu)} + \|\nabla_p f\|_{L^p(X, \nu;X)} . $$
%
Moreover, for every $f\in W^{1,p}_0(X, \nu)$ and $z\in X$ we set 
$$\frac{\partial f}{\partial z}:= \langle \nabla_pf, z\rangle. $$

Since $R$ is a bounded operator, it follows immediately from the definition that  
\begin{equation}
\label{relazione}
W^{1,p}_0(X, \nu) \subset W^{1,p}(X, \nu), \quad M_pf =R\nabla_pf , \;f\in W^{1,p}_0(X, \nu). 
\end{equation}
We shall use the following lemma, 

\begin{Lemma}
\label{Sobregolari}
Let $\varphi\in C^1(X)$ be such that $\|\nabla \varphi\|$ is bounded in $\varphi^{-1}(-r, r)$ for every $r>0$, and $\int_X (|\varphi|^p + \|\nabla \varphi\|^p) d\nu <+\infty$, for some $p\geq 1$. Then $\varphi\in W^{1,p}_0(X, \nu)$, and $\nabla_p\varphi = \nabla \varphi$. 
\end{Lemma}
\begin{proof} The  prooof is  similar to the one Lemma 2.4 of \cite{TAMS}, and it is omitted. \end{proof}

The integration formula \eqref{h1'} is rewritten as
$$\int_X \langle \nabla \varphi, y\rangle d\nu = \int_X  \langle \nabla \varphi, R (R^{-1}y)\rangle d\nu = \int_X \varphi\, v_{R^{-1}y} d\nu, \quad \varphi\in C^1_b(X), \; y\in R(X). $$

Concerning the adjoint operators, the domain of $M^*_p$ in $L^{p'}(X, \nu;X)$  is equal to $D(\nabla_p^*R)$, and we have 
\begin{equation}
\label{duali}
M^*_pF = \nabla_p^*(RF), \quad  F\in D( M^*_p) = D(\nabla_p^*R). 
\end{equation}

Since $W^{1,p}_0(X, \nu)\subset W^{1,q}_0(X, \nu)$ and $\nabla_p = \nabla_q$ on $W^{1,p}_0(X, \nu)$ for $p\geq q$, we have $D( \nabla^*_q)\subset D(\nabla ^*_p)$ and $\nabla^*_q = \nabla^*_p$ for   $p\geq q$. We set 
$$\nabla^*F : = \nabla_pF, \quad F\in \cap_{p>1} D(\nabla ^*_p). $$
In particular, every vector field $F(x) = \sum_{i=1}^n f_i(x) y_i$, with $f_i\in C^1_b(X)$ and $y_i\in R(X)\setminus \{0\}$ for $i=1, \ldots, n$, is in the domain of $\nabla_p^*$ for every $p>1$, and 
\begin{equation}
\label{nabla_p^*}
\nabla^* F(x) = \sum_{i=1}^n \left( -\frac{\partial f_i}{\partial y_i} + v_{R^{-1}y_i}f_i\right). 
\end{equation}
Such vector fields are the appropriate test functions for the definition of $BV$ functions. So, according to \eqref{C1btilde}, we set
$$\widetilde{C}^1_b(X, R(X)) := \{ F\in C^1_b(X, X):\; F(X) \mbox{ is contained in a finite dimensional subspace of}\;R(X)\}. $$
\begin{Definition}
\label{def:BV0}
For each  $u\in L^1(X, \nu)$ such that $uv_z\in L^1(X, \nu)$ for every $z\in X$  we set
\begin{equation}
\label{V0(u)} 
V_0(u)  :=   \sup \bigg\{ \int_X u\, \nabla^* F \, d\nu: \; F\in \widetilde{C}^1_b(X ,R(X)), \; \|F(x)\| \leq 1\; \forall x\in X\bigg\}, 
\end{equation}
and we say that $u\in BV_0(X, \nu)$  if there exists a measure  $m_0\in {\mathcal M}(X, X)$ such that for each $y\in X$ we have
\begin{equation}
\label{BV0}
\int_X u(\langle  \nabla \varphi, y\rangle - v_{R^{-1}y} \varphi)d\nu = - \int_X\varphi\,d\langle m_0,y\rangle, \quad   \varphi\in C^1_b(X), 
\end{equation}
where
$$\langle m_0,y\rangle(B) := \langle m(B), y\rangle , \quad  B\in {\mathcal B}(X). $$
\end{Definition}

For every $F\in  \widetilde{C}^1_b(X ,X)$  we have of course   $RF\in  \widetilde{C}^1_b(X ,R(X))$ and 
$\int_X u \,M^*F \,d\nu = \int_X u \nabla^* (RF) \,d\nu$ by \eqref{duali}. Therefore, if $V_0(u) <+\infty$ we have $V(u) \leq V_0(u)\|R\|_{\mathcal L(X)} <+\infty$. 
Correspondingly, if $u\in BV_0(X, \nu)$ then $u\in BV(X, \nu)$ and the measure $m$ of Definition \ref{def:BV} is just given by $Rm_0$. 

As in the case of the $BV$ functions, the test functions in the definition of $V_0(u)$ may be chosen in different sets. The proof of the next lemma is similar to the proof of Lemma \ref{Le:caratterizzazioni} with obvious modifications, and it is left to the reader. We recall that the spaces $\widetilde{\mathcal{FC}}^1_b(X, R(X))$, $\widetilde{W}^{1,p'}(X,\nu; R(X))$ are defined in \eqref{FC1btilde}, \eqref{W1ptilde}, respectively. 

\begin{Lemma}
\label{Le:caratterizzazioni0}
Let $u\in L^1(X, \nu)$ be such that $uv_z\in L^1(X, \nu)$ for every $z\in X$. Then
\begin{equation}
\label{equiv20}
V_0(u) =   \sup\bigg\{ \int_X u \nabla ^*F \, d\nu :\; F\in \widetilde{\mathcal{FC}}^1_b(X, R(X)), \; \|F\|_{L^{\infty}(X, \nu;X)} \leq 1\bigg\}
\end{equation}
If in addition $u\in L^p(X, \nu)$ for some $p>1$ we have 
\begin{equation}
\label{equiv21}
V_0(u) =  \sup\bigg\{ \int_X u \nabla ^*F \, d\nu :\; F\in \widetilde{W}^{1,p'}(X,\nu; R(X))\cap L^{\infty}(X, \nu;X), \; \|F\|_{L^{\infty}(X, \nu;X)} \leq 1\bigg\},
\end{equation}
and  for any orthonormal basis $\{e_k:\; k\in \N\}$ of $X$ contained in $R(X)$ we have
\begin{equation}
\label{equiv30}
V_0(u) = \sup\bigg\{ \int_X u \nabla^*F \, d\nu :\; F(x) = \sum_{i=1}^n \varphi_i(x) e_i, \; n\in \N, \; \varphi_i\in C^1_b(X),        \; \|F\|_{\infty}   \leq 1\bigg\}. 
\end{equation}
\end{Lemma}

As expected, a  result similar to
 \ref{Th:BV} holds.

\begin{Theorem}
\label{Th:BV0}
Let $u\in L^p(X, \nu)$ for some $p>1$. Then, $u\in BV_0(X, \nu)$ if and only if  $V_0(u) <+\infty$. 
In this case, the measure $m_0$ of Definition \ref{def:BV0} satisfies $|m_0|(X) = V_0(u)$. 
\end{Theorem}
\begin{proof}  The proof of the implication $u\in BV_0(X, \nu)\Rightarrow V_0(u)\leq |m_0|(X) $ is identical to the corresponding proof in Theorem  \ref{Th:BV}, and it is omitted. 

In fact, also the proof of the converse  follows the procedure of  Theorem  \ref{Th:BV}, with suitable modifications. 

Let   $\{h_k:\; k\in \N\}$ be an orthonormal basis of $X$  consisting of eigenvectors of $R$. 
As usual, we denote by $P_n$ the orthogonal projection on the subspace spanned by $h_1$, \ldots , $h_n$. 

For every $k\in \N$ set $z_k = R^{-1}h_k$. As we already remarked, if $V_0(u)<\infty$ then $V(u)<\infty$ and therefore $V_{z_k}(u)<\infty$ for every $k\in \N$.
By Theorem \ref{carBVz}, for every $k$ there exists a real measure $m_{z_k}$ such that 
$$ \int_X u(\langle R\nabla \varphi, z_k\rangle - v_{z_k} \varphi)d\nu = -\int_X \varphi\,dm_{z_k},  \quad \varphi \in C^1_b(X),  $$
namely
\begin{equation}
\label{corr326}
 \int_X u( \nabla \varphi, h_k\rangle - v_{z_k} \varphi)d\nu = -\int_X \varphi\,dm_{z_k},  \quad \varphi \in C^1_b(X). 
 \end{equation}
As in Theorem  \ref{Th:BV}, we construct  $m_0$ approximating it by the sequence of measures
$$ M_{0n}(B) =  \sum_{k=1}^{n} m_{z_k}(B) h_k =  \sum_{k=1}^{n} m_{R^{-1}h_k}(B) h_k , \quad B\in \mathcal B(X) . $$
By \eqref{vartot}, for every $n\in \N$ we have
$$\begin{array}{lll}
|M_{0n}|(X) &  = & \ds \sup   \left\{ \int_X \langle F, dM_{0n}\rangle :\; F\in C^1_b(X, X), \; \|F\|_{\infty} \leq 1 \right\} 
\\
\\
& =  & \ds \sup   \left\{ \int_X \langle F, dM_{0n}\rangle :\; F\in C^1_b(X, P_n(X)), \; \|F\|_{\infty} \leq 1 \right\} . 
\end{array}$$
Each vector field $F\in C^1_b(X, P_n(X))$ may be written as $F(x) = \sum_{k=1}^n f_k(x) h_k$, with $f_k\in C_b(X)$. Therefore 
it
belongs to $\widetilde{C}^1_b(X, R(X))$. By Theorem \ref{carBVz} we have 
$$\int_X \langle F, dM_{0n}\rangle = \int_X \sum_{k=1}^n f_k\,d m_{z_k} = -  \int_X u\sum_{k=1}^n (\langle \nabla f_k, h_k\rangle - v_{z_k}f_k) \,d\nu $$
so that, recalling \eqref{nabla_p^*}, 
$$\int_X \langle F, dM_{0n}\rangle =  \int_X  u \nabla^*F\,d\nu \leq V_0(u)\|F\|_{\infty}. $$
Therefore, $|M_{0n}|(X) \leq V_0(u)$ for every $n\in \N$. By lemma \ref{supmisure vett} the series $(M_{0n}(B))$ converges for every Borel set $B$, and  setting
$$ m_0(B) :=  \sum_{k=1}^{\infty } m_{R^{-1}h_k}(B) h_k, \quad B\in \mathcal B(X), $$
 $m_0\in \mathcal M(X,X)$ is such that $|m|(X) \leq V(u)$. By definition, we have $\langle m_0(B), h_k\rangle =  m_{R^{-1}h_k}(B)$ for every $k\in \N$ and therefore \eqref{corr326} yields  \eqref{BV0} if $y=h_k$ for some $k\in \N$. Since $z\mapsto v_z$ is linear, \eqref{BV0} holds if $y$ is any linear combination of the $h_k$. 
 
 If  $y$  is not  a linear combination of the $h_k$, we approach it by  $y_n:= P_ny = \sum_{k=1}^n \langle y,h_k\rangle h_k$. For each  $n\in \N$ and $\varphi \in C^1_b(X)$ we have  
\begin{equation}
\label{approxM_0n}
\int_X u( \langle   \nabla \varphi, y_n\rangle -  v_{R^{-1}y_n} \varphi)d\nu = - \int_X \varphi \,dm_{R^{-1}P_ny} =  - \int_X \varphi \sum_{k=1}^{n} \langle y, h_k\rangle  dm_{R^{-1}h_k}
 =  - \int_X \varphi   \langle y, M_{0n}(dx) \rangle . 
 \end{equation}
By  Lemma  \ref{supmisure vett} the sequence of measures $(  \langle y, M_{0n}(dx) \rangle  )$ converge weakly to  $ \langle y, m_0(dx) \rangle  $, so that the right hand side converges to $ - \int_X \varphi   \langle y, m_0(dx) \rangle  $ as $n\to \infty$. 

Concerning the left hand side,  $\langle   \nabla \varphi, y_n\rangle $ pointwise converges to $ \langle  \nabla \varphi, y\rangle$ and its modulus does not exceed  $\| \nabla \varphi\|_{\infty} \|y\|$. Therefore, 
$$\lim_{n\to \infty} \int_X u  \langle   \nabla \varphi, y_n\rangle  d\nu  =  \int_X u  \langle   \nabla \varphi, y\rangle  d\nu  . $$
Since each $h_k$ is an eigenvector of $R$, $P_n$ commutes with $R^{-1}$ on the range of $R$. Therefore, $R^{-1}y_n = R^{-1}P_ny = P_n(R^{-1}y)$  converges to $R^{-1}y$ as  $n\to \infty$. By \eqref{vz}, $\lim_{n\to\infty}   v_{R^{-1}y_n} =  v_{R^{-1}y} $ in $L^q(X, \nu)$ for every $q>1$, and we have 
$$\lim_{n\to \infty} \int_X u \, v_{R^{-1}y_n }\varphi\,d\nu   =  \int_X u \, v_{R^{-1}y }\varphi\,d\nu . $$
So, letting $n\to \infty $ in \eqref{approxM_0n} yields  \eqref{BV0}  for every $y\in R(X)$. 
\end{proof}

\begin{Remark} As in the proof of Theorem \ref{Th:BV}, the assumption $u\in L^p(X, \nu)$ for some $p>1$ is used only in the last step of the proof of Theorem \ref{Th:BV0}. Also the assumption 
that $R=R^*$ and  that there exists an orthonormal basis of $X$ consisting of eigenvectors of $R$ is crucially used only in the   last step 
of the proof. Up to that, we only needed that $R$ is one to one, and that there exists an orthonormal basis of $X$ contained in $R^*(X)$. 
\end{Remark}

 \subsection{A characterization through semigroups}
 \label{subs:semigroups}
 
 An elegant characterization of $BV$ functions $u$, that goes back to De Giorgi in the case of the Lebesgue measure in $\R^n$, exploits the behavior as $t\to 0$ of $T(t)u$, where $T(t)$ is a suitable semigroup of linear operators.  The key tool is the existence of a smoothing semigroup $T(t)$ of linear operators in $L^p(X, \nu)$, with good commutation properties with directional derivatives. 
 In finite dimension  the heat semigroup commutes with all directional derivatives, and this is the best situation. In finite and in infinite dimension, there is a very handy 
 commutation formula for the classical Ornstein-Uhlenbeck semigroup: we have $\partial/\partial h (T(t) \varphi) = e^{-t} T(t) \partial \varphi/\partial h$, for every $h$ in the Cameron-Martin space and $t>0$. In general, simple commutation formulae are not available, and the method works under some technical unelegant assumptions that, however,  are  satisfied in significant examples. See Section 5. 
 
 
 \begin{Proposition}
 \label{semigruppoA}
 Let $\{ T(t):\; t\geq 0\}$ be a strongly continuous semigroup of linear operators in $L^p(X, \nu)$ with $p>1$, such that 
 \begin{equation}
\exists q>1: \;\; T(t)(L^p(X,\nu) ) \subset W^{1,q}(X, \nu) \;\;\; \forall t>0. 
\label{regolarizz}
\end{equation} 
Then, 
\begin{itemize}
\item[(a)] for every $u\in L^p(X, \nu)$ and for every  $z\in X$   such that $R^*z\neq 0$ we have 
 \begin{equation}
 \label{implz} \liminf_{t\to 0^+} \int_X \left| \frac{\partial }{\partial R^*z}T(t)u\right| d\nu <+\infty 
  \Longleftrightarrow \liminf_{t\to 0^+} \int_X  |\partial^*_zT(t)u | d\nu <+\infty  \Longrightarrow  u\in BV_z(X, \nu), 
  \end{equation}
and in this case 
$$V_z(u) \leq  \liminf_{t\to 0^+} \int_X \left| \frac{\partial }{\partial R^*z}T(t)u\right| d\nu. $$
Moreover, 
 \begin{equation}
 \label{impl} \liminf_{t\to 0^+} \int_X \|M_qT(t)u\|\, d\nu <+\infty 
    \Longrightarrow  u\in BV(X, \nu), \quad V(u) \leq \liminf_{t\to 0^+} \int_X \|M_qT(t)u\|\, d\nu. 
  \end{equation}
\item[(b)] If in addition Hypothesis \ref{Hyp:R} holds and 
 \begin{equation}
\exists q>1: \;\; T(t)(L^p(X,\nu) ) \subset W^{1,q}_0(X, \nu) \;\;\; \forall t>0, 
\label{regolarizz0}
\end{equation} 
then  for every $u\in L^p(X, \nu)$ we have
 \begin{equation}
 \label{impl0} \liminf_{t\to 0^+} \int_X \|\nabla T(t)u\|\, d\nu <+\infty 
    \Longrightarrow  u\in BV_0(X, \nu), \quad V_0(u) \leq  \liminf_{t\to 0^+} \int_X \|\nabla T(t)u\|\, d\nu. 
  \end{equation}
\end{itemize}
\end{Proposition}
 \begin{proof}
(a) Fix $u\in L^p(X, \nu)$, and let $z\in X$ be such that $R^*z\neq 0$. Fix $M>0$, $\omega \in \R$ such that 
 $\|T(t) \|_{{\mathcal L}(L^p(X, \nu))} \leq M e^{\omega t}$ for every  $t>0$. Then,   
 $$\int_{X} | T(t)u(x) v_z(x)| \nu(dx)\leq M e^{\omega t} \|u\|_{L^p(X, \nu)} \|v_z\|_{L^{p'}(X, \nu)}, \quad t>0, $$
 so that, recalling \eqref{partial*z}, 
 $$ \liminf_{t\to 0^+} \int_X \bigg| \frac{\partial }{\partial R^*z}T(t)u\bigg| d\nu <+\infty 
  \Longleftrightarrow \liminf_{t\to 0^+} \int_X  |\partial^*_zT(t)u | d\nu <+\infty . $$
  
If one of the above equivalent conditions hold, it is easy to see that $u\in BV_z(X, \nu)$. 
Indeed, let $(t_n)\to 0$  be such that 
\begin{equation}
\label{Jz(u)}
\lim_{n\to\infty}  \int_X \bigg| \frac{\partial }{\partial R^*z}T(t_n)u\bigg| d\nu =: l \in \R. 
 \end{equation}
Since $L^p-\lim_{t\to 0} T(t)u= u$, for every $\varphi\in C^1_b(X)$ we have 
$$\int_X u \,\partial^*_z \varphi\,d\nu  = \lim_{n\to \infty} \int_X (T(t_n)u) \,\partial^*_z \varphi\,d\nu . $$
Since $T(t_n)u\in W^{1,q}(X, \nu)$ by assumption, the integration by parts formula \eqref{parti} yields 
$$\int_X  (T(t_n)u ) \,\partial^*_z \varphi\,d\nu =
-   \int_X \varphi \,\frac{\partial }{\partial R^*z}T(t_n)u\,d\nu, \quad n\in \N, $$
and therefore, by \eqref{Jz(u)}, 
$$\int_X u \,\partial^*_z \varphi\,d\nu  =  
-  \lim_{n\to \infty} \int_X \varphi \,\frac{\partial }{\partial R^*z}T(t_n)u\,d\nu \leq l\|\varphi\|_{\infty}. $$
Consequently, $u\in BV_z(X, \nu)$ and $V_z(u)\leq  \liminf_{t\to 0^+} \| \partial T(t)u/\partial R^*z \|_{L^1(X, \nu)}$. 

The proof of the other  statements are similar; let us prove  \eqref{impl}. Let $(t_n)\to 0$  be such that 
\begin{equation}
\label{J(u)}
\lim_{n\to\infty}   \int_X \|M_qT(t_n)u\|\, d\nu =: L \in \R . 
 \end{equation}
Since $L^p-\lim_{t\to 0} T(t)u= u$, for every $F\in \widetilde{C}^1_b(X, X)$ we have 
$$\int_X u \,M^*F\,d\nu  = \lim_{n\to \infty} \int_X (T(t_n)u) \,M^*F\,d\nu . $$
Since $M^*F = M^*_qF$ we get
$$\int_X u \,M^*F\,d\nu  =  \lim_{n\to \infty} \int_X\langle M_qT(t_n)u, F\rangle \,d\nu \leq L\|F\|_{\infty}. $$
Therefore, $u\in BV(X, \nu)$ and $V(u)\leq   \liminf_{t\to 0^+} \| M_qT(t)u \|_{L^1(X, \nu;X)}$. 
\end{proof}

In the next proposition we show that the converse implications in formulae \eqref{implz}, \eqref{impl}, \eqref{impl0} hold, under additional assumptions on the semigroup $T(t)$. 

 \begin{Proposition}
 \label{semigruppoB}
 Let $\{ T(t):\; t\geq 0\}$ be a strongly continuous semigroup of linear operators in $L^p(X, \nu)$ with $p>1$, satisfying \eqref{regolarizz}. 
 \begin{itemize}
 \item[(a)] 
 Let $z\in X$ be such that $R^*z\neq 0$ and such that 
 \begin{equation}
 \label{commutazione} \frac{\partial}{\partial R^*z}T(t)\varphi = S_1(t) \frac{\partial  \varphi }{\partial R^*z} + S_2(t) \varphi, \quad \varphi\in C^1_b(X), \; t>0, 
 \end{equation}
 where  $S_1(t) $, $S_2(t)\in  {\mathcal L } (L^p(X, \nu), L^1(X, \nu))$ for every $t>0$,  $S_1(t)^*(C^1_b(X)) \subset W^{1,p'}(X, \nu)\cap L^{\infty}(X, \nu)$, 
and for every $t>0$ there exist $c_1(t) $, $c_2(t) \geq 0$ such that
 \begin{equation}
 \label{regolarizzT*}
\left\{  \begin{array}{lll}
(i) &  \|  S_1(t)^*\varphi \|_{ \infty} \leq c_1(t)\|\varphi\|_{\infty}  , \quad   \;\varphi\in C^1_b(X), 
 \\
 \\
(ii)&   \|S_2(t)\varphi \|_{L^1(X, \nu)} \leq c_2(t)\|\varphi\|_{L^p(X, \nu)}, \quad    \;\varphi\in C^1_b(X). 
 \end{array} \right. 
 \end{equation}

Then for every $u\in L^p(X, \nu)$ and $t>0$ we have
 \begin{equation}
 \label{implicazione}
  u\in BV_z(X, \nu) \Longrightarrow   \int_X \left| \frac{\partial }{\partial R^*z}T(t)u\right| d\nu  \leq V_z(u)c_1(t) + c_2(t) <+\infty. 
  \end{equation}
\item[(b)]
Assume that 
 \begin{equation}
 \label{commutazionevett} M_qT(t)\varphi = {\bf S}_1(t)M \varphi  +  {\bf S}_2(t) \varphi, \quad \varphi\in C^1_b(X), \; t>0, 
 \end{equation}
 where for every $t>0$,  $ {\bf S}_1(t)\in  {\mathcal L } (L^p(X, \nu;X), L^1(X, \nu;X)) $,    $ {\bf S}_1(t)^*$ maps $\widetilde{C}^1_b(X;X))$ 
into  $\widetilde{W}^{1,p'}(X, \nu;X)$ $ \cap$  $L^{\infty}(X, \nu;X)$, 
and  there exist $C_1(t) $, $C_2(t) \geq 0$ such that  
 \begin{equation}
 \label{regolarizzT*vett}
\left\{  \begin{array}{lll}
(i) &  \|  {\bf S}_1(t)^*F \|_{ \infty} \leq C_1(t)\|F\|_{\infty}   , \quad  F\in \widetilde{C}^1_b(X, X), 
 \\
 \\
(ii) &   \| {\bf S}_2(t)\varphi \|_{L^1(X, \nu;X)} \leq C_2(t) \|\varphi \|_{  L^p(X, \nu )}, \quad  \varphi \in \widetilde{C}^1_b(X). 
 \end{array} \right. 
 \end{equation}
Then for every $u\in L^p(X, \nu)$ and $t>0$ we have
 \begin{equation}
 \label{implicazione1}
 u\in BV(X, \nu) \Longrightarrow   \int_X\|M_qT(t)u\|\,d\nu\leq V(u)C_1(t) + C_2(t) <+\infty . 
 \end{equation}
\item[(c)]
If Hypothesis \ref{Hyp:R} holds, and  in addition $T(t)(L^p(X, \nu))\subset W^{1,q}_0(X, \nu)$ for every $t>0$ and 
 \begin{equation}
 \label{commutazionevett0} \nabla_qT(t)\varphi = {\bf S}_1(t)\nabla  \varphi  +  {\bf S}_2(t) \varphi, \quad \varphi\in C^1_b(X), \; t>0, 
 \end{equation}
 where   for every $t>0$,  $ {\bf S}_1(t) \in  {\mathcal L } (L^p(X, \nu;X), L^1(X, \nu;X))$,  $S_1(t)^*$ maps $\widetilde{C}^1_b(X;X))$ into $ \widetilde{W}^{1,p'}_0(X, \nu;X)$ $\cap$ $ L^{\infty}(X, \nu;X)$, 
and there exist $C_1(t) $, $C_2(t) >0$ such that \eqref{regolarizzT*vett} holds, 
then for every $u\in L^p(X, \nu)$ and $t>0$ we have
 \begin{equation}
 \label{implicazione2}
 u\in BV_0(X, \nu) \Longrightarrow  \int_X\|\nabla_qT(t)u\|\,d\nu \leq V_0(u)C_1(t) + C_2(t) <+\infty . 
 \end{equation}
 \end{itemize}
 \end{Proposition}
 \begin{proof}
Assume   that \eqref{commutazione},  \eqref{regolarizzT*}   hold. By \eqref{regolarizzT*}(ii), $S_2(t)$ has an extension (still called $S_2(t)$) belonging to ${\mathcal L}(L^p(X, \nu), L^1(X, \nu))$.  Fix $u\in L^{p}(X, \nu)$ such that $V_z(u)<+\infty$. 
By Remark \ref{Rem:normaL^1},  for every $t>0$ we have
$$ \int_X  \bigg| \frac{\partial }{\partial R^*z}T(t)u\bigg| d\nu = \sup \bigg\{ \int_X   \varphi \,\frac{\partial }{\partial R^*z}T(t)u \, d\nu : \; \varphi\in C^1_b(X), \, \|\varphi\|_{\infty} \leq 1 \bigg\}. $$
For   $\varphi\in C^1_b(X)$  and for $t>0$ we have, still by formula   \eqref{parti}, 
$$
\begin{array}{l}
\ds  \int_X  \varphi \, \frac{\partial }{\partial R^*z}T(t)u\, d\nu =  \int_X u \left(  \frac{\partial }{\partial R^*z}T(t)\right)^*\varphi\,d\nu 
\\
\\
\ds =   \int_X u \left(  S_1(t)  \frac{\partial }{\partial R^*z} + S_2(t)\right)^*\varphi\,d\nu =    \int_X u  (  \partial^*_z  S_1(t)^*\varphi + S_2(t)^*\varphi) \,d\nu 
\\
\\
\leq V_z(u) \|S_1(t)^*\varphi \|_{\infty}  + \| S_2(t)^*\|_{{\mathcal L}(L^{\infty}(X, \nu), L^{p'}(X, \nu)} \|\varphi \|_{\infty} 
\\
\\
\leq (V_z(u) c_1(t) + c_2(t)) \|\varphi \|_{\infty}, 
\end{array}$$
and \eqref{implicazione} follows. 

 The proofs of \eqref{implicazione1}, \eqref{implicazione2} are similar; we prove  \eqref{implicazione1}. Fix $u\in L^{p}(X, \nu)$ such that $V(u)<+\infty$. 
By Remark \ref{Rem:normaL^1},  for every $t>0$ we have
$$ \int_X  \|M_qT(t)u\| \, d\nu = \sup \bigg\{ \int_X \langle M_qT(t)u, F\rangle \, d\nu : \; F\in \widetilde{C}^1_b(X, X), \, \|F\|_{\infty} \leq 1 \bigg\}. $$
\eqref{regolarizzT*vett} yields that ${\bf S}_2(t)$ has an extension (still called ${\bf S}_2(t)$) belonging to ${\mathcal L}(L^1(X, \nu), L^p(X, \nu;X))$, with norm $\leq C_2(t)$. 
For $ F\in \widetilde{C}^1_b(X, X)$   and for $t>0$ we have 
$$\begin{array}{l}
\ds  \int_X  \langle M_qT(t)u, F\rangle \, d\nu =  \int_X u ( M_qT(t) )^*F\,d\nu 
\\
\\
\ds =   \int_X u (  {\bf S}_1(t) M_q + {\bf S}_2(t) )^*F\,d\nu =    \int_X u  ( M_q^* {\bf S}_1(t)^* F +  {\bf S}_2(t)^*F) \,d\nu 
\\
\\
\leq V(u) \| {\bf S}_1(t)^*F \|_{\infty}  + \|  {\bf S}_2(t)^*\|_{{\mathcal L}(L^{\infty}(X, \nu;X), L^{p'}(X, \nu;X))} \|F\|_{\infty} 
\\
\\
\leq (V(u) C_1(t) + C_2(t)) \|F\|_{\infty} 
\end{array}$$
and \eqref{implicazione1} follows. 
 \end{proof}

Combining  the statements of Propositions \ref{semigruppoA} and  \ref{semigruppoB} we obtain the next corollary.

\begin{Corollary}
\label{uguali}
If the assumptions of  Proposition \ref{semigruppoB} hold and  the functions $C_1$, $C_2$ in  \eqref{regolarizzT*vett}(ii) are bounded near $t=0$, for every $u\in L^p(X, \nu)$ we have
$$\begin{array}{l} 
\ds u\in BV(X, \nu) \Longleftrightarrow  \liminf_{t\to 0} \int_X\|M_qT(t)u\|\,d\nu  <+\infty ,
\\
\\
\ds  u\in BV_0(X, \nu) \Longleftrightarrow  \liminf_{t\to 0} \int_X\|\nabla_qT(t)u\|\,d\nu  <+\infty. \end{array}$$
If in addition  
$\lim_{t\to 0} C_1(t) =1$, $\lim_{t\to 0} C_2(t) =0$, we respectively get 
$$V(u) = \liminf_{t\to 0}  \int_X  \| M_qT(t)u\| \, d\nu , \quad V_0(u) = \liminf_{t\to 0}  \int_X  \| \nabla_qT(t)u\| \, d\nu. $$
\end{Corollary}

 As a corollary of Proposition \ref{semigruppoB}, we have a further characterization of $BV$ and $BV_0$ functions. 
 
\begin{Corollary}
Let $u\in L^p(X, \nu)$ be such that  
\begin{equation}
\label{carattun}
\exists u_n\in W^{1,p}(X, \nu): \; \lim_{n\to\infty}\|u_n-u\|_{L^p(X, \nu)} =0, \; \sup_{n\in\N} \int_X \|M_pu_n\| \,d\nu <+\infty . 
\end{equation}
Then, $u\in BV(X, \nu)$. Similarly, if Hypothesis \ref{Hyp:R}  holds, and  $u\in L^p(X, \nu)$ is such that  
\begin{equation}
\label{carattun0}
\exists u_n\in W^{1,p}_0(X, \nu): \; \lim_{n\to\infty}\|u_n-u\|_{L^p(X, \nu)} =0, \; \sup_{n\in\N} \int_X \|\nabla_pu_n\| \,d\nu <+\infty , 
\end{equation}
then $u\in BV_0(X, \nu)$. 

Conversely, if the assumptions of  Proposition \ref{semigruppoB}(b) hold,  every $u\in L^p(X, \nu)\cap BV(X, \nu)$ satisfies \eqref{carattun}. If the assumptions of  Proposition \ref{semigruppoB}(c) hold,  every $u\in L^p(X, \nu)\cap BV_0(X, \nu)$ satisfies \eqref{carattun0}. 
\end{Corollary} 
 \begin{proof} Let \eqref{carattun} hold. For every $F\in \widetilde{C}^1_b(X, X)$ we have
 $$\int_X u\,M^*F\,d\nu = \lim_{n\to\infty} \int_X  u_nM^*F\,d\nu =  \lim_{n\to\infty} \int_X \langle M_pu_n, F\rangle d\nu, $$
 where $| \int_X \langle M_pu_n, F\rangle d\nu | \leq  \|M_pu_n\|_{L^1(X, \nu;X)} \|F\|_{\infty}$, for every $n\in \N$. Therefore, 
 $V(u)\leq \liminf_{n\to\infty} $ $ \|M_pu_n\|_{L^1(X, \nu;X)} <+\infty$, and $u\in BV(X, \nu)$. 
 
 The converse is an immediate consequence of  Proposition \ref{semigruppoB}(b),  taking  $u_n:=T(1/n)u$. 
 
 The  statements with $W^{1,p}_0(X, \nu)$, $BV_0(X, \nu)$ replacing $W^{1,p}(X, \nu)$, $BV(X, \nu)$ are proved in the same way. 
 \end{proof}

 
\section{Sets with finite perimeter}

In this section we consider  sublevel sets of suitable Sobolev functions $g:X\mapsto \R$, 
and we give sufficient conditions for their characteristic functions belong to $BV(X, \nu)$ or to $BV_0(X, \nu)$. 
 
\begin{Definition}
If the characteristic function $u:= \one_A$ of a Borel set $A$ belongs to $BV(X, \nu)$,  the measure $|m|$ in Definition \ref{def:BV}
 is called {\em perimeter measure}, and  $p(A):=|m|(X)$ is called the {\em perimeter} of $A$. 
 
 If Hypothesis \ref{Hyp:R} holds and $ \one_A$  belongs to $BV_0(X, \nu)$, we set $p_0(A) := |m_0|(X)$, where $m_0$ is the measure  in Definition \ref{def:BV0}. 
\end{Definition}

The simplest examples of sets with finite perimeter are halfspaces. For every $a\in X\setminus \{0\} $ and $r\in \R$, we set 
$$H_{a,r}:= \{ x\in X: \; \langle x, a\rangle <r\}. $$

\begin{Proposition}
\label{halfspaces}
For every $a\in X\setminus \{0\} $ and $r\in \R$ the characteristic function of $H_{a,r}$ belongs to $BV(X, \nu)$, and $p(H_{a,r})=0$ if $Ra =0$, 
$$p(H_{a,r}) = \left\{
\begin{array}{cl}
 0 & \mbox{\rm if}\; \;Ra =0, 
\\
\\
\ds  -\frac{1}{\|Ra\|}\int_{H_{a,r}} v_{Ra}\,d\nu  & \mbox{\rm if}\; \; Ra\neq 0. 
\end{array}\right.$$
If Hypothesis  \ref{Hyp:R} holds, the characteristic function of $H_{a,r}$ belongs also  to $BV_0(X, \nu)$, and
$$p_0(H_{a,r})= -\frac{1}{\|a\|}\int_{H_{a,r}} v_{a}\,d\nu . $$
\end{Proposition}
\begin{proof}
We approximate $\one_{H_{a,r}}$ by Sobolev functions,  introducing  the functions $\theta_{\eps}:\R\mapsto \R$, defined  for $\eps >0$ by 
\begin{equation}
\label{e21n}
\theta_\varepsilon(\xi)=\left\{\begin{array}{l}
1,\quad\mbox{\rm if}\;\xi\le r-\varepsilon,\\
\\
\ds -\frac{1}{\varepsilon}( \xi -r) ,\quad\mbox{\rm if}\; r-\varepsilon< \xi< r,\\
\\
0,\quad\mbox{\rm if}\;\xi\ge r .
\end{array}\right.
\end{equation}
If $\int_X\|x\|^2\, \nu(dx) <+\infty$, the function $x\mapsto \langle x, a\rangle$ belongs to  $W^{1,2}(X, \nu)$ by  \cite[Lemma 2.4]{TAMS}, and 
the composition $g_{\eps}(x):= \theta_\varepsilon(\langle x, a\rangle)$ belongs to $W^{1,2}(X, \nu)$ by  \cite[Lemma 2.2]{TAMS}. If the second moment of $\nu$ is infinite, 
$x\mapsto \langle x, a\rangle$ could not belong to $L^2(X, \nu)$, however since it belongs to $C^1(X)$ the proof of Lemma 2.2 of \cite{TAMS} still works. In any case,  $g_{\eps}\in W^{1,2}(X, \nu)$ and 
$$
M_2g_{\eps} = -\frac{1}{\eps} \one_{\{x:\, r-\eps < \langle x, a\rangle <r\}} Ra . 
$$
Moreover,  $g_{\eps} \to \one_{H_{a, r}}$ a.e. as $\eps\to 0^+$, and  for every $F\in \widetilde{C}^1_b(X;X)$ the Dominated Convergence Theorem yields
$$ \int_{H_{a, r}} M^*F \, d\nu =\lim_{\eps \to 0} \int_{X} g_{\eps}\,  M^*F \, d\nu  . $$
On the other hand, for every $\eps >0$ we have
$$  \int_{X} g_{\eps}\,  M^*F \, d\nu  =   \int_{X} \langle M_2g_{\eps} , F \rangle \, d\nu  = - \frac{1}{\eps} \int_{\{x:\, r-\eps < \langle x, a\rangle <r\}} 
 \langle Ra , F \rangle \, d\nu $$
 so that  $\int_{H_{a, r}} M^*F \, d\nu =0$ if $Ra=0$, and in this case $V( \one_{H_{a, r}}) =0$. If instead $Ra\neq 0$ we get 
$$\begin{array}{l}
\ds  \left|  \int_{X} g_{\eps}  M^*F \, d\nu  \right| \leq \frac{\|F\|_{\infty}}{\|Ra\|} \; \frac{1}{\eps}  \int_{\{x:\, r-\eps < \langle x, a\rangle <r\}} \|Ra\|^2\,d\nu
\\
\\
\ds = - \frac{\|F\|_{\infty}}{\|Ra\|} \int_X \langle M_2g_{\eps} , Ra\rangle \,d\nu = -  \frac{\|F\|_{\infty}}{\|Ra\|} \int_X g_{\eps}v_{Ra} d\nu. 
\end{array}$$
Letting $\eps\to 0$ and using again the Dominated Convergence Theorem we obtain 
$$ \left|  \int_{H_{a, r}}  M^*F \, d\nu  \right| \leq  -  \frac{\|F\|_{\infty}}{\|Ra\|} \int_{H_{a, r}} v_{Ra} \,d\nu, $$
so that $ \one_{H_{a, r}}\in BV(X, \nu)$ and $V( \one_{H_{a, r}})\leq  - 1/\|Ra\| \int_{H_{a, r}} v_{Ra} \, d\nu$. To prove the opposite inequality we choose the constant vector field $F_0(x):= -Ra/\|Ra\|$. We have $M^*F_0 = - v_{Ra}/\|Ra\|$, so that 
$$ \int_{H_{a, r}} M^*F_0 \, d\nu = - \frac{1}{ \|Ra\|} \int_{H_{a, r}} v_{Ra} \, d\nu , $$
which implies $V( \one_{H_{a, r}})\geq  - 1/\|Ra\| \int_{H_{a, r}} v_{Ra}\,  d\nu$, and the first statement is proved. 

The proof of the second statement is the same, with obvious modifications: it is sufficient to replace $M_2g_{\eps} $ by $\nabla_2 g_{\eps} =  -\frac{1}{\eps} \one_{\{x:\, r-\eps < \langle x, a\rangle <r\}} a$, and the vector field $F_0$ by the constant vector field $a/\|a\|$. 
\end{proof}

If halfspaces are replaced by sublevel sets of a good function $g$, things are not so easy. The following result holds.

\begin{Proposition}
\label{Pr:almost}
\begin{itemize} 
\item[(a)] Let $g\in W^{1,1}(X, \nu)$. Then $\one_{g^{-1}(-\infty, r)} \in BV(X, \nu)$ for a.e. $r\in \R$. 
 \item[(b)] Let $g\in W^{1,1}_0(X, \nu)$. Then $\one_{g^{-1}(-\infty, r)} \in BV_0(X, \nu)$ for a.e. $r\in \R$. 
\end{itemize}
\end{Proposition}
\begin{proof}
Let us consider the function 
$$\mu (r):= \int_{\{x:\,  g(x)<r\}}   \|M_pg\|\, d\nu , \quad r\in \R , $$
which is increasing, and thus left differentiable at a.e. $r\in \R$. 

Let us fix $r$ such that $\mu $ is left differentiable at $r$. 
For such $r$, we approximate $\one_{\{x:\, g(x)<r\}}$ as in the previous proposition,  using  the functions $\theta_{\eps}:\R\mapsto \R$ defined in \eqref{e21n}. 
By  \cite[Lemma 2.2]{TAMS}, the composition $\theta_\varepsilon \circ g$ belongs to $W^{1,1}(X, \nu)$, and 
$$
M_1(\theta_\varepsilon \circ g) = -\frac{1}{\eps} \one_{\{x:\, r-\eps \leq  g(x)\leq r\}} M_1(g) =  -\frac{1}{\eps} \one_{\{x:\, r-\eps \leq  g(x) < r\}} M_1(g). 
$$
The second equality is a consequence of the fact that $M_1g =0$ a.e in the set $g^{-1}(r)$, by Corollary 2.3 of \cite{TAMS} (we remark that Lemma 2.2 and Corollary 2.3 of \cite{TAMS} were stated for $p>1$ but their proofs works as well for $p=1$). 

Moreover $\theta_\varepsilon \circ g \to \one_{\{x:\, g(x)<r\}}$ a.e. as $\eps\to 0^+$. For every $F\in \widetilde{C}^1_b(X;X)$ with $\|F\|_{\infty} \leq 1$  the Dominated Convergence Theorem yields
\begin{equation}
\label{1}
 \int_{\{x:\,  g(x)<r\}} M^*F\, d\nu = \lim_{\eps\to 0^+} \int_X (\theta_\varepsilon \circ g)  M^*F\, d\nu. 
 \end{equation}
On the other hand, for every $\eps>0$ we have 
$$\begin{array}{l}
\ds \int_X (\theta_\varepsilon \circ g)  M^*F\, d\nu = -\int_X \langle M_1 (\theta_\varepsilon \circ g) , F\rangle d\nu
\\
\\
\ds = \frac{1}{\eps} \int_{\{x:\, r-\eps \leq  g(x) < r\}} \langle M_1g, F\rangle d\nu  \leq \frac{1}{\eps} \int_{\{x:\, r-\eps \leq  g(x) < r\}} \|M_1g\|\, d\nu. 
 \end{array} $$
Recalling \eqref{1}, we get  
$$ \begin{array}{l}
\ds
\int_{\{x:\,  g(x)<r\}} M^*F\, d\nu  \leq  \lim_{\eps\to 0^+}  \frac{1}{\eps} \int_{\{x:\, r-\eps \leq  g(x)<r\}} \|M_1g\|\, d\nu 
\\
\\
\ds =  \lim_{\eps\to 0^+}  \frac{\mu(r) - \mu(r-\eps)}{ \eps}  = \mu_-'(r). \end{array}$$
Therefore, $V(\one_{g^{-1}(-\infty, r)}) \leq \mu_-'(r)<+\infty $, and Statement (a) follows from Theorem \ref{Th:BV}.

The proof of Statement (b) is the same. \end{proof}

In \cite{TAMS} we considered a general class of functions $g$ whose sublevel sets turn out to have finite perimeter. 
More precisely, we assumed that 
\begin{equation}
\label{g}
g\in \bigcap_{p>1}W^{1,p}(X, \nu), \quad \frac{M_pg}{\|M_pg\|^2} \in D(M^*_p) \;\; \text{for every} \;p>1. 
\end{equation}
%
We constructed a family of real nonnegative measures $\sigma^g_r$, enjoying the following property: if  $\langle Mg, z \rangle \in C_b(X) \cup_{q>1} W^{1,q}(X, \nu)$ for some $z\in X$,  then
\begin{equation}
\label{partiTAMS}
\int_{\{g<r\}} (\langle R\nabla \varphi, z\rangle - v_z \varphi) d\nu = \int_X  T(\langle Mg, z \rangle) \varphi \,d\sigma^g_r, \quad \varphi\in C_b^1(X). 
\end{equation}
Here $T$ is the trace operator, which is bounded from $W^{1,q}(X, \nu)$ to $L^1(X, \sigma^g_r)$ for every $q>1$ and $r\in \R$. 
It is defined as follows. The starting point is the estimate, proved in \cite{TAMS}, 
\begin{equation}
\label{tracciaTAMS}
\int_X |\varphi|\, d\sigma^g_r \leq K_q \|\varphi\|_{W^{1,q}(X, \nu)}, \quad \varphi\in C^1_b(X). 
\end{equation}
So, fixed $f\in W^{1,q}(X, \nu)$,  all the sequences $(\varphi_n)$   of $C^1_b(X)$ functions that converge to $f$ in $W^{1,q}(X, \nu)$ converge also  in  $L^1(X, \sigma^g_r)$ to a common  limit, denoted by $T(f)$.  

Formula \eqref{partiTAMS} is just  formula (4.9) of \cite{TAMS}. 

The vector valued trace operator, ${\bf T}: W^{1,q}(X, \nu;X)\mapsto L^1(X, \sigma^g_r;X)$, is defined in an obvious way as $ {\bf T}(F)= \sum_{k=1}^{\infty}
T(f_k)e_k$ for any $F= \sum_{k=1}^{\infty}f_ke_k \in W^{1,q}(X, \nu;X)$. It follows easily by the definition that  $\langle {\bf T}(F), z\rangle = T(\langle F, z\rangle )$, for every $F\in W^{1,q}(X, \nu;X)$ and $z\in X$ ($X$-valued Sobolev spaces are defined in an obvious way, see \cite[Sect. 5]{TAMS}).

\begin{Proposition}
\label{surface}
Let \eqref{g} holds. Then
\begin{itemize}
\item[(i)] If $z\in X$ is such that $\langle Mg, z \rangle \in C_b(X) \cup_{q>1} W^{1,q}(X, \nu)$
then for each $r\in \R$ the characteristic function $\one_{\{x:\, g(x) < r\}}$ belongs to $BV_z(X, \nu)$, and
\begin{equation}
\label{surface1}
 m_z(dx)  = T( \langle Mg, z  \rangle) \sigma^g_r(dx). 
 \end{equation}
\item[(ii)] If  in addition $g\in W^{2,q}(X, \nu)$ for some $q>1$, then $\one_{\{x:\, g(x) < r\}}$ belongs to $\ BV(X, \nu)$ and the vector measure $m(dx) = \sigma(x)|m|(dx)$ of Definition \ref{def:BV} is given by 
\begin{equation}
\label{surface2} \sigma (x) := {\bf T}\bigg(\frac{Mg}{\|Mg\|}\bigg), \quad |m|(dx)  := T(\|Mg\|) \sigma^g_r(dx)
 \end{equation}
\end{itemize}
\end{Proposition}
\begin{proof}
Statement (i) is an easy consequence of \eqref{partiTAMS}. Indeed,  \eqref{partiTAMS} yields \eqref{BV},  with $ m_z$ given by \eqref{surface1}.

If $g\in W^{2,q}(X, \nu)$ for some $q>1$, for every $z\in X$ the function $\langle Mg, z\rangle $ belongs to $W^{1,q}(X, \nu)$, and the measure $m_z$ in \eqref{surface1} is well defined. 
Moreover,  the function $\|Mg\| $ belongs to $W^{1,q}(X, \nu)$ so that its trace is well defined, the vector field $Mg$ belongs to $W^{1,q}(X, \nu;X)$ and since $1/\|Mg\|\in L^p(X, \nu)$ for every $p$, the quotient 
$F:= Mg/\|Mg\| $ belongs to $W^{1,s}(X, \nu;X)$ for every $s<q$ (see the proof of Theorem 5.3 of \cite{TAMS}). 
So, its vector valued  trace ${\bf T}(F) $ is well defined. Let us prove that  it has unit norm, $ \sigma^g_r$-a.e. Setting $F = \sum_{k=1}^{\infty} 
f_k e_k$, with $f_k(x) = \langle Mg(x), e_k\rangle/ \|Mg\|$, for every $k\in \N$ and $1<s<q$ the function $f_k$ belongs to $ L^{\infty}(X, \nu) \cap W^{1,s}(X, \nu) $. By Lemma \ref{Le:Sobolev}(ii), $f_k^2\in W^{1,s}(X, \nu) $,  its trace is well defined and equal to $(T(f_k))^2$ as an element of $L^1(X, \sigma^g_r)$. So, for $\sigma^g_r$-a.e $x$ we have 
$$\|{\bf T}(F)(x)\|^2 = \bigg\| \sum_{k=1}^{\infty} T(f_k)(x)e_k\bigg\|^2 =  \sum_{k=1}^{\infty}( T(f_k)(x))^2  =  \sum_{k=1}^{\infty}T(f_k^2)(x) = T\bigg( \sum_{k=1}^{\infty}f_k^2\bigg)
= T(\one )(x) = 1, $$
where $\one $ is the constant function, $\one (x) =1$ for every $x$. 

For every $z\in X$ we have 
$\langle {\bf T} ( Mg/\|Mg\|), z\rangle = T(\langle Mg, z\rangle / \|Mg\|) = T(\langle Mg, z\rangle )/T( \|Mg\|)$. Therefore, if $m(dx) = \sigma(x)|m|(dx)$ with  $\sigma$ and $|m|$  defined 
in \eqref{surface2}, for every Borel set $A\subset X$  we have
$$\langle m(A), z\rangle = \frac{T(\langle Mg, z\rangle )}{T( \|Mg\|)} \,  T(\|Mg\|) \sigma^g_r(A) =  m_z(A), $$
and statement (ii) follows. 
\end{proof}
 
Proposition \ref{surface} yields that the measure $ T(\|Mg\|) \sigma^g_r(dx) $ is the perimeter measure of the set $\Omega = g^{-1}(-\infty, r)$. In particular, it does not depend on the defining function $g$ but only on the sublevel set $\Omega$: if $g_1$, $g_2\in W^{2,q}(X, \nu)$ for some $q>1$ satisfy assumption \eqref{g} and $g_1^{-1}(-\infty, r_1)= g_2^{-1}(-\infty, r_2)$ for some $r_1$, $r_2\in \R$, then $ T(\|Mg_1\|) \sigma^{g_{1}}_{r_1}(A) = T(\|Mg_2\|) \sigma^{g_{2}}_{r_2}(A) $, for every Borel set $A$. This was shown in \cite{TAMS} under the additional assumption $q>2$. 

We already know that halfspaces have finite perimeter, by Proposition \ref{halfspaces}. Fixed any $a\neq 0$, the function $g(x) := \langle x, a\rangle$ satisfies the assumptions of Proposition \ref{surface}(ii)  provided that $g\in L^p(X, \nu)$ for every $p\in (1, +\infty)$. This happens for every $a$  if $\nu$ has finite moments of any order. In this case Proposition \ref{surface} gives a representation of the measure $m$ of Definition \ref{def:BV}.  

\section{Examples}

\subsection{Gaussian and weighted Gaussian measures}
\label{Gaussian}

Let $\gamma$ be a nondegenerate Gaussian measure in $X$, with mean $0$ and covariance $Q$. The choice $R=Q^{1/2}$ gives us the usual setting of the Malliavin calculus. 
Indeed, the Cameron-Martin space $H$,  consisting of all the elements $h\in X$ along which $\gamma$ is Fomin differentiable,  is just the range of $Q^{1/2}$, and Hypothesis \ref{h1'} is satisfied, see e.g. \cite[Ch. 2]{Boga}. For every $z\in X$, the function $v_z(x)$ is what is called $\hat{h}$ in \cite{Boga}, with $h= Q^{1/2}z$, and what is called $W_z$ in \cite{DPZbrutto}. If $\{e_k:\; k\in \N\}$ is any orthonormal basis of $X$ consisting of eigenvectors of $Q$, $Qe_k = \lambda_k e_k$, the functions $v_z$ have the nice representation formula
\begin{equation}
\label{vzgaussiano}
v_z(x) = \sum_{k=1}^{\infty}\frac{\langle x, e_k\rangle \, \langle z, e_k\rangle}{ \lambda_{k}^{1/2} }, 
\end{equation}
where the series converges in $L^p(X, \gamma)$ for every $p\in [1, +\infty)$. 

Comparing with  the notation of  \cite[Ch. 5]{Boga}, the operator $M_p$ used here coincides with the realization of $Q^{-1/2}\nabla_H$ in $L^p(X, \gamma)$, and our Sobolev spaces $W^{1,p}(X, \gamma)$  coincide with the classical spaces ${\mathbb D}^{1,p}(X, \gamma)$ of the Malliavin calculus; moreover $M_p^*F $ is equal to minus the Gaussian divergence of $Q^{1/2}F$, 
for every $F\in D(M_p^*)$. The (easy) proofs of these statements  may be found in \cite[Sect. 6]{TAMS}. 

Therefore, our notion of $BV$ functions coincide with the one already considered in  \cite{F,FH,AMMP}. In such papers everything is referred to the Cameron-Martin space $H=Q^{1/2}(X)$. The vector measure $Du$ of  \cite{F,FH}, called  $D_{\gamma}u$ in  \cite{ AMMP}, coincides with $Q^{1/2}m$, where $m$ is our $X$-valued vector measure from Definition \ref{def:BV}.  
 
The Orlicz space $Y:=L\,($log$L)^{1/2}(X, \gamma)$, which is properly contained in $L^1(X, \gamma)$ and properly contains $L^p(X, \gamma)$ for every $p>1$, is of particular relevance, because all the functions $v_z(x) $ belong to the dual  space $Y'$ (more precisely, for every  $u\in Y$ the product $uv_z$ belongs to 
$L^1(X, \gamma)$ and $u\mapsto \int_X u\,v_z\,d\gamma $ belongs to $Y'$), 
 and moreover if $z_n\to z$, formula \eqref{vzgaussiano} yields that $v_{z_n} \to v_z$ in $Y'$. See \cite[Sect. 3]{FH}. 
 Therefore, the proofs of Theorems \ref{Th:BV}, \ref{Th:BV0} and of the formulae \eqref{equiv1+}, \eqref{equiv2+}, \eqref{equiv3}, \eqref{equiv21}, \eqref{equiv30} work as well, taking $u\in Y$ instead of $u\in L^p(X, \gamma)$ for some $p>1$. 

Proposition \ref{semigruppoB} is particularly simple in this case if we take as $T(t)$ the classical Ornstein-Uhlenbeck semigroup, 
$$T(t)f(x):= \int_X f(e^{-t}x + \sqrt{1-e^{-2t}}y)\,\gamma (dy) . $$
We refer to \cite[Ch. 2, Ch. 5]{Boga} for the properties of $T(t)$. In particular, we recall that 
for every $h=Rz \in H$ and $\varphi\in C^1_b(X)$ we have 
$\partial /\partial h\, T(t)\varphi = e^{-t} T(t) \partial \varphi /\partial h$, so that we can take ${\bf S}_2(t) =0$,   ${\bf S}_1(t) =  e^{-t}{\bf T}(t)$, 
where  ${\bf T}(t)F(x) := \int_X F(e^{-t}x + \sqrt{1-e^{-2t}}y)\,\gamma (dy)$  (vector valued integral), 
and all the assumptions of Proposition \ref{semigruppoB}(b) are immediately satisfied, since $T(t)$ maps $L^p(X, \gamma)$ to $W^{1,p}(X, \gamma)$ for every $p>1$. 
Therefore, Propositions  \ref{semigruppoA}  and  \ref{semigruppoB}(b) yield, for every $u\in L^p(X, \nu)$ with $p>1$, 
$$ u\in BV(X, \nu)  \Longleftrightarrow  \liminf_{t\to 0} \int_X \|M_pT(t)u\|\,d\gamma <+\infty,  $$
and in this case Remark \ref{uguali} yields $V(u) =  \liminf_{t\to 0}  \|M_pT(t)u\|_{L^1(X, \gamma; X)}$. 
The  commutation formula between $T(t)$ and the directional derivatives yields $M_pT(t+s)u = M_pT(t)(T(s)u) = e^{-t}{\bf T}(t)(M_p T(s)u)$, and since ${\bf T}(t)$ is a contraction semigroup in $L^1(X, \nu;X)$, the 
function $t\mapsto  \int_X \|M_pT(t)u\|\,d\gamma  $ is decreasing. Therefore the  above  $\liminf$ is in fact a limit.  

So, the results of Section \ref{sectionBV} are the same of \cite{FH,AMMP}, but the proofs are  different since in  \cite{FH,AMMP} specific features of Gaussian measures were used. 
\vspace{3mm}

Since in this case  $R=R^*$ is a compact one to one operator,  there exists an orthonormal basis of $X$ consisting of eigenvectors of $R$ (namely, eigenvectors of $Q$). The contents of \S \ref{subs:other} 
fits the setting of \cite{DPZbrutto} as far as Sobolev spaces are concerned. Our  definition of $BV_0(X, \gamma)$ is equivalent to the one of \cite{AmDaPa10} and  to the one of \cite{RZZ1} with the choice $H_1=X$. 

Now it is convenient to  choose as $T(t)$ the Ornstein-Uhlenbeck semigroup given by 
 $$T(t)f (x) := \int_X f(y) \,{\mathcal N}_{\;t,x}(dy), $$
 where ${\mathcal N}_{\;t,x}$ is the Gaussian measure with mean $e^{tA}x$, $A=
(-2Q)^{-1} $ and covariance $Q(I-e^{-2tA })$. For a detailed study of $T(t)$ see e.g. \cite{DP04}. 

Let us check that $T(t)$ satisfies the assumptions of  Proposition \ref{semigruppoB}(c). First,  $T(t)$ maps $L^p(X, \gamma)$ into $W^{1,p}_0(X, \gamma)$ for every $p>1$, and  fixed any orthonormal basis $\{e_k:\;k\in \N\}$ consisting of eigenvectors of $Q$, say $Qe_k =\lambda_ke_k$, 
we have $\partial/\partial e_k \,T(t) f = e^{-\alpha_k t} T(t)\partial f/\partial e_k $, with $\alpha_k = 1/2 \lambda_k$. Therefore, $\nabla T(t) f = e^{tA}{\bf T}(t)\nabla f $, where 
 ${\bf T}(t)F(x) := \int_X F(y) \,{\mathcal N}_{\;t,x}(dy)$  (vector valued integral), and we may take  ${\bf S}_1(t) = e^{tA}{\bf T}(t) $, ${\bf S}_2(t) =0$. 
$ {\bf S}_1(t)$ is a contraction semigroup in $L^p(X, \gamma;X)$ for every $p$. 
 For any $F= \sum_{k=1}^n f_k z_k \in \widetilde{C}^1_b(X, X)$ a simple computation gives  $ {\bf S}_1(t)^*F = \sum_{k=1}^n T(t)f_k e^{tA}z_k$. So, $ {\bf S}_1(t)^*F\in \widetilde{W}^{1,p}_0(X, \gamma;X)$ and since $ {\bf S}_1(t)$ is a contraction  in $L^1(X, \gamma;X)$, $ {\bf S}_1(t)^*$ is a contraction too  in $L^{\infty}(X, \gamma;X)$,  for every $t$. Therefore, the assumptions of Proposition  \ref{semigruppoB}(c)  are satisfied, with $C_1(t) = 1$, $C_2(t) =0$. By Proposition  \ref{semigruppoB}(c),  for every $u\in L^p(X, \nu)$ with $p>1$ we have 
$$ u\in BV_0(X, \nu) \Longleftrightarrow \liminf_{t\to 0} \int_X \|\nabla_pT(t)u\|\,d\gamma <+\infty,  $$
and by Corollary  \ref{uguali} in this case we have $V(u) =  \liminf_{t\to 0}  \|\nabla_pT(t)u\|_{L^1(X, \gamma; X)}$. 
Again, the  commutation formula yields $\nabla_pT(t+s)u = \nabla_pT(t)(T(s)u) = e^{tA}{\bf T}(t)(\nabla_p T(s)u)$, and since ${\bf T}(t)$ is a contraction semigroup in $L^1(X, \nu;X)$ and $e^{tA}$ is a contraction semigroup in $X$, the function $t\mapsto  \int_X \|\nabla_pT(t)u\|\,d\gamma  $ is decreasing, and the above  $\liminf$ is in fact a limit.

\vspace{3mm}

\subsubsection{Weighted Gaussian measures}
Let us consider now a weighted Gaussian measure with nonnegative weight $w$,
$$\nu(dx) = w(x) \gamma(dx), $$
where $\gamma$ is a nondegenerate Gaussian measure with mean $0$ and covariance $Q$. 

It is easily seen that Hypothesis \ref{h1'} is satisfied, still with $R= Q^{1/2}$, if $w$, $\log w \in W^{1,p}(X, \gamma)$ for every $p>1$. In this case  the measure $\nu$ is Fomin differentiable along the directions of $R(X)$, and for every $z\in X$ we have $v_z = \hat{h} +  \partial \log w /\partial h$, with $h=Q^{1/2}z$. Each $v_z$ belongs to $L^q(X, \nu)$ for every $q\geq 1$ by the H\"older inequality. So, Hypothesis \ref{h1'} holds, and the results of Section \ref{sectionBV}  are applicable  (for a detailed study of Sobolev spaces for weighted Gaussian measures see \cite{Simone}). 

Concerning Proposition \ref{semigruppoB}, to have such a good semigroup $T(t)$ we need more assumptions on the weight. We write it in the form 
\begin{equation}
\label{defpeso}
w(x)=\frac{1}{\int_X e^{-2U}d\gamma} e^{-2U(x)}
\end{equation}
 to agree with the notation of the paper \cite{DPL}, from which we borrow   assumptions and results. In particular, we assume

\begin{Hypothesis}
\label{hyp_conv}
$U:X\mapsto \R \cup \{+\infty\} $ is  convex, lower semicontinuous and   bounded from below; $U\in \cap_{p>1}W^{1,p}(X, \gamma) $. 
\end{Hypothesis}
 
As a consequence of \cite[Lemma 2.7]{DPL}, under Hypothesis \ref{hyp_conv} the quadratic form 
$$(u,v)\mapsto \int_X \langle \nabla u, \nabla v\rangle \,d\nu, \quad u, \;v\in W^{1,2}_{0}(X, \nu)$$
is a Dirichlet form. The associated operator $K:D(K)\subset L^2(X, \nu)\mapsto L^2(X, \nu)$, 
\begin{equation}
\label{K}
\left\{
\begin{array}{lll} D(K) & = & \left\{ u\in L^2(X, \nu):  \exists f\in L^2(X, \nu) \; \mbox{s.t.} 
 \int_X \langle \nabla u, \nabla \varphi \rangle \,d\nu = -\int_X f\,\varphi\, d\nu \; \forall \varphi\in 
 W^{1,2}_{0}(X, \nu)\right\}, 
 \\
 \\
 Ku & = &  f, 
 \end{array}\right. 
 \end{equation}
is the infinitesimal generator of the realization in $L^2(X, \nu)$ of a Markov semigroup $T(t)$ which enjoys  better regularization properties than the one associated to the quadratic form
$$(u,v)\mapsto \int_X \langle M_2 u, M_2 v\rangle \,d\nu, \quad u, \;v\in W^{1,2} (X, \nu), $$
and we shall use it for the characterization of both $BV$ and $BV_0$ functions.

If in addition   $\nabla U$ is Lipschitz continuous, by \cite[Prop. 3.8]{DPL} $T(t)$ is the  transition semigroup  of the stochastic differential equation $dX=(AX-\nabla U(X))dt+dW(t)$, where $A= (-2Q)^{-1}$, and $W(t)$  is any cylindrical $X$-valued Wiener process defined in a probability space $(\Omega , \mathcal F, \P)$. More precisely, 
for every $\varphi\in C_b(X)$ we have $T(t)\varphi = \E(\varphi(X(t,x)))$, where $X(\cdot,x)$ is the unique solution to
\begin{equation}
  \label{e1.3}
\left\{
\begin{array}{l}
dX_t=(AX_t-\nabla U(X_t))dt+dW(t), 
\\
\\
X_0=x. 
\end{array}\right. 
\end{equation}

As in the case of nonweighted Gaussian measures  is convenient to   fix once and for all an orthonormal basis $\{e_k:\; k\in \N\}$ of $X$ consisting of eigenvectors of $Q$,   $Qe_k = \lambda_k e_k$ for each $k\in \N$. 
Then, 
\begin{equation}
 \label{vek}
  v_{e_k}(x) =  \frac{ \langle x, e_k\rangle }{\sqrt{\lambda_k} }  -2  \sqrt{\lambda_k}\langle \nabla U(x), e_k\rangle, \quad k\in \N , 
\end{equation}
so that 
\begin{equation}
\label{intpartipeso}
\int_X \psi \,\frac{\partial \varphi}{\partial e_k} \,d\nu = - \int_X \frac{\partial \psi}{\partial e_k}\,\varphi \,d\nu + \int_X\psi\,\varphi \left( 
 \frac{ \langle x, e_k\rangle }{ \lambda_k}   -2 \langle \nabla U(x), e_k\rangle\right) \nu(dx), \quad \psi, \; \varphi\in C^1_b(X). 
\end{equation}

The operator $K$ has a nice expression on good functions; in fact if   $f\in {\mathcal{FC}}^2_b(X)$ is of the type 
$f(x) = \theta( \langle x, e_1\rangle, \ldots, \langle x, e_n\rangle) $ for some $\theta \in C^2_b(\R^n)$, using \eqref{intpartipeso}  yields  $f\in D(K)$ and 
\begin{equation}
\label{Kcilindrico}
\begin{array}{lll}
Kf (x)&  = & \ds  \frac{1}{2}\sum_{k=1}^n \frac{\partial ^2 f}{\partial e_k^2}(x) + \sum_{k=1}^n \frac{\partial f}{\partial e_k} (x)\left( -  \frac{1}{2}
\frac{ \langle x, e_k\rangle}{\lambda_k} - \frac{\partial U}{\partial e_k} (x)\right)
\\
\\
& = & 
\ds \frac{1}{2}\;\Tr\; [ D^2f(x) ] +  \langle   x,A\nabla f (x)\rangle  - \langle \nabla U(x),  \nabla f(x)\rangle .
\end{array}
\end{equation} 
Since $K$ is self-adjoint,  $T(t)$ is a self-adjoint operator   in $L^2(X, \nu)$,  for every $t>0$. Since $K$ is also dissipative,  $T(t)$  is a contraction analytic semigroup.  
Therefore, $\|T(t)u\|_{L^2(X, \nu)}\leq \|u\|_{L^2(X, \nu)}$ for every $t>0$, $u\in L^2(X, \nu)$; moreover $T(t)$ maps $L^2(X, \nu)$ into $D(K)$  and there exists $C>0$ such that $\|KT(t)f\|_{L^2(X, \nu)} \leq Ct^{-1}\|f\|_{L^2(X, \nu)}$ for every $t>0$. We can take  $C=e^{-1}$, see the proof of next Proposition \ref{Kalpha}. 
For every $f\in L^2(X, \nu)$, such estimate and the definition of $K$ give
\begin{equation}
\label{stimabase}
\int_X \|\nabla_2T(t)u\|^2d\nu = - \int_X T(t)f\,KT(t)f\,d\nu \leq \|T(t)f\|_{L^2(X, \nu)} \|KT(t)f\|_{L^2(X, \nu)} \leq \frac{1}{te}\|f\|_{L^2(X, \nu)}^2. 
\end{equation} 

\begin{Theorem}
\label{Th:AoP}
Let Hypothesis \ref{hyp_conv} hold, and assume in addition that $U\in C^1(X)$ and that $\nabla U$ is Lipschitz continuous. Let 
$\nu $ be defined by \eqref{defpeso} and choose $R= Q^{1/2}$. Then for every $u\in L^2(X, \nu)$   we have
\begin{equation}
\label{caratt_semigruppo_peso}u\in BV(X, \nu) \iff  \liminf_{t\to 0^+} \int_X \left\| M_2T(t)u\right\| d\nu <+\infty , 
\end{equation}
and in this case 
 $V(u) =  \liminf_{t\to 0^+} \int_X \left\| M_2T(t)u\right\| d\nu$; 
\begin{equation}
\label{caratt_semigruppo_peso_0}u\in BV_0(X, \nu) \iff  \liminf_{t\to 0^+} \int_X \left\| \nabla_2T_2(t)u\right\| d\nu <+\infty , 
\end{equation}
and in this case 
 $V_0(u) =  \liminf_{t\to 0^+} \int_X \left\| \nabla_2T(t)u\right\| d\nu $. 
\end{Theorem}
\begin{proof}
Since   $T(t)$  is an analytic semigroup, for every $t>0$ we have $T(t)(L^2(X, \nu)) \subset D(K) \subset W^{1, 2}_0(X, \nu) \subset W^{1,2}(X, \nu)$, and \eqref{regolarizz} is satisfied with $p=q=2$.  Proposition \ref{semigruppoA} yields 
$$ \liminf_{t\to 0^+} \int_X \|M_2T(t)u\|\, d\nu <+\infty     \Longrightarrow  u\in BV(X, \nu);$$
$$ \liminf_{t\to 0^+} \int_X \|\nabla_2T(t)u\|\, d\nu <+\infty 
    \Longrightarrow  u\in BV_0(X, \nu). $$
To prove that the converse holds, we shall show that the assumptions in (b) and (c) of Proposition  \ref{semigruppoB}   are satisfied  with $p=2$ and $( {\bf S_1}(t)F)(x)  = e^{tA}{\bf T}(t)F(x)$. Here, as before, ${\bf T}(t)F(x) := {\mathbb E}(F(X(t, x))$ for $F\in C_b(X, X)$ is canonically extended to a contraction semigroup in all spaces $L^p(X, \nu;X)$ for $p\geq 1$.

With such choices of $p$ and $ {\bf S_1}(t)$, 
\eqref{regolarizzT*vett}(i)  is  satisfied, since for every $F  \in \widetilde{C}^1_b(X, X)$ we have $ {\bf S_1}(t)^* F= {\bf T} (t)(e^{tA}F(\cdot))$, so that  $  {\bf S_1}(t)^* F\in W^{1,2}_0(X, \nu; X) \subset  W^{1,2}(X, \nu; X)$,   and recalling that $\| {\bf T } (t)G\|_{\infty} \leq \|G\|_{\infty}$ for $G\in L^2(X, \nu;X) \cap L^{\infty}(X, \nu;X)$,  for every $x\in X$ we have
\begin{equation}
\label{stimaS1}
\| ( {\bf S_1}(t)^* F)(x)\| = \|({\bf T}(t)(e^{tA}F(\cdot))(x)\| \leq \| e^{tA}\|_{{\mathcal L}(X)}\|F\|_{\infty}, \quad t>0, 
\end{equation}
and therefore \eqref{regolarizzT*vett}(i) is  satisfied with $C_1(t)= 1$ for every $t>0$. 

Now we show that  \eqref{regolarizzT*vett}(ii)  holds, with ${\bf S_2}(t) = M_2T(t) -  e^{tA}{\bf T}(t)M_2$ for the characterization of $BV$ functions, 
 and with 
${\bf S_2}(t) = \nabla_2T(t) -  e^{tA}{\bf T}(t)\nabla_2$ for the characterization of $BV_0$ functions. More precisely, we have to show that 
\begin{equation}
\label{argh}
  \|  M_2T(t)f -  e^{tA}{\bf T}(t)M_2f\|_{ L^1(X, \nu; X)} \leq C_2(t) \|f\|_{L^p(X, \nu)}, \quad t>0, \; f\in C^1_b(X)
  \end{equation}
to prove $\Leftarrow$ of \eqref{caratt_semigruppo_peso}, and 
\begin{equation}
\label{argh0}
  \|  \nabla_2T(t)f -  e^{tA}{\bf T}(t)\nabla_2f\|_{ L^p(X, \nu; X)}  \leq C_2(t) \|f\|_{L^p(X, \nu)}, \quad t>0, \; f\in C^1_b(X)
    \end{equation}
to prove $\Leftarrow$ of \eqref{caratt_semigruppo_peso_0}. In both cases we shall check that $\lim_{t\to 0}C_2(t) =0$, so that $V(u)$ and $V_0(u)$ will be characterized through Corollary \ref{uguali}.

To this aim we recall that, since $\nabla U$ is Lipschitz continuous,  $U\in W^{2,2}(X, \gamma )$ and for $\gamma$-a.e $x\in X$ there exists the Gateaux derivative of $\nabla U$ at $x$, denoted by $D^2U(x)$ (e.g., \cite[Thms. 5.11.1, 5.11.2]{Boga}). Denoting by $L$ the Lipschitz constant of $\nabla U$, we have $\|D^2U(x)\|_{{\mathcal L}(X)}\leq L$, for 
$\gamma$-a.e $x\in X$, and therefore for $\nu$-a.e $x\in X$. 

Moreover, we
consider the above mentioned  orthonormal basis $\{e_k:\; k\in \N\}$ of $X$ consisting of eigenvectors of $Q$,  
and we use the commutation formula 
\begin{equation}
\label{commk}
\frac{\partial T(t)f }{\partial e_k} (x) - e^{-t/2\lambda_k}T(t)\left( \frac{\partial f }{\partial e_k} \right)(x) = - \int_0^t e^{-(t-s)/2\lambda_k}(T(t-s) \langle D^2U(\cdot )e_k,   \nabla T(s)f (\cdot ))(x)\rangle  \,ds , 
\end{equation}
that holds for  $t>0$, $f \in C^1_b(X)$,  $k\in \N$, and whose proof is deferred to the Appendix. 
Once \eqref{commk} is established,  we rewrite it  as
\begin{equation}
\label{comm}
\nabla T(t)f    - (e^{tA}{\bf T}(t)\nabla f)  = - \int_0^t  e^{ (t-s)A}{\bf T}(t-s) ( D^2U \cdot  \nabla T(s)f  )    \,ds. 
\end{equation}
Recalling that  $M_2 = Q^{1/2}\nabla_2$ and that $Q^{1/2} = (-A/2)^{-1/2}$,  we get  
\begin{equation}
\label{commCM}
M_2 T(t)f    - (e^{tA}{\bf T}(t)M_2 f)  = - Q^{1/2}\int_0^t e^{ (t-s)A}{\bf T}(t-s) ( D^2U \cdot  \nabla T(s)f )   \,ds .
\end{equation}
To obtain \eqref{argh} and \eqref{argh0} we have to estimate the right hand sides of \eqref{commCM} and of  \eqref{comm}, respectively, with $f\in C^1_b(X)$ which is dense in $L^2(X, \nu)$. 
Setting  ${\bf S_2}(t) f:= M_2 T(t)f   - (e^{tA}{\bf T}(t)M_2 f)$ we obtain
\begin{equation}
\label{stimaS2}
\begin{array}{lll}
\|{\bf S}_2(t) f \|_{L^2(X, \nu;X)} & \leq & \ds \|Q^{1/2}\|_{{\mathcal L}(X)}\int_0^t \| e^{ (t-s)A}\|_{{\mathcal L}(X)}\esssup_{x\in X}\|D^2U(x)\|_{\mathcal L(X)} \| \nabla T(s)f\|_{ L^2(X, \nu;X)} ds
\\
\\
&
\leq  & \ds  \|Q^{1/2}\|_{{\mathcal L}(X)}C\int_0^t  s^{-1/2} ds  \;L \|f\|_{L^2(X, \nu)} 
\\
\\
& = & 2CL\|Q^{1/2}\|_{{\mathcal L}(X)}\sqrt{t}  \|f\|_{L^2(X, \nu)} , 
\end{array}
\end{equation}
where we used estimates \eqref{stimabase} and $\|e^{tA}\|_{\mathcal L (X)} \leq 1$. Estimate  \eqref{argh} follows. Arguing similarly, setting  ${\bf S_2}(t) f:= \nabla_2 T(t)f   - (e^{tA}{\bf T}(t)\nabla_2 f)$ we obtain, for $f\in C^1_b(X)$,   
$$\|{\bf S}_2(t) f \|_{L^2(X, \nu;X)}  \leq  2CL \sqrt{t}  \|f\|_{L^2(X, \nu)} $$
and  \eqref{argh0} follows. Corollary \ref{uguali} yields the statement. 
\end{proof}

Theorem \ref{Th:AoP} gives a proof to some of the statements of \cite{AmDaGoPa12}.

\vspace{3mm}

Hypothesis \ref{hyp_conv} is rather restrictive. In explicit examples it may be considerably weakened, allowing for a convex $U\in W^{2,p}(X, \gamma)$, as in the next example, still borrowed from \cite{DPL}. It is motivated by a stochastic reaction-diffusion equation, 
\begin{equation}
\label{reazdiff}
\left\{\begin{array}{lll}
dX(t)=[AX(t)- f(X(t))]dt+ dW(t),\\
\\
X(0)=x, 
\end{array}\right.
\end{equation}
where $A$ is  the realization of the second order derivative with Dirichlet boundary condition in $X:=L^2((0, 1), d\xi)$, i.e. $D(A)=W^{2,2}(0, 1) \cap W^{1,2}_0(0, 1) $, $Ax = x''$ (the interval $(0,1)$ is endowed with the Lebesgue measure). The nonlinearity $f:\R\mapsto \R$ is an increasing polynomial, with degree $d>1$, and   $W$ is  an $X$--valued cylindrical Wiener process. 
As before,  we consider the Gaussian measure $\gamma$  with mean $0$ and covariance 
$Q= (-2A)^{-1}$. A convenient orthonormal basis of $X$ consisting of eigenvectors of $Q$, that we shall consider from now on,  is the set of the functions
$$e_k(\xi) := \sqrt{2\pi} \sin(k\pi\xi), \quad k\in \N, \;\xi\in (0, 1). $$
For every $x\in X$ equation \eqref{reazdiff}  has a unique generalized  solution $X(\cdot, x)$ (e.g., \cite[Ch. 7]{Cerrai},  \cite[Ch. 4]{DP04}));  the associated transition semigroup is  defined as usual by 
\begin{equation}
\label{transition}
P(t)\varphi (x) := \E [\varphi (X(t,x))], \quad \varphi\in C_b(X), \; t\geq 0. 
\end{equation}

Let $\Phi:\R\mapsto \R$ be any primitive of $f$ and set
\begin{equation}
\label{e4.2}
U(x)= \left\{ \begin{array}{lll}  & \ds \int_0^{1} \Phi(x(\xi))d\xi, & x\in L^{d+1}(0,1), 
\\
\\
 & +\infty , & x \notin L^{d+1}(0,1). 
\end{array}\right. 
\end{equation}
In \cite[Sect. 5]{DPL} we proved that $U\in W^{1,p}_0(X, \gamma) \cap W^{2,p}(X, \gamma)$ for every $p>1$, and that it is convex and lower semicontinuous. Therefore, $U$ satisfies Hypothesis \ref{hyp_conv} and the measure $\nu$ defined in \eqref{defpeso} satisfies Hypothesis  \ref{h1'}. Moreover, we proved that $\nabla_p U(x) = f\circ x$ for $\gamma$-a.e. $x\in X$, so that \eqref{reazdiff} is similar to \eqref{e1.3}, with a Sobolev gradient in the drift term instead of a Lipschitz continuous one.

As before, we denote by $T(t)$ the semigroup 
generated by the operator $K$ defined in \eqref{K} (in the next proposition we shall identify $T(t)f$ with $P(t)f$ for every $f\in C_b(X)$).

To characterize $BV$ and $BV_0$ functions through the semigroup $T(t)$ we  cannot apply Theorem \ref{Th:AoP}, since $U$ has not Lipschitz continuous gradient. We shall use an approximation procedure, approaching $U$ through its Moreau-Yosida approximations $U_{\alpha}$ defined for $\alpha >0$ by 
\begin{equation}
\label{e1.8}
U_\alpha(x):=\inf\left\{U(y)+\frac{|x-y|^2}{2\alpha},\;y\in H\right\},\quad x\in X. 
\end{equation}
Then,  $U_{\alpha}(x) \leq U(x)$ for every $x\in X$ and $U_{\alpha}(x)$ converges monotonically  to $U(x)$ for each $x$ as $\alpha \to 0$. Moreover, each $U_{\alpha}$ is differentiable at any point and $\nabla U_{\alpha}$ is Lipschitz continuous (e.g.,  \cite[Ch. 2]{Brezis}). 
We have in fact 
\begin{equation}
\label{gradUalpha}
\nabla U_\alpha(x) =  f_{\alpha} \circ x,\;\; D^2U_{\alpha}(x)e_k =  (f_{\alpha}' \circ x)\cdot e_k,\quad x \in X, \;\alpha >0, \; k\in \N, 
\end{equation}
where $ f_{\alpha } = f\circ(I+\alpha f)^{-1}$ is the Yosida approximation of $f$,  and $\cdot$ denotes pointwise multiplication  (e.g.,  \cite[Prop. 5.4]{DPL}). 

We set
$$\nu_{\alpha}(dx) := \frac{1}{\int_X e^{-2U_{\alpha}}d\gamma} e^{-2U_{\alpha}(x)}\gamma(dx),\quad \alpha >0. $$
Since $U_{\alpha}(x)\leq U(x)$ for every $x\in X$, we have
\begin{equation}
\label{confronto}
\|f\|_{L^1(X, \nu)} \leq \|f\|_{L^1(X, \nu_{\alpha})}\frac{ \int_X e^{-2U_{\alpha}(x) }\gamma(dx)}{ \int_X e^{-2U (x) }\gamma(dx)}, 
\end{equation}
for every $\gamma$-measurable function $f$. 

For every $\alpha >0$ we consider the operators $K_{\alpha}$ defined by \eqref{K} with $\nu$ replaced by $\nu_{\alpha}$, and the semigroups $T_{\alpha}(t)$ generated by $K_{\alpha}$ in $L^2(X, \nu_{\alpha})$. They satisfy (see \eqref{stimabase})  
\begin{equation}
\label{Talphagenerale}
\left\{ \begin{array}{ll}
(i) & 
\|T_{\alpha}(t)\varphi\|_{L^2(X, \nu_{\alpha})} \leq  \|\varphi\|_{L^2(X, \nu_{\alpha})},
\\
\\
(ii) &
\|\nabla_2 T_{\alpha}(t)\varphi\|_{L^2(X, \nu_{\alpha};X)} \leq  ( \|K_{\alpha}T_{\alpha}(t)\varphi \|_{L^2(X, \nu_{\alpha})} \| \varphi \|_{L^2(X, \nu_{\alpha})})^{1/2} 
\end{array}\right. 
\end{equation}
for every $\varphi\in L^2(X, \nu_{\alpha})$ and $t>0$.  

\begin{Proposition}
\label{Kalpha}
Let $P(t)$ be the transition semigroup of \eqref{reazdiff}. For every $\varphi\in C_b(X)$ and $t>0$, we have 
$$P(t)\varphi = T(t)\varphi, \quad \lim_{\alpha\to 0}\|T_{\alpha}(t)\varphi - T(t)\varphi\|_{L^2(X, \nu)} =0. $$
Moreover, 
for every $\beta\in (0, 1/2] $ there exists $C_{\beta}>0$, independent of $\alpha$, such that 
\begin{equation}
\label{interpolazione}
\|Q^{-\beta} \nabla_2 T_{\alpha}(t)f\|_{L^2(X, \nu_{\alpha};X)} \leq  C_{\beta}\left( 1+ \frac{1}{t^{(1+\beta)/2}}\right)  \| f\|_{L^2(X, \nu_{\alpha})}, \quad t>0, \;\alpha >0. 
\end{equation}
\end{Proposition}
\begin{proof} 
Let us consider the transition semigroups $P_{\alpha}(t)$ defined by $(P_{\alpha}(t)\varphi )(x) : =\E \varphi(X_{\alpha}(t, x)$,  where $X_{\alpha}(\cdot, x)$ is the solution to 
\eqref{reazdiff} with $f$ replaced by its Yosida approximation $f_{\alpha}$. By \cite[Prop. 3.8]{DPL} we have 
$$R(\lambda, K_{\alpha})\varphi  = \int_0^{\infty} e^{-\lambda t}P_{\alpha}(t)\varphi \; dt, \quad \lambda >0. $$
Now we let $\alpha\to 0$. From the proof of Theorem 3.9 of \cite{DPL} it follows  that $R(\lambda, K_{\alpha})\varphi  $ weakly converges in $L^2(X, \nu)$ to $R(\lambda, K )\varphi =  \int_0^{\infty} e^{-\lambda t}T(t)(t)\varphi \; dt $. On the other hand, by \cite{Cerrai}, \cite[Thm. 4.8]{DP04}, we have $\lim_{\alpha \to 0} \|X_{\alpha}(\cdot, x) - X (\cdot, x) \|_{C([0, T]; L^2(\Omega, \P;X))} =0$ for every $x\in X$, and therefore by the Dominated Convergence Theorem $\lim_{\alpha \to 0}\| P_{\alpha}(t)\varphi (x) - P(t)\varphi (x)\|=0$ for every $x\in X$; since $\|P_{\alpha}(t)\varphi \|_{\infty} \leq \|\varphi\|_{\infty} $, $\|P (t)\varphi \|_{\infty} \leq \|\varphi\|_{\infty} $,  still by the Dominated Convergence Theorem we get 
$\lim_{\alpha\to 0} \| P_{\alpha}(t)\varphi  - P(t)\varphi \|_{L^2(X, \nu)} =0$ for every $t>0$ and 
$$\lim_{\alpha\to 0} \left\| \int_0^{\infty} e^{-\lambda t}P_{\alpha}(t)\varphi \; dt - \int_0^{\infty} e^{-\lambda t}P(t)\varphi \; dt \right\|_{L^2(X, \nu)} =0, \quad \lambda>0. $$ 
So, we have
$$\int_0^{\infty} e^{-\lambda t}T(t)\varphi \; dt = \int_0^{\infty} e^{-\lambda t}P(t)\varphi \; dt, \quad \lambda >0. $$
The functions $t\mapsto T(t)\varphi$, $t\mapsto P(t)\varphi$ are continuous and bounded  in $[0, +\infty)$ with values in $L^2(X, \nu)$. Since their Laplace transforms coincide for $\lambda >0$, they coincide in $[0, +\infty)$. 

To prove \eqref{interpolazione}, we preliminary remark that for every Hilbert space $H$ and for every self-adjoint dissipative linear operator $L:D(L)\subset  H\mapsto   H$ and $t>0$
we have $\|Le^{tL}\|_{ \mathcal L(H)} \leq 1/et$. This can be seen using the spectral decomposition $\{E_{\lambda}: \; \lambda \leq 0\}$ of $L$ and writing 
$Le^{tL} = \int_{-\infty}^0 \lambda e^{t\lambda}dE_{\lambda}$, so that $\|Le^{tL}\|_{ \mathcal L(H)} \leq \sup_{\lambda <0} | \lambda e^{t\lambda}| = 1/et$. 

By \cite[Prop. 3.7]{DPL} with $\lambda =1$, for every $\alpha >0$ we have
\begin{equation}
\label{Prop37}
 \|Q^{-1/2}\nabla_2 \varphi\|_{L^2(X, \nu_{\alpha};X)}^2 \leq 8 \|\varphi - K_{\alpha}\varphi\|_{L^2(X, \nu_{\alpha})}^2, \quad \varphi\in D(K_{\alpha}). 
 \end{equation}
Taking $\varphi = T_{\alpha}(t)f$ for any $f\in L^2(X, \nu_{\alpha})$, and using \eqref{Talphagenerale}   we get 
$$\begin{array}{lll}
\|Q^{-1/2} \nabla_2 T_{\alpha}(t)f\|_{L^2(X, \nu_{\alpha};X)}^2 &  \leq & 8(\|T_{\alpha}(t)f\|_{L^2(X, \nu_{\alpha})} + \|K_{\alpha}  T_{\alpha}(t)f\|_{L^2(X, \nu_{\alpha})})^2
\\
\\
& \leq & \ds 8\left(\|f\|_{L^2(X, \nu_{\alpha})} + \frac{1}{te}\| f\|_{L^2(X, \nu_{\alpha})}\right)^2. \end{array}$$
Therefore \eqref{interpolazione} holds for $\beta =1/2$. On the other hand, by \eqref{Talphagenerale}(ii) and \eqref{Prop37} , 
$$\| \nabla_2 T_{\alpha}(t)f\|_{L^2(X, \nu_{\alpha};X)} \leq   ( \|K_{\alpha}T_{\alpha}(t)f\|_{L^2(X, \nu_{\alpha})} \|  f\|_{L^2(X, \nu_{\alpha})})^{1/2} 
\leq  \sqrt{\frac{1}{et} }\| f\|_{L^2(X, \nu_{\alpha})}. $$
Therefore, for  every $\beta\in (0, 1/2)$, 
$$\begin{array}{l}
\|Q^{-\beta} \nabla_2 T_{\alpha}(t)f\|_{L^2(X, \nu_{\alpha};X)} \leq (\|Q^{-1/2} \nabla_2 T_{\alpha}(t)f\|_{L^2(X, \nu_{\alpha};X)})^{\beta} ( \|  \nabla_2 T_{\alpha}(t)f\|_{L^2(X, \nu_{\alpha};X)})^{1-\beta}
\\
\\
\ds \leq C_{1/2}^{\beta}\left(1  + \frac{1}{t}\right)^\beta \left(\frac{1 }{et}\right)^{(1-\beta)/2} \| f\|_{L^2(X, \nu_{\alpha})}, \end{array}$$
and \eqref{interpolazione} follows. 
\end{proof}

\begin{Theorem}
\label{ReactDiff}
 Let 
$\nu $ be defined by \eqref{defpeso}, with $U$ given by \eqref{e4.2},  and choose $R= Q^{1/2}$.  For every $u\in L^2(X, \nu)$   we have
\begin{equation}
\label{caratt_semigruppo_pesoRD}u\in BV(X, \nu) \iff  \liminf_{t\to 0^+} \int_X \left\| M_2T(t)u\right\| d\nu <+\infty , 
\end{equation}
and in this case $V(u) =  \liminf_{t\to 0^+}  \| M_2T(t)u \|_{L^1(X, \nu;X)}$. Moreover, 
\begin{equation}
\label{caratt_semigruppo_peso_0RD}u\in BV_0(X, \nu) \Longleftarrow \liminf_{t\to 0^+} \int_X \left\| \nabla_2T_2(t)u\right\| d\nu <+\infty . 
\end{equation}
\end{Theorem}
\begin{proof} 
The proof of the implications $\Longleftarrow$  is the same as the first part of the proof of Theorem \ref{Th:AoP}, and it is omitted. 
To  show that the implication $\Longrightarrow$ of \eqref{caratt_semigruppo_pesoRD} holds we shall prove that the assumptions  of Proposition  \ref{semigruppoB}(b)   are satisfied  with $p=2$,  $( {\bf S_1}(t)F)(x)  = e^{tA}{\bf T}(t)F(x)$, and of course $ {\bf S_2}(t)\varphi = M_2T(t)\varphi - e^{tA}{\bf T}(t)M_2\varphi$. Since \eqref{regolarizz} holds with $p=2$, we have to prove that \eqref{regolarizzT*vett} holds. 

Once again, the proof of \eqref{regolarizzT*vett}(i) is the same as in   Theorem \ref{Th:AoP}, and it is omitted. The constant $C_1(t)$ is now $\|e^{tA}\|_{{\mathcal L}(X)}\leq 1$ for every $t$. 

We  show here that  \eqref{regolarizzT*vett}(ii) holds; 
to this aim we consider again the semigroups $T_{\alpha}(t)$. Since the functions $U_{\alpha}$ satisfy the assumptions of Theorem \ref{Th:AoP}, for every $f\in C^1_b(X)$ and $t>0$ formula \eqref{commCM}  yields
\begin{equation}
\label{commCMalpha}
M_2 T_{\alpha}(t)f    - (e^{tA}{\bf T}_{\alpha}(t)M_2 f)  = - Q^{1/2}\int_0^t e^{ (t-s)A}{\bf T}_{\alpha}(t-s) ( D^2U_{\alpha} \cdot  \nabla T_{\alpha}(s)f )   \,ds  . 
\end{equation}
Now we let $\alpha \to 0$.  We remark  that the estimate \eqref{stimaS2} that we got for the left hand side of \eqref{commCMalpha} is not useful here, because it depends on the Lipschitz constant of $\nabla U_{\alpha}$ that blows up as $\alpha\to 0$.

By \eqref{stimagradQ}(i), $\|\nabla T_{\alpha}(t)f \|_{\infty} \leq \|\nabla f \|_{\infty}$. Therefore, $\{T_{\alpha}(t)f : \;0<\alpha<1\}$ is bounded in $W^{1,2}(X, \nu)$, so that there exists a sequence $\alpha_k\to 0$ such that $(T_{\alpha_k}(t)f )$ weakly converges in $W^{1,2}(X, \nu)$. Since $\lim_{k\to \infty} T_{\alpha}(t)f  = T (t)f $ in $L^2(X, \nu)$ by Proposition \ref{Kalpha}, 
we have $T_{\alpha}(t)f \rightharpoonup  T (t)f $ in $W^{1,2}(X, \nu)$ as $\alpha \to 0$. Therefore, $M_2 T_{\alpha}(t)f $ weakly converges to $M_2 T (t)f $ in $L^2(X, \nu; X)$. 
Still by Proposition \ref{Kalpha}, $\lim_{\alpha \to 0}{\bf T}_{\alpha}(t)M_2 f = {\bf T} (t)M_2 f$ in $L^2(X, \nu; X)$, consequently $\lim_{\alpha \to 0}e^{tA}{\bf T}_{\alpha}(t)M_2 f = e^{tA} {\bf T} (t)M_2 f$ in $L^2(X, \nu; X)$. Summing up, the left hand side of \eqref{commCMalpha}  weakly converges to 
$M_2 T (t)f    - (e^{tA}{\bf T} (t)M_2 f) $ in $L^2(X, \nu; X)$, and therefore in $L^1(X, \nu; X)$. 
So, 
$$\|M_2 T (t)f    - e^{tA}{\bf T} (t)M_2 f \|_{L^1(X, \nu; X)}  \leq   \liminf_{\alpha \to 0} \|M_2 T_{\alpha}(t)f    - e^{tA}{\bf T}_{\alpha}(t)M_2 f \|_{L^1(X, \nu; X)} $$
and recalling \eqref{confronto} we get 
\begin{equation}
\label{magglimite}
\|M_2 T (t)f    - e^{tA}{\bf T} (t)M_2 f\|_{L^1(X, \nu; X)}  \leq   \liminf_{\alpha \to 0} \|M_2 T_{\alpha}(t)f    - e^{tA}{\bf T}_{\alpha}(t)M_2 f \|_{L^1(X, \nu_{\alpha}; X)} .
\end{equation}
Now we estimate $ \|M_2 T_{\alpha}(t)f    - e^{tA}{\bf T}_{\alpha}(t)M_2 f \|_{L^1(X, \nu_{\alpha}; X)}$. Fix any $\beta \in (1/4, 1/2)$. Using \eqref{commCMalpha} and recalling that 
$\|e^{tA}\|_{\mathcal L(X)} \leq 1$ and that $ T_{\alpha}(t)$ is a contraction semigroup in $L^1(X, \nu_{\alpha}; X)$, we get 
\begin{equation}
\label{lunga}
\begin{array}{l}
 \ds  \|M_2 T_{\alpha}(t)f    - e^{tA}{\bf T}_{\alpha}(t)M_2 f \|_{L^1(X, \nu_{\alpha}; X)} =
 \\
 \\
\ds =  \left\|\int_0^t e^{ (t-s)A}{\bf T}_{\alpha}(t-s) ( Q^{1/2} D^2U_{\alpha} \cdot  \nabla T_{\alpha}(s)f )   \,ds  \right\|_{L^1(X, \nu_{\alpha}; X)}
\\
\\ 
\leq \ds   \int_0^t \| Q^{1/2} D^2U_{\alpha} \cdot  \nabla T_{\alpha}(s)f \|_{L^1(X, \nu_{\alpha}; X)} ds =   \int_0^t\| Q^{1/2} D^2U_{\alpha}Q^{\beta} \cdot  Q^{-\beta} \nabla T_{\alpha}(s)f \|_{L^1(X, \nu_{\alpha}; X)} ds
\\
\\ 
\leq \ds  \int_0^t \left( \int_X \| Q^{1/2} D^2U_{\alpha}Q^{\beta}\|_{{\mathcal L}(X) }^2 d\nu  \right)^{1/2} \left( \int_X \| Q^{-\beta} \nabla T_{\alpha}(s)f\|  ^2 d\nu \right)^{1/2} ds
\\
\\ 
\leq \ds \left( \int_X \| Q^{1/2} D^2U_{\alpha}Q^{\beta}\|_{{\mathcal L}(X)} ^2 d\nu \right)^{1/2}  \int_0^t 
  C_{\beta}\left( 1+ \frac{1}{s^{(1+\beta)/2}}\right)\|f\|_{L^2(X, \nu_{\alpha})}, 
  \end{array}
  \end{equation}
where we used  \eqref{interpolazione} in the last estimate. Let us show that  the integral $\int_X \| Q^{1/2} D^2U_{\alpha}Q^{\beta}\|_{{\mathcal L}(X)} ^2 d\nu$ is bounded by a constant independent of $\alpha$. For every $x\in X$ we have
$$\| Q^{1/2} D^2U_{\alpha}Q^{\beta}\|_{{\mathcal L}(X)}^2  \leq \sum_{k,j=1}^{\infty} \langle Q^{1/2} D^2U_{\alpha}(x) Q^{\beta}e_k, e_j\rangle^2 = 
\sum_{k,j=1}^{\infty} \lambda_j \lambda_k^{2\beta} \langle D^2U_{\alpha}(x) e_k, e_j\rangle^2 . $$
We estimate the right hand side  using  \eqref{gradUalpha} and recalling that $|f_{\alpha}' (s)| \leq C(1+|s|^{d-1})$ for some constant $C$ independent of $\alpha$. We get 
$$\begin{array}{l}
\ds | \langle D^2U_{\alpha}(x) e_k, e_j\rangle | = \left| \int_0^1 f_{\alpha}'(x(\xi))e_k(\xi)e_j(\xi)d\xi\right| \leq 2\pi \int_0^1 |f_{\alpha}'(x(\xi))|\,d\xi
\\
\\
\leq \ds 2\pi \int_0^1C(1+ |x(\xi)|^{d-1})d\xi = 2\pi C(1+ \|x\|_{L^{d-1}(0, 1)}^{d-1}), 
\end{array}$$
for every $k$, $j\in \N$. Therefore, 
$$\int_X \| Q^{1/2} D^2U_{\alpha}Q^{\beta}\|_{{\mathcal L}(X)} ^2 d\nu \leq \sum_{j=1}^{\infty} \lambda_j  \sum_{k=1}^{\infty}\lambda_k^{2\beta} \int_X 4\pi^2C^2(1+ \|x\|_{L^{d-1}(0, 1)}^{d-1})^2\nu(dx). $$
Since  $\lambda_k = (k\pi )^{-2}$ and $\beta>1/4$,  we have $\sum_{k=1}^{\infty}\lambda_k^{2\beta}<+\infty$. Moreover, by e.g. \cite[Lemma 5.1]{DPL} the  functions $x\mapsto  \|x\|_{L^{p}(0, 1)}$ belong to $L^q(X, \nu)$ for every $p$, $q\geq 1$. Therefore there exists $\widetilde{C}_{\beta}>0$, independent of $\alpha$, such that 
 $(\int_X \| Q^{1/2} D^2U_{\alpha}Q^{\beta}\|_{{\mathcal L}(X)} ^2 d\nu )^{1/2}\leq \widetilde{C}_{\beta} $ for every $\alpha>0$. Replacing in \eqref{lunga} we get 
$$\begin{array}{lll}
\|M_2 T_{\alpha}(t)f    - e^{tA}{\bf T}_{\alpha}(t)M_2 f \|_{L^1(X, \nu_{\alpha}; X)}  & \leq & 
\ds \widetilde{C}_{\beta}  \int_0^t   C_{\beta}\left( 1+ \frac{1}{s^{(1+\beta)/2}}\right)\|f\|_{L^2(X, \nu_{\alpha})}, 
\\
\\
& = & \widetilde{C}_{\beta}   C_{\beta} (t + 2t^{(1-\beta)/2}/(1-\beta))\|f\|_{L^2(X, \nu_{\alpha})}. 
\end{array}$$
Recalling that $\lim_{\alpha\to 0} \|f\|_{L^2(X, \nu_{\alpha}) }= \|f\|_{L^2(X, \nu ) }$  and using \eqref{magglimite} we get 
 $$ \|M_2 T (t)f    - (e^{tA}{\bf T} (t)M_2 f) \|_{L^1(X, \nu ; X)} \leq 
 \widetilde{C}_{\beta}   C_{\beta} (t + 2t^{(1-\beta)/2}/(1-\beta))\|f\|_{L^2(X, \nu )}, $$
so that \eqref{regolarizzT*vett}(ii) holds, with $\lim_{t\to 0} C_2(t) =0$. Therefore, all the assumptions of Proposition \ref{semigruppoB}(b) are satisfied.
Recalling that $C_1(t) = 1$ for every $t>0$, Corollary \ref{uguali} yields the statement. 
  \end{proof}

\begin{Remark}
The procedure of Theorem \ref{ReactDiff} does not work to prove the converse of   \eqref{caratt_semigruppo_peso_0RD}, because instead of estimating $\int_X \| Q^{1/2} D^2U_{\alpha}Q^{\beta}\|_{{\mathcal L}(X)} ^2 d\nu$ we should estimate $\int_X \|  D^2U_{\alpha}Q^{\beta}\|_{{\mathcal L}(X)} ^2 d\nu$, which blows up as $\alpha\to 0$. 
\end{Remark}
  

\subsection{A non Gaussian product measure}
\label{giocattolo}

We recall the construction and some properties of a family of product measures introduced in \cite{TAMS}. 

Fix any $m\geq 1$, and for  $\mu>0$ define the probability measure on $\R$
   \begin{equation}
\label{e1}
 \nu_{\mu}(d\xi):= a\,\mu^{-\frac{1}{2m}}\;e^{-\frac{|\xi|^{2m}}{2m\mu}}\,d\xi,\quad \xi\in\R,
\end{equation}
where $a :=  (2m)^{1- 1/2m}/ 2 \Gamma(1/2m)$ is a normalization constant such that $ \nu_{\mu}(\R)=1$. For every $N>0$ we have
\begin{equation}
\label{e2}
\int_\R | \xi|^{2N}\nu_{\mu}(d\xi)=a\mu^{-\frac{1}{2m}}\int_\R |\xi| ^{2N}e^{-\frac{|\xi|^{2m}}{2m\mu}}\,d\xi=:b_{N}\mu^{N/m}
\end{equation}
where
$$
b_{N}= a \int_{\R} |\tau|^{N/m} e^{-\frac{|\tau|^{2m}}{2m }}\,d\tau = (2m)^{\frac{N}{m}}\;\frac{\Gamma(\frac{2N+1}{2m})}{\Gamma(\frac1{2m})}.
$$

We choose positive numbers  $\mu_h $, $h\in \N$,  such that 
\begin{equation}
\label{e5}
 \sum_{h=1}^\infty\mu_h^{\frac1m}<\infty 
\end{equation}
so that the product measure  on $\R^{\N}$ defined by
\begin{equation}
\label{e4}
\nu :=\prod_{h=1}^\infty \nu_{\mu_h},
\end{equation}
is well defined and it is concentrated on $\ell^2$. In the space $X:=L^2(0,1)$ we fix once and for all an orthonormal basis $\{e_k: \; k\in \N\}$ of $X$
consisting of equibounded functions: 
$$|e_k(\xi)| \leq C, \quad k\in \N, \; \xi\in (0, 1), $$ 
and we consider the standard isomorphism from $X$ to $\R^{\N}$, $x\mapsto (x_k)$ where $x_k:=\langle x, e_k\rangle$. The induced measure in $X$ is still called $\nu$. In \cite{TAMS} we proved that Hypothesis \ref{h1'} is satisfied if we choose $R=Q^{1/2}$, where $Q$ is the covariance of $\nu$, 
$$Qe_h= b_{1}\,\mu_h^{\frac1m}\,e_h,\quad h\in\N. $$
We also explicitly exhibited the functions $v_z$ for every $z\in X$, however their expression is not needed here. Notice that if $m=1$ then  $\nu$ is  the Gaussian measure   $N_{0, Q}$.   

 \begin{Proposition} 
 \label{normaLp}
 For every $p\geq 2$, $\nu (L^p(0,1)) = 1$ and the function $x\mapsto  \|x\|_{L^p(0,1)}$ belongs to 
 $L^q(X, \nu)$ for every $q> 1$. 
 Moreover, 
 $$\lim_{n\to \infty} \int_X  (\|P_n x - x\|_{L^p(0,1)})^q\,\nu(dx) =0, $$
 for every $q>1$. 
  \end{Proposition}
\begin{proof}
For every $x\in X$ and $n\in \N$, $P_nx \in L^p(0,1)$ for every $p$, since it is a linear combination of  elements of $L^{\infty}(0,1)$. 

As a first step, we take $p=2l$ with $l\in \N$, and we show that 
 the  sequence $(x, \xi)\mapsto P_nx(\xi)$ is bounded (and, in fact, convergent)   in $L^{2l}(X\times (0,1); \nu\times \lambda_1)$,  where 
$\lambda_1$ is the Lebesgue measure.  Notice that 
$$\int_X \|P_nx\|_{L^{2l}(0, 1)}^{2l}\nu(dx) = \int_X\int_0^1|P_nx(\xi)|^{2l}d\xi\, \nu(dx), \quad n\in \N. $$
Recalling that 
$$(a_1+\ldots +a_n)^{2l} =   \sum_{k_1, \ldots, k_n \in \{0, \ldots, 2l\}, \; \sum_{j=1}^n k_j = 2l} \frac{(2l)!}{(k_1)!\cdot \ldots \cdot (k_n)!} a_1^{k_1}\cdot\ldots \cdot a_n^{k_n}$$
 we get
\begin{equation}
\label{stima2l} 
\begin{array}{l}
\ds \int_X \|P_nx\|_{L^{2l}(0,1)}^{2l}\nu(dx) = 
\\
\\
= \ds  \sum_{k_{j}  \in \{0, \ldots, 2l\}, \; \sum_{j =1}^n k_{j} = 2l}  \frac{(2l)!}{(k_1)!\cdot \ldots \cdot (k_n)!} 
\int_X \int_0^1 x_1^{k_1}\cdot\ldots \cdot x_n^{k_n} e_1(\xi)^{k_1}\cdot\ldots \cdot e_n(\xi)^{k_n}  d\xi\,\nu(dx), 
\end{array}
\end{equation}
where
$$\int_X \int_0^1 x_1^{k_1}\cdot\ldots \cdot x_n^{k_n} e_1(\xi)^{k_1}\cdot\ldots \cdot e_n(\xi)^{k_n}  d\xi\,\nu(dx) = 
\int_X x_1^{k_1}\cdot\ldots \cdot x_n^{k_n} \,\nu(dx)\int_0^1  e_1(\xi)^{k_1}\cdot\ldots \cdot e_n(\xi)^{k_n}  d\xi$$
vanishes if some of the $k_j$ are odd. 
What remains is (setting $k_j = 2h_j  $ for $k_j$ even)
$$\sum_{h_{j}  \in \{0, \ldots, l\}, \; \sum_{j =1}^n h_{j} = l}  \frac{(2l)!}{(2h_1)!\cdot \ldots \cdot (2h_n)!} b_{m, h_1}\cdot \ldots \cdot b_{m, h_n}\mu_1^{h_1/m}\cdot \ldots \cdot
\mu_n^{h_n/m}\int_0^1  e_1(\xi)^{2h_1}\cdot\ldots \cdot e_n(\xi)^{2h_n}  d\xi. $$
Using the estimates
$$\int_0^1  e_1(\xi)^{2h_1}\cdot\ldots \cdot e_n(\xi)^{2h_n}  d\xi \leq C^{2l}, \quad b_{h_j} \leq M_l:= \max \{b_{m, r}: \; r\in \{0, 1, \ldots, l\}\}$$
and recalling that $b_{0}=1$, so that $M_l\leq M_l^{h_j}$ for every $j$, we get 
$$
\begin{array}{l}
\ds  \int_X \|P_nx\|_{L^{2l}(0,1)}^{2l}\nu(dx)  \leq \sum_{h_{j}  \in \{0, \ldots, l\}, \; \sum_{j =1}^n h_{j} = l}  \frac{(2l)!}{(2h_1)!\cdot \ldots \cdot (2h_n)!}
(M_l\mu_1^{1/m})^{h_1}\cdot \ldots \cdot (M_l\mu_n^{1/m})^{h_n} C^{2l}
\\
\\
\ds \leq   \frac{C^{2l}(2l)!}{l!}
\sum_{h_{j}  \in \{0, \ldots, l\}, \; \sum_{j =1}^n h_{j} = l}  \frac{l!}{(h_1)!\cdot \ldots \cdot (h_n)!}(M_l\mu_1^{1/m})^{h_1}\cdot \ldots \cdot (M_l\mu_n^{1/m})^{h_n} 
\\
\\
\ds =    \frac{(M_lC)^{2l}(2l)!}{l!} \left( \sum_{j=1}^n  \mu_j^{1/m}\right)^{2l}
\end{array}$$
 and similarly
\begin{equation}
\label{stima2l2} 
\int_X \int_0^1 |P_{n+h}x(\xi) - P_n x(\xi)|^{2l} d\xi\,\nu(dx) = \int_X \|P_{n+h}x - P_n x\|_{L^{2l}(0,1)}^{2l}\nu(dx) \leq  \frac{(M_lC)^{2l}(2l)!}{l!} \left( \sum_{j=n}^{n+h} \mu_j^{1/m}\right)^{2l}
\end{equation}
so that the sequence $(x, \xi)\mapsto  P_n x(\xi)$ is a Cauchy sequence in $L^{2l}(X\times (0,1); \nu\times \lambda_1)$, and it converges to some $\Phi \in L^{2l}(X\times (0,1); \nu\times \lambda_1)$ for every $l\in\N$.  Taking  $l=1$ we get  $\Phi(x, \xi) = x(\xi)$, since 
$$\int_X \int_0^1 |P_{n }x(\xi) -  x(\xi)|^{2} d\xi\,\nu(dx) = \int_X \|P_nx-x\|^2 \nu(dx), $$
that vanishes as $n\to\infty$ since $ \|P_nx-x\|^2\to 0$ as $n\to\infty$ for every $x$, and $ \|P_nx-x\|^2\leq \|x\|^2\in L^1(X, \nu)$. Therefore, 
$$\lim_{n\to \infty} \int_X \int_0^1 |P_n x(\xi) - x(\xi)|^{2l} d\xi\,\nu(dx) =0, $$
and
 for every $q\geq p\geq 2$, for every integer $l$ such that $2l\geq q$ we have 
$$\begin{array}{l}
\ds \int_X  \|P_nx -x  \|_{L^p(0,1)} ^q\nu(dx)  = \int_X  \left( \int_0^1|P_nx -x |^p\, d\xi \right)^{q/p}\nu(dx) 
\\
\\
\ds \leq   \int_X  \left( \int_0^1|P_nx-x|^{2l}\, d\xi\right)^{q/2l}\nu(dx) 
\leq  \left(\int_X  \int_0^1|P_nx(\xi)-x(\xi)|^{2l}\, d\xi \,\nu(dx) \right)^{q/2l}. 
  \end{array}$$
The statement follows.  \end{proof}
 
As a consequence of  Proposition \ref{normaLp} we can exhibit a family of nontrivial Sobolev functions.

\begin{Proposition} 
\label{palleLp}
 Fix a  function $\Phi: \R\mapsto \R$ belonging to $C^{1+\alpha}_{loc}(\R)$ for some $\alpha\in (0, 1)$, such that 
 \begin{equation}
\label{e4.5}
\lim_{|s|\to \infty} \Phi(s) = +\infty , 
 \end{equation}
 and such that there exist $r \geq 1$, $C_1 >0$ satisfying 
 \begin{equation}
\label{e4.4}
 |\Phi'(s))| \leq C_1(1+|s|^{r}), \quad s \in \R . 
 \end{equation}

 Then the function $F:X\mapsto \R$, 
$$
F(x)= \left\{ \begin{array}{lll}  & \ds \int_0^{1} \Phi(x(\xi))d\xi, & x\in L^{p_1+1}(0,1), 
\\
\\
 & +\infty , & x \notin L^{p_1+1}(0,1), 
\end{array}\right. 
$$
belongs to $W^{1,q}_0(X, \nu)$ for every $q>1$, and $\nabla_q F(x) = \Phi '\circ x$ for a.e. $x\in X$ (namely, for each  $x\in L^{2r}(0,1)$). 
\end{Proposition}
\begin{proof} 
By \eqref{e4.4} there exists $C_2>0$ such that $|\Phi (s)| \leq C_2(1+|s|^{1+r})$, for every $s\in \R$. Therefore, for every $x\in X$, 
$$\left| \int_0^1\Phi(x(\xi)) d\xi \right| \leq C_2\left(1+\int_0^1 |x(\xi)|^{1+r}d\xi \right) = C_2(1 + \|x\|_{L^{1+r}(0,1)}^{1+r}). $$
By proposition \ref{normaLp}, $F\in L^q(X, \nu)$ for every $q>1$. Let us prove that $F\in W^{1,q}_0(X, \nu)$, approaching it by 
$$F_n(x) := \int_0^{1} \Phi_n(x(\xi))d\xi, \quad n\in \N, \; x\in X, $$
where $\Phi_n$ is a regularized truncation of $\Phi$, defined as usual introducing  a function $\theta\in C^{\infty}_c(\R)$ such that $\theta (t) = t$ for $|t|\leq 1$, 
$\theta (t)=$ constant for $t\geq 2$ and for $t\leq -2$, $\|\theta'\|_{\infty} \leq 1$, 
$$\Phi_n(s) := n \theta\left( \frac{\Phi(s)}{n}\right), \quad n\in \N, \; s\in \R. $$
Every $F_n$ belongs to $C^1(X)$, being of the type $x\mapsto \int_0^1 \Psi(x(\xi))d\xi$, with $\Psi\in C^{1+\alpha}_b(\R)$, and
$$\nabla F_n(x)(\xi)  = (\Phi_n' \circ x) (\xi)= \theta' \left(\frac{\Phi (x(\xi))}{n}\right) \Phi '(x(\xi)), \quad n\in N, \; x\in X, \; \xi\in (0, 1). $$
Let us estimate $\|\nabla F_n(x)\|$. For every $x\in X$ we have 
 $ \nabla F_n(x)(\xi) =0$ if $|\Phi (x(\xi))| \geq 2n$. On the other hand,  \eqref{e4.5} implies that $\Phi ^{-1}(-2n, 2n)$ is bounded, say $\Phi ^{-1}(-2n, 2n)\subset (-c_n, c_n)$, so that 
  $| \nabla F_n(x)(\xi)| \leq |\Phi '(x(\xi))| \leq C_1(1+ c_n^r)$ if $|\Phi (x(\xi))| <2n$. Therefore, for every $x\in X$ we have $\|\nabla F_n(x)\|\leq C_1(1+ c_n^r)$, so that $F_n\in C^1_b(X)$. 
 
Since $\lim_{n\to \infty} \Phi_n(x(\xi)) = \Phi(x(\xi))$ and  $|\Phi_n(x(\xi))| \leq |\Phi(x(\xi))| \leq C_2 (1+ |x(\xi)|^{1+r})$ for a.e. $\xi\in (0, 1)$, by the Dominated Convergence Theorem  we have 
$$\lim_{n\to \infty} F_n(x) = F(x)$$
for every $x\in L^{ r+1}(0,1)$ (so, for $\nu$-a.e $x\in X$ by proposition \ref{normaLp}). Moreover, 
$$| F_n(x) - F(x)|^q\leq (| F_n(x)| + |F(x)| )^q\leq (2C_2 (1+ \|x \|_{L^{1+r}(0,1)}^{1+r}))^q, $$
and since $x\mapsto  \|x \|_{L^{1+r}(0,1)}^{1+r}\in L^q(X, \nu)$ for every $q>1$ by proposition \ref{normaLp}, again by  the Dominated Convergence Theorem we get 
$$\lim_{n\to \infty} \|F_n - F\|_{L^q(X, \nu)} =0. $$
Let us prove that $  \nabla F_n$ converges to the vector field $x\mapsto \Phi'\circ x$ in $L^q(X, \nu;X)$, namely that 
\begin{equation}
\label{derivata}
\lim_{n\to \infty} \int_X \left(\int_0^1 \left| \left( \theta' \left(\frac{\Phi (x(\xi))}{n}\right)-1\right)\Phi '(x(\xi))\right|^2d\xi\right)^{q/2}\nu(dx) =0. 
\end{equation}
For every $x\in L^{2r}(0, 1)$ we have $\lim_{n\to\infty}  \theta'  ( \Phi (x(\xi))/n)  -1=0$ for a.e. $\xi\in (0, 1)$, moreover $|( \theta'  ( \Phi (x(\xi))/n)  -1) \Phi '(x(\xi))|\leq 
2C_1(1+ |x(\xi)|^r) \in L^2 (0,1)$. By the Dominated Convergence Theorem, $\lim_{n\to \infty} 
\int_0^1| ( \theta'  ( \Phi (x(\xi))/n)  -1) \Phi '(x(\xi))|^2 d\xi =0$. Moreover, 
$$\int_0^1\left| \left( \theta' \left(\frac{\Phi (x(\xi))}{n}\right)-1\right)\Phi '(x(\xi))\right|^2d\xi  \leq \int_0^1 (2 C_1(1+ |x(\xi)|^r))^2 d\xi \leq 8C_1^2 (1+ \|x\|_{L^{2r}(0,1)}^{2r}), $$
so that, again by Proposition \ref{normaLp} and by the Dominated Convergence Theorem, \eqref{derivata} holds. Therefore, $F\in W^{1,q}_0(X, \nu)$ and $\nabla_q F(x) =  \Phi'\circ x$, for $\nu$-a.e. $x\in X$. \end{proof}

 By Proposition \ref{Pr:almost}(b)  It follows that for almost all $r\in \R$, the characteristic function of the set $F^{-1}(-\infty, r)$ belongs to $BV_0(X, \nu)$. 
 
 In the paper \cite{TAMS} we proved that the function $F(x) : = \|x\|^2$ belongs to $W^{2,q}(X,\nu)$ for every $q>1$, and it satisfies \eqref{g}. Therefore, by Proposition \ref{surface} the $L^2$-balls $B(0,r)$ have finite perimeter for every $r>0$. 
Taking $\Phi(s): = |s|^p$ with $p>2$, Proposition \ref{palleLp} yields that 
the characteristic function of the $L^p$-balls $B(0,r)\subset L^p(0,1)$ (that are much less regular subsets of $X$ from the topological point of view) belong to $BV_0(X, \nu)$ for almost every $r\in \R$.


 \section{Acknowledgements}
 We thank Michael R\"ockner and Michele Miranda for very useful discussions. 
Our work was partially supported by the research project 2015233N54
 PRIN 2015    ``Deterministic and stochastic evolution equations".

 \appendix

 \section{Proof of  \eqref{commk}. }

\small

\begin{Proposition}
\label{Prop:Appendix}
Let Hypothesis \ref{hyp_conv} hold, and assume in addition that $U\in C^1(X)$ and that $\nabla U$ is Lipschitz continuous. Then  \eqref{commk} holds. 
\end{Proposition}
\begin{proof}
The formal derivation of \eqref{commk} is easy: setting $v(t,x):=(T(t)f)(x)$, $v $ satisfies $\partial v/\partial t = Kv$. Taking the derivative along $e_k$ and using \eqref{Kcilindrico} we should get
$$\frac{\partial }{\partial t} \frac{\partial }{\partial e_k} v = \frac{\partial }{\partial e_k}Kv = K\left(  \frac{\partial v}{\partial e_k} \right) - \frac{1}{2\lambda_k}   \frac{\partial v}{\partial e_k}  - \langle D^2Ue_k, \nabla v\rangle$$
which yields \eqref{commk}. However, we do not know whether $\partial T(t)f/\partial e_k$ is time differentiable and belongs to $D(K)$, even for $f\in \mathcal {FC}^{\infty}_b(X)$. To justify  \eqref{commk} we follow a tedious approximation procedure,  already used in \cite[Sect. 3.2.2]{DPL} to prove regularity results for the elements of $D(K)$,  that allows to use regularity results for PDEs in finite dimension.

Let $L$ be such that 
$$\|\nabla U(x)-\nabla U(y)\| \leq L\|x-y\|, \quad x, \;y\in X. $$
We approach $U$ by  $U_n(x) := U(P_nx)$, where $P_n$ is defined by \eqref{Pn}. 

The function  
$$u_n: \R^n\mapsto \R, \quad u_n(\xi): = U_n\left(\sum_{i=1}^n \xi_ie_i\right)$$ 
is such that  $U_n(x) = u_n(  \langle x, e_1\rangle, \ldots  \langle x, e_n\rangle)$,   it is convex and it has Lipschitz continuous gradient as well as $U$. This is not enough for our aims, and we approach it again by 
$$u^{\eps}_{n}(\xi) = \int_{\R^n} u_n(\xi -\eps y)\theta_n(y)dy,  \quad \xi\in \R^n, $$
where $\theta_n: \R^n\mapsto \R$ is any smooth nonnegative compactly supported function with $\int_{\R^n} \theta_n(y)dy =1$. Then  $u^{\eps}_{n}$ is smooth and convex,   $\nabla u^{\eps}_{n}$ is Lipschitz continuous with  Lipschitz constant $\leq L$, so that   $\|D^2 u^{\eps}_{n}(\xi)\|_{{\mathcal L}(\R^n)} \leq L$ for every $n\in \N$,  $\eps>0$, $\xi\in \R^n$. 
We introduce the differential operator $K_{n, \eps}$ defined by 
$$K_{n, \eps}\varphi(\xi) = {\mathcal L}_n \varphi(	xi) -  \langle \nabla u^{\eps}_{n}(\xi), \nabla \varphi (\xi)\rangle, \quad \xi\in \R^n, $$
where $ {\mathcal L}_n$  is the Ornstein-Uhlenbeck operator 
$${\mathcal L}_n \varphi (\xi) =  \frac{1}{2} \sum_{i=1}^n \left( \frac{\partial^2  \varphi}{\partial \xi_i^2}(\xi) -  \lambda_i^{-1}\xi_i  \frac{\partial  \varphi}{\partial \xi_i} (\xi)\right), \quad \xi \in \R^n.$$
and we consider the
Cauchy problem  in $\R^n$, 
\begin{equation}
\label{Tneps(t)}
\left\{ \begin{array}{l}
\ds \frac{\partial}{\partial t}v_{n,  \eps}(t, \xi) =K_{n, \eps}v_{n,  \eps}(t, \xi)\rangle , \quad t>0, 
\\
\\
v_n(0, \xi) = v_0(\xi)
\end{array}\right. 
\end{equation}
where $v_0\in B_b(\R^n)$ (the space of the Borel bounded functions).   Since the  drift coefficients are globally Lipschitz, and the dissipativity condition  
$$- \sum_{i,j=1}^n\left( \delta_{ij} \lambda_i^{-1}+  \frac{\partial  ^2u^{\eps}_{n}}{\partial \xi_i\partial\xi_j}(\xi) \right)\eta_i \eta_j \leq 0, \quad \xi, \;\eta\in \R^n$$
holds, the regularity results of \cite[Ch. 1]{Cerrai} and of \cite{LV1} yield existence of a Markov, strong Feller semigroup $e^{tK_{n, \eps}}$ that maps $B_b(\R^n)$ into
$ C^{2+\alpha}_b(\R^n)$ for every  $t>0$, $\alpha \in (0,1)$. Moreover, for $v_0\in C_b(\R^n)$, $ v_n(t, \xi):= (e^{tK_{n, \eps}}  v_0) (\xi)$ is the unique bounded classical solution 
to \eqref{Tneps(t)}, and since all the coefficients of $K_{n, \eps}$  are smooth,   $(t,\xi) \mapsto (e^{tK_{n, \eps}} v_0)(\xi)\in C^{\infty}((0, +\infty)\times \R^n)$  by the classical local regularity results in parabolic equations. So, for  $t>0$ both sides of the equation in \eqref{Tneps(t)} are differentiable with respect to the space variables, and  
$$  \frac{\partial}{\partial \xi_k}\frac{\partial  v_n}{\partial t} = \frac{\partial}{\partial \xi_k}\bigg( {\mathcal L}_nv(t, \xi) - \langle \nabla u^{\eps}_{n}(\xi), \nabla_{\xi} v_n(t, \xi)\rangle\bigg)$$
namely
$$\frac{\partial }{\partial t}  \frac{\partial}{\partial \xi_k} v_n = \left( {\mathcal L}_n  - \frac{1}{2\lambda_k}\right) \frac{\partial}{\partial \xi_k} v_n  - \langle \nabla u^{\eps}_{n}(\xi), \nabla_{\xi}\frac{\partial}{\partial \xi_k}v_n(t, \xi)\rangle - 
\sum_{i=1}^n \left(  \frac{\partial^2  }{\partial \xi_i\partial \xi_k}u^{\eps}_{n}(\xi) \right) \frac{\partial  }{\partial \xi_i} v_n . $$
Therefore, if  $v_0\in C^1_b(\R^n)$, 
\begin{equation}
\label{commneps}
 \frac{\partial}{\partial \xi_k}(e^{tK_{n, \eps}} v_0) - e^{-t/2\lambda_k}e^{tK_{n, \eps}}  \frac{\partial v_0}{\partial \xi_k}= - \int_0^t e^{-(t-s)/2\lambda_k}e^{(t-s)tK_{n, \eps}}  (   \langle D^2u^{\eps}_{n} \cdot \nabla e^{sK_{n, \eps}} v_0, e_k \rangle )\,ds. 
\end{equation}

Now we go back from  $\R^n$ to $X$, setting $f_n(\xi) = f(\sum_{i=1}^n\xi_i e_i) $  for $f\in B_b(X)$,  and 
$$(T_{n, \eps}(t) f )(x) := (e^{tK_{n, \eps}} f_n)(\langle x, e_1\rangle, \ldots, \langle x, e_n\rangle). $$
Then, $T_{n, \eps}(t) $ is a Markov strong Feller semigroup, and it maps $B_b(X)$ into $\mathcal{FC}^2_b(X)\cap C^{\infty}(X)$. 

We are going to establish estimates on $T_{n, \eps}(t) f $ that yield estimates on $\nabla T(t)f$ and convergence results. 
First of all,  for  $v_0\in C^1_b(\R^n)$ we have (\cite[Sect. 3.2.2]{DPL})  
\begin{equation}
\label{gradgrad}
\|\nabla e^{tK_{n, \eps}}  v_0\|_{\infty} \leq \|\nabla v_0\|_{\infty}, \quad t>0, \; v_0\in C^1_b(\R^n). 
\end{equation}
Taking into account the dissipativity condition, the procedure of \cite[Prop. 2.6]{MS} with $\Lambda = \R^n$, $V\equiv 0$ gives (with the choice $a= 1$)
$$\|\nabla e^{tK_{n, \eps}}  v_0\|_{\infty} \leq    \frac {1}{\sqrt{t}} \|v_0\|_{\infty}, \quad t>0, \; v_0\in C_b(\R^n), $$
so that, by the semigroup law and \eqref{gradgrad}, 
\begin{equation}
\label{funzgrad}
\|\nabla e^{tK_{n, \eps}}  v_0\|_{\infty} \leq  \max\left\{ 1,  \frac {1}{\sqrt{t}} \right\} \|v_0\|_{\infty}, \quad t>0, \; v_0\in B_b(\R^n), \quad t>0. 
 \end{equation}
By \eqref{gradgrad}, for $f\in C^1_b(X)$ we have
\begin{equation}
\label{gradienteTNeps}
\|\nabla T_{n, \eps}(t) f\|_{\infty} \leq \|\nabla f\|_{\infty}, \quad t>0, 
\end{equation}
and by \eqref{funzgrad} for $f\in B_b(X)$ we have
\begin{equation}
\label{funzgradTNeps}
\|\nabla T_{n, \eps}(t) f\|_{\infty} \leq \max\left\{ 1,  \frac {1}{\sqrt{t}} \right\}\|f\|_{\infty}, \quad t>0. 
\end{equation}
Moreover, for $k>n$ we have $\frac{\partial }{\partial e_k}T_{n, \eps}(t) f =0$, while for  $k\leq n$ formula  \eqref{commneps} yields
\begin{equation}
\label{commapprox}
\frac{\partial }{\partial e_k}T_{n, \eps}(t) f - e^{-t/2\lambda_k}T_{n, \eps}(t) \frac{\partial f}{\partial e_k}  = - \int_0^t e^{-(t-s)/2\lambda_k}T_{n, \eps}(t-s) (   \langle D^2U^{\eps}_{n} \cdot\nabla T_{n, \eps}(s)f, e_k \rangle )\,ds 
\end{equation}
where 
$$U^{\eps}_{n} (x) := u^{\eps}_{n}(\langle x, e_1\rangle,   \ldots, \langle x, e_n\rangle). $$
Since   $\|D^2 u^{\eps}_{n}(\xi)\|_{{\mathcal L}(\R^n)} \leq L$ for every $\xi\in \R^n$, 
\begin{equation}
\label{stimaD2U}
\|D^2 U^{\eps}_{n}(x)\|_{{\mathcal L}(X)} \leq L, \quad \eps >0, \; n\in \N, \; x\in X. 
\end{equation}
We rewrite \eqref{commapprox} as 
\begin{equation}
\label{commapproxvett}
P_n( \nabla T_{n, \eps}(t) f )(x) = e^{-tA}({\bf T}_{n, \eps}(t)P_n\nabla f (x) )  - \int_0^t e^{-(t-s)A} ({\bf T}_{n, \eps}(t-s) ( P_n D^2U^{\eps}_{n} \cdot\nabla T_{n, \eps}(s)f ) )(x)\,ds . \end{equation}
At every $x\in X$, both  sides of \eqref{commapproxvett} belong to $D(Q^{-1/2})= D(-2A)^{1/2}$. Using the estimate $\|(-2A)^{1/2} e^{tA}\|_{{\mathcal L}(X)} \leq C(1+ t^{-1/2})$ and \eqref{gradienteTNeps}, \eqref{stimaD2U}, for $f\in C^1_b(X)$  we get 
$$\|Q^{-1/2} e^{-tA}({\bf T}_{n, \eps}(t) P_n \nabla f (x))\|_X \leq C(1+ t^{-1/2})\| {\bf T}_{n, \eps}(t)P_n \nabla f (x)\|_X \leq C(1+ t^{-1/2})  \|\nabla f\|_{\infty}, $$
$$\begin{array}{l}
\ds \| Q^{-1/2}  \int_0^t e^{-(t-s)A} {\bf T}_{n, \eps}(t-s) (  P_n  D^2U^{\eps}_{n} \cdot\nabla T_{n, \eps}(s)f )(x) \,ds \|
\\
\\
\ds \leq  \int_0^t C(1+ (t-s)^{-1/2})  \|  \langle D^2U^{\eps}_{n} \cdot\nabla T_{n, \eps}(s)f\|_{\infty} ds
 \leq  \int_0^t C(1+ (t-s)^{-1/2}) L \|\nabla  f\|_{\infty} ds .
 \end{array}$$
Summing up, and recalling that $(I-P_n) ( \nabla T_{n, \eps}(t) f ) \equiv 0$, we obtain 
\begin{equation}
\label{QgradienteTNeps}
\|Q^{-1/2}\nabla T_{n, \eps}(t) f(x)\|  \leq C_1(t+t^{-1/2}) \|\nabla f\|_{\infty}, \quad t>0, \;x\in X, \; f\in C^1_b(X), 
\end{equation}
with $C_1$ independent of $n$ and $\eps$.

Now we let  $n\to\infty$, $\eps\to0$.  Since,  for every $f\in C_b(X)$ and $t>0$,  $x\mapsto (T_{n, \eps}(t)f)(x)$ is in $\mathcal{FC}^2_b(X)$ and it depends only on $\langle x, e_1\rangle,   \ldots, \langle x, e_n\rangle$, it belongs to $D(K)$ and by \eqref{Kcilindrico} we have
$$\frac{d}{dt} (T(t)f - T_{n, \eps}(t)f) = K(T(t)f - T_{n, \eps}(t)f) - \sum_{i=1}^n \left( \frac{\partial U}{\partial e_i}-\frac{\partial U^{\eps}_{n} }{\partial e_i}\right)\frac{\partial}{\partial e_i} T_{n, \eps}(t)f, \quad t>0$$
so that 
$$ T(t)f - T_{n, \eps}(t)f = -\int_0^t T(t-s)\left( \sum_{i=1}^n \left(\frac{\partial U}{\partial e_i}-\frac{\partial U^{\eps}_{n}}{\partial e_i} \right)\frac{\partial}{\partial e_i} T_{n, \eps}(s)f\right)ds, $$
and 
$$\frac{\partial }{\partial e_k}(  T(t)f - T_{n, \eps}(t)f) = -  \frac{\partial }{\partial e_k}\int_0^tT(t-s)\left(\sum_{i=1}^n \left(\frac{\partial U}{\partial e_i}-\frac{\partial U^{\eps}_{n} }{\partial e_i} \right)\frac{\partial }{\partial e_i}T_{n, \eps}(s)f\right)ds, $$
which we rewrite as 
$$P_n \nabla ( T(t)f - T_{n, \eps}(t)f) =  - P_n \nabla  \int_0^t T(t-s)( \langle P_n (\nabla U - \nabla U^{\eps}_{n} ), \nabla T_{n, \eps}(s)f  ) ds.$$

We recall that  $T(t)$ is a contraction semigroups in $L^2(X, \nu)$. Using  \eqref{funzgradTNeps} we get 
\begin{equation}
\label{stimaconvsgr}
\begin{array}{l}
\| T(t)f - T_{n, \eps}(t)f \|_{L^2(X, \nu)} \leq \ds \int_0^t \| \nabla U -\nabla U^{\eps}_{n} \|_{L^2(X, \nu;X)}\| \nabla T_{n, \eps}(s)f\|_{L^2(X, \nu;X)} ds
\\
\\
\leq C_2(t+t^{1/2}) \|  f\|_{\infty} \| \nabla U -\nabla U^{\eps}_{n} \|_{L^2(X, \nu;X)},  
\end{array}
\end{equation}
with $C_2$ independent of $n$ and $\eps$. Using  \eqref{stimabase} and \eqref{gradienteTNeps} we get 
$$\begin{array}{l}
\ds \|P_n \nabla ( T(t)f - T_{n, \eps}(t)f) \|_{L^2(X, \nu;X)} \leq \ds \int_0^t \frac{C}{\sqrt{t-s}}\| \nabla U -\nabla U^{\eps}_{n} \|_{L^2(X, \nu;X)}\| \nabla T_{n, \eps}(s)f\|_{L^2(X, \nu;X)} ds
\\
\\
\leq C \sqrt{t}  \|\nabla f\|_{\infty} \| \nabla U -\nabla U^{\eps}_{n} \|_{L^2(X, \nu;X)}. 
\end{array}$$
So, we estimate 
$$\begin{array}{l} \| \nabla U -\nabla U^{\eps}_{n} \|_{L^2(X, \nu;X)} \leq  \| \nabla U -\nabla U_{n} \|_{L^2(X, \nu;X)} +  \| \nabla U_n -\nabla U^{\eps}_{n} \|_{L^2(X, \nu;X)}
\\
\\
\leq \ds  \left(\int_X \|\nabla U(x) - \nabla U(P_nx)\|^2d\nu\right)^{1/2} +   \| \nabla U_n -\nabla U^{\eps}_{n} \|_{L^{\infty}(X, \nu;X)}
\\
\\
\leq \ds  \left(\int_X \|\nabla U(x) - \nabla U(P_nx)\|^2d\nu\right)^{1/2} + \eps L . 
\end{array}$$
Since  $\nabla U$ is continuous and it has at most linear growth at infinity, the first term vanishes as  $n\to\infty$ by the Dominated Convergence Theorem. Therefore, 
$\lim_{(n, \eps) \to (0, \infty)}  \| \nabla U -\nabla U^{\eps}_{n} \|_{L^2(X, \nu;X)} =0$, and \eqref{stimaconvsgr} yields
\begin{equation}
\label{convsemigruppi}
\lim_{(n, \eps )\to (\infty, 0)} \| T(t) - T_{n, \eps}(t) \|_{{\mathcal L}(C_b(X), L^2(X, \nu))} = 0, \quad t>0. 
\end{equation}
%
%
%
Moreover,  since 
$$\|  \nabla ( T(t)f - T_{n, \eps}(t)f) \|_{L^2(X, \nu;X)} ^2 = \|P_n \nabla ( T(t)f - T_{n, \eps}(t)f) \|_{L^2(X, \nu;X)} ^2 +   \|(I-P_n) \nabla   T(t)f\|_{L^2(X, \nu;X)} ^2, $$
where $\lim_{n\to\infty} \|(I-P_n) \nabla   T(t)f\|_{L^2(X, \nu;X)}=0$, 
we also have
\begin{equation}
\label{convgradienti}
\lim_{(n, \eps )\to (\infty, 0)} \| \nabla  T(t)f - \nabla T_{n, \eps}(t)f \|_{L^2(X, \nu; X)} =0, \quad t>0, \; f\in C_b(X). 
\end{equation}

Let us use \eqref{convgradienti} to get bounds for $\nabla T(t)f$, when $f\in C^1_b(X)$. 
By \eqref{convgradienti}, there exists a sequence $(\nabla T_{n_k, \eps_k}(t)f(x))$ with $n_k\to \infty$, $\eps_k\to 0$, that converges to $\nabla T(t)f(x)$ for $\nu$-a.e.$x\in X$. For such $x$'s \eqref{gradienteTNeps} yields 
$\| \nabla T(t)f(x)\|\leq \|\nabla f\|_{\infty}$. %
Still for such $x$'s, 
the sequence 
 $(\nabla T_{n_k, \eps_k}(t) f(x))$ is bounded in $D(Q^{-1/2})$ by  \eqref{QgradienteTNeps}, and since $D(Q^{-1/2})$ is a Hilbert space, up to a further subsequence $(\nabla T_{n_k, \eps_k}(t) f(x))$ weakly converges to some  $h\in D(Q^{-1/2})$; since $X-\lim_{k\to\infty} \nabla T_{n_k, \eps_k}(t) f(x) =  \nabla T (t)f(x)$, we have $h= 
 \nabla T (t)f(x) \in D(Q^{-1/2})$ and 
$$\|  \nabla T (t)f(x)\|_{D(Q^{-1/2})}  \leq \liminf_{k\to \infty} \|\nabla T_{n_k, \eps_k}(t) f(x)\|_{D(Q^{-1/2})} \leq \|\nabla f\|_{\infty}+
   C_2(t+t^{-1/2}) \|\nabla f\|_{\infty}. $$
Therefore,  for every $t>0$ and $f\in C^1_b(X)$, 
\begin{equation}
\label{stimagradQ}
\left\{\begin{array}{ll}
(i) &  \|\nabla T(t)f\|_{\infty} \leq \|\nabla f\|_{\infty} , 
\\
\\
 (ii)& \|Q^{-1/2}\nabla T(t)f\|_{\infty} \leq C_3(t + t^{-1/2})\|\nabla f\|_{\infty} ,
 \end{array}\right. 
 \end{equation}
for some $C_3>0$. 

Let us go back to  \eqref{commapprox}. 
Using \eqref{convsemigruppi} and \eqref{convgradienti}, the left hand side of  \eqref{commapprox} converges to the left hand side of \eqref{commk}
in $L^2(X, \nu)$ as $n\to\infty$, $\eps \to 0$. The difference between the respective right hand sides is split as 
$$\begin{array}{l}
\ds  \int_0^t e^{-(t-s)/2\lambda_k}(T_{n, \eps}(t-s)- T (t-s))(  \langle D^2U^{\eps}_{n} \cdot\nabla T_{n, \eps}(s)f, e_k \rangle )\,ds 
\\
\\
\ds +  
\int_0^t e^{-(t-s)/2\lambda_k}T (t-s) (   \langle D^2U^{\eps}_{n}   \cdot( \nabla T_{n, \eps}(s)f -\nabla T(s)f), e_k \rangle )\,ds 
\\
\\
\ds + 
\int_0^t e^{-(t-s)/2\lambda_k}T (t-s) (   \langle (D^2U^{\eps}_{n}   -D^2U)  \cdot \nabla T(s)f, e_k \rangle )\,ds 
\\
\\
:= I^{(1)}_{n, \eps}(t) +  I^{(2)}_{n, \eps} (t) + I^{(3)}_{n, \eps} (t). $$
\end{array}$$
By \eqref{convsemigruppi}, \eqref{stimaD2U}, and \eqref{gradienteTNeps},  for every $s\in (0, t)$ we have 
$$\lim_{(n, \eps) \to (\infty, 0)} \|(T_{n, \eps}(t-s)- T (t-s))(  \langle D^2U^{\eps}_{n} \cdot \nabla T_{n, \eps}(s)f, e_k \rangle )\|_{L^2(X, \nu)} =0. $$
Moreover, for every $n$ and $\eps$ we have, by  \eqref{stimaD2U} and \eqref{gradienteTNeps},  and recalling that $T_{n, \eps}(t)$, $T(t)$ are contraction semigroups in $C_b(X)$, 
$$\|(T_{n, \eps}(t-s)- T (t-s))(  \langle D^2U^{\eps}_{n} \cdot\nabla T_{n, \eps}(s)f, e_k \rangle )\|_{L^{\infty}(X, \nu)} \leq L \|\nabla f\|_{\infty}. $$
By the Dominated Convergence Theorem, 
$$ \lim_{(n, \eps) \to (\infty, 0)}  I^{(1)}_{n, \eps}(t) =0. $$
Similarly, by \eqref{stimaD2U} and \eqref{convgradienti}, for every $s\in (0, t)$ we have 
$$\lim_{(n, \eps) \to (\infty, 0)} \| T (t-s) (   \langle D^2U^{\eps}_{n}   \cdot( \nabla T_{n, \eps}(s)f -\nabla T(s)f), e_k \rangle )\|_{L^2(X, \nu)} =0, $$
while for every $n$ and $\eps$ we have, recalling that $T(t)$ is a contraction semigroup in $L^2(X, \nu)$ and using \eqref{stimaD2U},  \eqref{gradienteTNeps} and \eqref{stimagradQ}(i), 
$$\| T (t-s) (   \langle D^2U^{\eps}_{n}   \cdot( \nabla T_{n, \eps}(s)f -\nabla T(s)f), e_k \rangle )\|_{L^2(X, \nu)} \leq 2L  \|\nabla f\|_{\infty}, $$
and again by the Dominated Convergence Theorem, 
$$ \lim_{(n, \eps) \to (\infty, 0)}  I^{(2)}_{n, \eps} (t)=0. $$

Concerning $  I^{(3)}_{n, \eps}(t)$, we prove that it converges weakly to $0$ in $L^2(X, \nu)$. By \eqref{stimaD2U}, $(U^{\eps}_{n} )$ is bounded in $W^{2,2}(X, \gamma)$,  and therefore in $W^{2,2}(X, \nu)$, by a constant independent of $n$ and $\eps$. A sequence $(U^{\eps_k}_{n_k} )$ converges weakly in $W^{2,2}(X, \nu)$, and since $U^{\eps}_{n} \to U$ as 
$(n, \eps)\to  (\infty, 0)$, the weak limit is $U$ and we have $U^{\eps}_{n} \rightharpoonup U$ in $W^{2,2}(X, \nu)$, as 
$(n, \eps)\to  (\infty, 0)$. For every $\psi\in L^2(X, \nu)$, the functional 
$$\varphi\mapsto F(\varphi) := \int_X \int_0^t e^{-(t-s)/2\lambda_k}T (t-s) \left(   \sum_{j=1}^{\infty}  \frac{\partial ^2\varphi}{\partial e_k \partial e_j}  \frac{\partial T(t)f}{\partial e_j }  \right) \,ds \,\psi\,d\nu $$
belongs to the dual space of $W^{2,2}(X, \nu)$, since 
$$\begin{array}{lll}
|F(\varphi)| & \leq & \ds \int_0^t \|  \sum_{j=1}^{\infty}  \frac{\partial ^2\varphi}{\partial e_k \partial e_j}  \frac{\partial T(s)f}{\partial e_j  }\|_{L^2(X, \nu)} ds \|\psi\|_{L^2(X, \nu)}
\\
\\
& \leq & \ds \int_0^t  \left(\int_X   \sum_{j=1}^{\infty}  \left(\frac{\partial ^2\varphi}{\partial e_k \partial e_j}\right) ^2 \lambda_j\,d\nu\right)^{1/2} \|Q^{-1/2}\nabla T(s)\|_{L^2(X, \nu;X)}
 ds \, \|\psi\|_{L^2(X, \nu)}
\\
\\
& \leq & \ds   \int_0^t  C_3(s+ s^{-1/2})ds\, \|\nabla f\|_{\infty} \|\varphi\|_{W^{2,2}(X, \nu)}  \|\psi\|_{L^2(X, \nu)}
\end{array}$$
where we used estimate  \eqref{stimagradQ}(ii). Therefore, $ \lim_{(n, \eps) \to (\infty, 0)} F(U^{\eps}_{n}) - F(U) =0$, namely for every $\psi\in L^2(X, \nu)$ we have 
$ \lim_{(n, \eps) \to (\infty, 0)} \int_X  I^{(3)}_{n, \eps}(t)\, \psi\, d\nu =0$. So, $ I^{(3)}_{n, \eps}(t)$ weakly converges to $0$ as 
$(n, \eps) \to (\infty, 0)$. Summing up, the right hand side of  \eqref{commapprox}  weakly converges to the left hand side of \eqref{commk}
in $L^2(X, \nu)$ as $(n, \eps) \to (\infty, 0)$, and   \eqref{commk} follows. 
\end{proof}

\end{document}